\documentclass[11pt,oneside, english,reqno,twoside a4paper]{amsart}
\usepackage{lmodern}

\usepackage[foot]{amsaddr}
\usepackage{amsmath}
\usepackage{amssymb,nicefrac,bm,upgreek,mathtools,verbatim}
\usepackage[final]{hyperref}
\usepackage[mathscr]{eucal}
\usepackage{dsfont}
\usepackage[normalem]{ulem}
\usepackage{amsopn,esint}
\usepackage[T1]{fontenc}
\usepackage[shortlabels]{enumitem}
\usepackage[T1]{fontenc}
\usepackage[utf8]{inputenc}
\usepackage[utf8]{inputenc}
\usepackage{apptools}
\usepackage{amstext}
\usepackage{amsthm}
\usepackage{enumitem}
\usepackage{tikz}
\usetikzlibrary{patterns}
\usepackage{amssymb}
\usepackage{diagbox}
\usepackage{pict2e}
\AtAppendix{\counterwithin{lemma}{section}} 

\usepackage[margin=1in]{geometry}

\newcommand{\stkout}[1]{\ifmmode\text{\sout{\ensuremath{#1}}}\else\sout{#1}\fi}

\newtheorem{lemma}{Lemma}[section]
\newtheorem{theorem}{Theorem}[section]

\theoremstyle{definition}
\newtheorem{definition}{Definition}[section]

\newtheorem{remark}{Remark}[section]
\numberwithin{theorem}{section}
\numberwithin{equation}{section}

\hypersetup{
  colorlinks=true,
  citecolor=black,
  linkcolor=black,
  urlcolor = blue,
  anchorcolor = blue,
  frenchlinks=false,
  pdfborder={0 0 0},
  naturalnames=false,
  hypertexnames=false,
  breaklinks}
\usepackage[capitalize,nameinlink]{cleveref}

\usepackage[abbrev,msc-links,nobysame,alphabetic]{amsrefs}

\crefname{section}{Section}{Sections}
\crefname{subsection}{Section}{Sections}
\crefname{condition}{Condition}{Conditions}
\crefname{hypothesis}{Hypothesis}{Conditions}
\crefname{assumption}{Assumption}{Assumptions}
\crefname{lemma}{Lemma}{Lemmas}
\crefname{fact}{Fact}{Facts}

\Crefname{figure}{Figure}{Figures}

\crefformat{equation}{\textup{#2(#1)#3}}
\crefrangeformat{equation}{\textup{#3(#1)#4--#5(#2)#6}}
\crefmultiformat{equation}{\textup{#2(#1)#3}}{ and \textup{#2(#1)#3}}
{, \textup{#2(#1)#3}}{, and \textup{#2(#1)#3}}
\crefrangemultiformat{equation}{\textup{#3(#1)#4--#5(#2)#6}}%
{ and \textup{#3(#1)#4--#5(#2)#6}}{, \textup{#3(#1)#4--#5(#2)#6}}%
{, and \textup{#3(#1)#4--#5(#2)#6}}

\Crefformat{equation}{#2Equation~\textup{(#1)}#3}
\Crefrangeformat{equation}{Equations~\textup{#3(#1)#4--#5(#2)#6}}
\Crefmultiformat{equation}{Equations~\textup{#2(#1)#3}}{ and \textup{#2(#1)#3}}
{, \textup{#2(#1)#3}}{, and \textup{#2(#1)#3}}
\Crefrangemultiformat{equation}{Equations~\textup{#3(#1)#4--#5(#2)#6}}%
{ and \textup{#3(#1)#4--#5(#2)#6}}{, \textup{#3(#1)#4--#5(#2)#6}}%
{, and \textup{#3(#1)#4--#5(#2)#6}}

\crefdefaultlabelformat{#2\textup{#1}#3}
\newcommand{\vertiii}[1]{{\left\vert\kern-0.25ex\left\vert\kern-0.25ex\left\vert #1
    \right\vert\kern-0.25ex\right\vert\kern-0.25ex\right\vert}}

\newcommand{\mfz}{\mathfrak{z}}
\newcommand{\mfw}{\mathfrak{w}}

\newcommand{\mcf}{\mathcal{F}}

\newcommand{\aaa}{{\mathcal{A}}}  

\newcommand{\la}{\lambda}
\newcommand{\La}{\Lambda}

\newcommand{\RR}{\mathbb{R}}
\newcommand{\NN}{\mathbb{N}}

\newcommand{\Rn}{{\mathbb{R}^{n}}}

\DeclareMathOperator{\dv}{div}

\newcommand{\apprle}{\lesssim}

\newcommand{\loc}{\mathrm{loc}}

\def\YYint#1#2#3{{\setbox0=\hbox{$#1{#2#3}{\iint}$}
		\vcenter{\hbox{$#2#3$}}\kern-.50\wd0}}

\def\Xint#1{\mathchoice
	{\XXint\displaystyle\textstyle{#1}}%
	{\XXint\textstyle\scriptstyle{#1}}%
	{\XXint\scriptstyle\scriptscriptstyle{#1}}%
	{\XXint\scriptscriptstyle\scriptscriptstyle{#1}}%
	\!\int}
\def\XXint#1#2#3{{\setbox0=\hbox{$#1{#2#3}{\int}$}
		\vcenter{\hbox{$#2#3$}}\kern-.50\wd0}}

\def\dashint{\Xint-}

\usepackage{color}
\definecolor{dmagenta}{rgb}{.4,.1,.4}
\definecolor{dblue}{rgb}{.0,.0,.5}
\definecolor{mblue}{rgb}{.0,.4,.7}
\definecolor{ddblue}{rgb}{.0,.0,.4}
\definecolor{dred}{rgb}{.9,.0,.0}
\definecolor{dgreen}{rgb}{.0,.5,.0}
\definecolor{Eeom}{rgb}{.0,.0,.5}
\definecolor{dbrown}{rgb}{.6,.0,.0}

\newcommand{ \mr }{ \mathbb{R} }

\newcommand{\Norm}[1]{\left|\hspace{-0.3mm}\left| #1 \right|\hspace{-0.3mm}\right|}
\newcommand{\iints}[1]{{\int\hspace{-0.28cm}\int_{#1}}}
\newcommand{\iintss}{{\int\hspace{-0.28cm}\int}}
\newcommand{ \miints }{{\iintss\hspace{-0.56cm} -\hspace{-0.15cm}-}}
\newcommand{\miint}[1]{{\miints_{\hspace{-0.13cm}#1}}}

\allowdisplaybreaks

\newcommand{\ttl}{}
\begin{document}
\title[Higher integrability for parabolic double phase problems]
{\ttl}
\begin{center}
{\bf \Large Gradient higher integrability for degenerate parabolic double phase systems with two modulating coefficients}


\medskip

{\small  Jehan Oh\footnote{jehan.oh@knu.ac.kr} and Abhrojyoti Sen\footnote{sen@math.uni-frankfurt.de}
}
\medskip

{\small
$^1$ Department of Mathematics, Kyungpook National University, Daegu, 41566,\\
Republic of Korea}

\medskip
{\small
$^2$ Goethe-Universit\"{a}t Frankfurt, Institut f\"{u}r Mathematik, Robert-Mayer-Str. 10,\\
D-60629 Frankfurt, Germany} 
 



\begin{abstract} We establish an interior gradient higher integrability result for weak solutions to degenerate parabolic double phase systems involving two modulating coefficients. To be more precise, we study systems of the form
\begin{align*}
    u_t-\dv \left(a(z)|Du|^{p-2}Du+ b(z)|Du|^{q-2}Du\right)=-\dv \left(a(z)|F|^{p-2}F+ b(z)|F|^{q-2}F\right),
\end{align*}
where $2\leq p\leq q < \infty$ and the modulating coefficients $a(z)$ and $b(z)$ are non-negative, with $a(z)$ being uniformly continuous and $b(z)$ being H\"{o}lder continuous. We further assume that the sum of two modulating coefficients is bounded from below by some positive constant. To establish the gradient higher integrability result, we introduce a suitable intrinsic geometry and develop a delicate comparison scheme to separate and analyze the different phases--namely, the $p$-phase, $q$-phase and $(p,q)$-phase. To the best of our knowledge, this is the first regularity result in the parabolic setting that addresses general double phase systems within the framework of weak solutions.
\end{abstract}
\end{center}
\thanks{Jehan Oh is supported by the National Research Foundation of Korea (NRF) grant funded by the Korea government [Grant Nos. RS-2023-00217116, RS-2025-00555316, and RS-2025-25415411]. Abhrojyoti Sen is supported by postdoctoral research grants from the Alexander von Humboldt Foundation, Germany.}

\maketitle
\tableofcontents

\section{Introduction and main result}

In this paper, we study the higher integrability of the spatial gradient of a weak solution to degenerate parabolic double phase systems of the form
\begin{equation}    \label{eq: main equation}
    u_t -\operatorname*{div}\mathcal{A}(z,u,Du)=-\operatorname*{div}\mathcal{B}(z,F) \quad \text{in $\Omega_T$,}
\end{equation}
where $\Omega_T:=\Omega\times (0,T)$ with a bounded open set $\Omega \subset \mr^n$ for $n\geq 2$.
Here, the Carath\'{e}odory vector fields $\aaa=\aaa(z, v, \xi):\Omega_T\times \mathbb{R}^N \times \mathbb{R}^{Nn} \to \mathbb{R}^{Nn}$ and $\mathcal{B}=\mathcal{B}(z, \xi):\Omega_T \times \mathbb{R}^{Nn}\to \mathbb{R}^{Nn}$ with $N\geq 1$ satisfy the following structural assumptions: there exist two constants $0<\nu\leq L$ such that for any $z \in \Omega_T, \; v\in \mr^N$ and $\xi\in \mr^{Nn}$,
\begin{equation}    \label{eq: growth condition of A}
    \nu H(z,|\xi|)\leq \mathcal{A}(z,v,\xi)\cdot \xi, \qquad |\mathcal{A}(z,v,\xi)||\xi|\leq LH(z,|\xi|)
\end{equation}
and
\begin{equation}    \label{eq: growth condition of B}
    |\mathcal{B}(z,\xi)||\xi|\leq LH(z,|\xi|),
\end{equation}
where the function $H:\Omega_T\times  \mr^+\rightarrow \mr^+$ is defined by 
$$
H(z,\kappa)=a(z)\kappa^p+b(z)\kappa^q
$$
for $z\in\Omega_T$ and $\kappa\in\mr^+$.  
In addition, the source term $F:\Omega_T\rightarrow \mr^{Nn}$ satisfies 
\begin{equation}    \label{eq: assumption of F}
    \iints{\Omega_T} H(z,|F(z)|)\,dz < +\infty.
\end{equation}

When $a(\cdot) \equiv 1$ and $b(\cdot) \equiv 0$, the system reduces to the classical parabolic $p$-Laplace one. In this case, a comprehensive regularity theory, including higher integrability and H\"{o}lder continuity, is well-known, see \cite{DiBF84, DiBF85, DiBF_2_85, KL2002, 1993_Degenerate_parabolic_equations_DiBenedetto, phdthesis} and the references given there.

When $a(\cdot) \equiv 1$ and $b(\cdot) \not\equiv 0$, it reduces to the parabolic double phase system with a single modulating coefficient. This classical model, originally introduced by Zhikov for the elliptic case \cite{Zhikov1986, Zhikov1993, Zhikov1995, Zhikov1997}, has been widely used in the study of double phase problems, which arise in the modeling of strongly anisotropic materials. Notably, this model also serves as a prototypical example exhibiting the Lavrentiev phenomenon, as discussed in \cite{Esposito2004, Fonseca2004, Zhikov1993, Zhikov1995}. Such problems have been extensively studied and continue to attract attention. In the elliptic setting, a broad range of regularity results, including $C^{\alpha}/ C^{1, \alpha}$-estimates, Calder\'{o}n-Zygmund estimates and Harnack inequality, have been established in the pioneering works \cite{Baroni2015, Colombo2015, Colombo2016, Colombo2015a, DeFilippis2019, Baasandorj2020, Byun2020, Byun2017}. On the other hand, in the parabolic setting, results such as gradient higher integrability \cite{2023_Gradient_Higher_Integrability_for_Degenerate_Parabolic_Double-Phase_Systems, KS24}, Calder\'{o}n–Zygmund estimates \cite{kim2024calderonzygmundtypeestimateparabolic, kim2025calderonzygmundtypeestimatesingular}, existence and uniqueness of solutions \cite{ST16, Chlebicks2019}, and H\"{o}lder continuity \cite{KIM2025113231} have been obtained only recently.  However, the classical model involves only one modulating coefficient, which is limited in describing materials composed of both $p$-hardening and $q$-hardening components. Although the case involving two modulating coefficients is mathematically more demanding, it provides a more suitable framework for modeling strongly anisotropic composites. For the double phase problems with two modulating coefficients, the elliptic case has been studied in recent works such as \cite{KO24, Baroni2018, KO2025}, but to the best of our knowledge, there are no results for the parabolic case.
This paper addresses the parabolic double phase problem with two modulating coefficients, aiming to prove the regularity of solutions by introducing an appropriate intrinsic geometry and developing a delicate phase separation.

To avoid the simultaneous vanishing of the two modulating coefficients, we assume that their sum is bounded from below by a positive constant, that is, 
\begin{equation}\label{eq:coefficient_sum_condition}
    a(z)+b(z)\geq \delta, \quad \forall z\in \Omega_T
\end{equation}
for some $\delta>0$.
As in the classical model, we assume that the modulating coefficient $b:\Omega_T \to \mathbb{R}^+$ is in $C^{\alpha, \frac{\alpha}{2}}(\Omega_T)$ for some $\alpha \in (0,1]$ and we impose the restriction 
\begin{equation}\label{def_pq}
2\leq p \leq q \leq  p + \frac{2\alpha}{n+2}<\infty.
\end{equation}
In particular, the assumption $b \in C^{\alpha, \frac{\alpha}{2}}(\Omega_T)$ implies that $b \in L^{\infty}(\Omega_T)$ and there exists a constant $[b]_{\alpha}>0$ such that
\begin{equation*}
|b(x_1,t_1)-b(x_2,t_2)|\leq [b]_{\alpha}\left(|x_1-x_2|^{\alpha}+|t_1-t_2|^{\frac{\alpha}{2}}\right) 
\end{equation*}
for any $(x_1,t_1), (x_2,t_2) \in \Omega_T$.

We also assume that the modulating coefficient $a: \Omega_T \to \mathbb{R}^{+}$ is bounded and uniformly continuous, that is, there exists a modulus of continuity $\omega_a$ such that
\begin{equation}   \label{eq:assumption of a}
    |a(z_1)-a(z_2)|\leq \omega_a(d(z_1,z_2))
\end{equation}
for any $z_1, z_2 \in \Omega_T$, where $d$ is the parabolic metric defined in \cref{parabolic metric}.

\vspace{0.1cm}

To state the main theorem, we first introduce the definition of a weak solution to \cref{eq: main equation}.
\begin{definition}\label{weak solution}
	A function $u : \Omega\times (0, T] \to \mathbb{R}^N$ with 
	\begin{align*}
		u \in C_{\loc}\left(0,T; L^2_{\loc}\left(\Omega, \mathbb{R}^N\right)\right)\cap L^q_{\loc}\left(0, T; W^{1, q}_{\loc}\left(\Omega, \mathbb{R}^N\right)\right)
	\end{align*}
	is said to be a weak solution to \eqref{eq: main equation} if for every compact subset $K\subset \Omega$ and for every subinterval $[t_1, t_2]\subset (0, T]$, there holds
	\begin{align*}
		\int_{t_1}^{t_2}\int_K -u\cdot \varphi_t \,dx\,dt+\int_{K}u \cdot \varphi \, dx\Big|_{t_1}^{t_2}+\int_{t_1}^{t_2} \int_K \aaa(z, u, D u)\cdot D \varphi \,dx \,dt=\int_{t_1}^{t_2}\int_{K} \mathcal{B}(z, F)\cdot D\varphi\, dx\, dt
	\end{align*}
	for any test function $\varphi \in C^{\infty}_0 (K \times(0, T], \mathbb{R}^N).$ 
\end{definition}

We now state the main theorem of the paper.
\begin{theorem}\label{main_thm}
	Let $u$ be a weak solution of \eqref{eq: main equation} with \eqref{eq: growth condition of A}--\eqref{eq:assumption of a} in force. Then there exist positive constants $\varepsilon_0 = \varepsilon_0(\textnormal{\texttt{data}})\in (0,1)$ and $c=c (\textnormal{\texttt{data}})\geq 1$ such that the following estimate
	\begin{align*}
		\miint{Q_{r}(z_0)} H_1(z, |D u|)^{1+\varepsilon}\, dz \leq c\left(\miint{Q_{2r}(z_0)}H_1(z, |D u|)\,dz\right)^{1+\frac{\varepsilon q}{2}}+ c\left(\miint{Q_{2r}(z_0)}H_1(z, |F|)^{1+\varepsilon}\,dz\right)^{\frac{q}{2}}
	\end{align*}
	holds for every $Q_{2r}(z_0)\subset \Omega_T$ and any $\varepsilon \in (0, \varepsilon_0)$, where $Q_{(\cdot)}(z_0)$, $H_1(z, \cdot)$ and $\textnormal{\texttt{data}}$ are defined in Subsection \ref{Subsec:Notation} \ref{not5.5}, \ref{not6} and \ref{not9}, respectively.
\end{theorem}
We make a couple of remarks on the result above.
\begin{remark}
It should be noted that the above theorem is stated under the technical assumption $|Du|\in L^q(\Omega_T).$ However, a more natural assumption would be to consider
\begin{align*}
u \in C_{\loc}\left(0,T; L^2_{\loc}\left(\Omega, \mathbb{R}^N\right)\right)\cap L^1_{\loc}\left(0, T; W^{1, 1}_{\loc}\left(\Omega, \mathbb{R}^N\right)\right)
\end{align*}
along with the integrability condition
\begin{align*}
\iint_{\Omega_T}H(z, |Du|)\,dz=\iint_{\Omega_T} \left[ a(z)|Du|^p+b(z)|Du|^q \right] dz< \infty.
\end{align*}
A parabolic Lipschitz truncation method (see \cite{KIM2025110738} for $a(\cdot)\equiv 1$) can be used to obtain the energy estimate under this assumption. This will be pursued in future work.
\end{remark}

\begin{remark}
This paper focuses on the degenerate regime $2\leq p\leq q < \infty$. Extending the analysis to the singular regime $\frac{2n}{n+2}<p<2$ would be an interesting direction. Recently, in \cite{S2025}, a unified scaling approach was introduced to establish gradient higher integrability for degenerate and singular parabolic multi-phase problems. Applying such an intrinsic scaling technique in our setting would be an important avenue for further exploration.
\end{remark}
We close this discussion by highlighting some of the novelties and strategies to prove \cref{main_thm}.

\smallskip
(i)\,\, A commonly used strategy for establishing higher integrability of gradients in parabolic equations typically involves two main steps: first, proving reverse H\"{o}lder-type inequalities on intrinsic cylinders, and second, demonstrating the existence of such cylinders within certain superlevel sets. The final proof follows from a Vitali covering argument combined with Fubini's theorem.

Although our approach to proving \cref{main_thm} aligns broadly with \cite{2023_Gradient_Higher_Integrability_for_Degenerate_Parabolic_Double-Phase_Systems}, we introduce a refined comparison scheme to appropriately distinguish different phases. The presence of the coefficient $a(\cdot)$ and the parabolic nature of the system \eqref{eq: main equation} introduce additional intricacy, setting our approach apart from \cite{2023_Gradient_Higher_Integrability_for_Degenerate_Parabolic_Double-Phase_Systems} and \cite{KO24}. Specifically, depending on the smallness of $a(\cdot)$ and $b(\cdot)$, we classify the behavior into distinct phases: the $p$-phase, the $q$-phase, and the mixed $(p, q)$-phase, which is one of the main novelties of this paper.
We establish reverse H\"{o}lder inequalities for each phase and implement a Vitali covering argument that addresses nine distinct cases.

\smallskip
(ii)\,\, Now, we describe our phase separation procedure and compare it with that of parabolic double phase systems with a single modulating coefficient, i.e., $a(\cdot)\equiv 1$. First, consider the elliptic $p$-phase, characterized by $b(z_0)\leq [b]_{\alpha}\tau^{\alpha}.$ In this case, two possibilities arise: either $K\lambda^p\geq b(z_0)\lambda^q$ or $K\lambda^p\leq b(z_0)\lambda^q.$ It is shown in \cite{2023_Gradient_Higher_Integrability_for_Degenerate_Parabolic_Double-Phase_Systems} that the conditions $b(z_0)\leq [b]_{\alpha}\tau^{\alpha}$ and $K\lambda^p\leq b(z_0)\lambda^q$ cannot simultaneously hold. Moreover, the following assumption
\begin{align*}
\miint{Q^{\lambda}_{\tau}(z_0)} \left[ H(z, |Du|)+H(z, |F|)\right] dz<\lambda^p
\end{align*}
is imposed to characterize the $p$-phase.
Next, we consider the elliptic $(p, q)$-phase, where $b(z_0)\geq [b]_{\alpha}\tau^{\alpha}.$ Again, two cases emerge: either $K\lambda^p\geq b(z_0)\lambda^q$ or $K\lambda^p\leq b(z_0)\lambda^q.$ Additionally, the assumption in \cite{2023_Gradient_Higher_Integrability_for_Degenerate_Parabolic_Double-Phase_Systems} for $(p, q)$-phase includes
\begin{align*}
    \miint{G^{\lambda}_{\tau}(z_0)} \left[ H(z, |Du|)+H(z, |F|)\right] dz<\lambda^p+b(z_0)\lambda^q.
\end{align*}
Note that, if $K\lambda^p\geq b(z_0)\lambda^q$ then the above condition essentially leads to the analysis of $p$-intrinsic case, reducing the scope of nontrivial distinct phases to the case $K\lambda^p\leq b(z_0)\lambda^q.$

In contrast, our setting is more intricate due to the presence of the coefficient $a(\cdot),$ leading to eight distinct possibilities depending on whether $a(z_0)\geq \omega_a(\tau)$ or $a(z_0)\leq \omega_a(\tau),$ as detailed in \cref{Sec:Higher_Integrability}. These distinctions result in three separate cases, unlike the framework in \cite{2023_Gradient_Higher_Integrability_for_Degenerate_Parabolic_Double-Phase_Systems}, as described in \eqref{eq: p-phase condition}, \eqref{eq: q-phase condition}, and \eqref{Eq 3.8}-\eqref{Eq 3.9}. Depending on the size of the coefficients $a(\cdot)$ and $b(\cdot)$ in the respective intrinsic cylinders, the system \eqref{eq: main equation} can behave as a $p$-system, or a $q$-system or a $(p, q)$-system. Moreover, note that we only need a uniform continuity of $a(\cdot).$

The remaining part of the paper is organized as follows.
In \cref{Sec:Lemmas}, we set our notations,  record some known results and prove the energy estimate.
\cref{Sec:Reverse_H} is dedicated to proving the reverse H\"{o}lder inequalities for the $p$-, $q$-, and $(p,q)$-phases.
In \cref{Sec:Higher_Integrability}, we establish our gradient higher integrability result stated in \cref{main_thm}. The proof begins with a stopping time argument, followed by the use of the Vitali covering lemma. Finally, we employ Fubini's theorem together with reverse H\"{o}lder inequalities to conclude the result.

\section{Preliminaries} \label{Sec:Lemmas}

\subsection{Notations}\label{Subsec:Notation}
Here we collect the standard notation that will be used throughout the paper:
\begin{enumerate}[label=(\roman*),series=theoremconditions]
\item \label{not1} We shall denote by $n$ the space dimension and by $z=(x,t)$ a point in $ \Rn\times (0,T)$. We will always assume that $n \geq 2$.   
\item\label{not2} A ball with center $x_0 \in \mathbb{R}^n$ and radius $r$ is denoted as
 \begin{align*}
     B_r(x_0)=\left\{x \in \Rn : |x-x_0|<r\right\}.
 \end{align*}
\item\label{general cylinder} In general, a parabolic cylinder centered at $z_0=(x_0, t_0)$ is denoted as 
\begin{align*}
    U_{r, \varrho}(z_0):=B_r(x_0)\times \left(t_0-\varrho, t_0+\varrho\right)=:B_r(x_0)\times l_{\varrho}(t_0).
\end{align*}
\item\label{not5} Let $\lambda\geq 1.$ We shall use the following three types of cylinders for the $p$-, $q$-, and $(p, q)$-phases. 

\noindent The notation $Q_{\varrho, \lambda}(z_0)$ is used to denote $p$-intrinsic cylinders at $z_0=(x_0, t_0),$
\begin{align*}
Q_{\varrho, \lambda}(z_0):=B_{\varrho}(x_0)\times \left(t_0-\frac{\lambda^{2-p}\varrho^2}{a(z_0)}, t_0+\frac{\lambda^{2-p}\varrho^2}{a(z_0)}\right)
=:B_{\varrho}(x_0)\times I_{\varrho, \lambda}(t_0).
\end{align*}
The notation $J_{\varrho, \lambda}(z_0)$ is used to denote $q$-intrinsic cylinders at $z_0=(x_0, t_0),$
\begin{align*}
J_{\varrho, \lambda}(z_0):=B_{\varrho}(x_0)\times \left(t_0-\frac{\lambda^{2-q}\varrho^2}{b(z_0)}, t_0+\frac{\lambda^{2-q}\varrho^2}{b(z_0)}\right)
=:B_{\varrho}(x_0)\times I^q_{\varrho, \lambda}(t_0).
\end{align*}
The notation $G_{\varrho, \lambda}(z_0)$ is used to denote $(p,q)$-intrinsic cylinders at $z_0=(x_0, t_0),$
\begin{align*}
G_{\varrho, \lambda}(z_0):=B_{\varrho}(x_0)\times \left(t_0-\frac{\lambda^{2}\varrho^2}{\Lambda}, t_0+\frac{\lambda^{2}\varrho^2}{\Lambda}\right)
=:B_{\varrho}(x_0)\times I^{(p,q)}_{\varrho, \lambda}(t_0),
\end{align*}
where $\Lambda:=a(z_0)\lambda^p+b(z_0)\lambda^q.$
\item \label{not5.5} In the case $\la =1$ and $a(z_0)=1,$ we denote 
\[
Q_\varrho(z_0) := B_{\varrho}(x_0) \times (t_0 - \varrho^2,t_0+\varrho^2) =: B_\varrho(x_0) \times I_\varrho(t_0).
\]
\item\label{not6} We shall use the notation $H_1(z,\kappa) := a(z) \kappa^p + b(z) \kappa^q + 1$ for $z \in \Omega_T$ and $\kappa\in\mr^+$.

\item\label{not7} Integration with respect to either space or time only will be denoted by a single integral $\int$ whereas integration on $\Omega\times (0, T)$ will be denoted by a double integral $\iint$. 
	\item\label{not9} The notation $a \lesssim b$ is shorthand for $a\leq c b$ where $c$ is a universal constant which depends on $``\text{\texttt{data}}"$ where 
 \begin{align*}
\text{\texttt{data}}:=\left(n, N, p, q, \alpha, \nu, L, \delta, \omega_a, [b]_{\alpha}, ||a||_{L^{\infty}(\Omega_T)}, ||b||_{L^{\infty}(\Omega_T)}, ||H(z, |D u|)||_{L^1(\Omega_T)}, ||H(z, |F|)||_{L^1(\Omega_T)}\right).
 \end{align*}
 \item \label{M1}
 We denote
 \begin{align*}
     M:=\iint_{\Omega_T} \left[ H(z, |Du|)+H(z, |F|) \right] dz < \infty.
 \end{align*}
 \item \label{integral avarage} We shall denote
 \begin{align*}
     (f)_{\mathcal{D}}:= \frac{1}{|\mathcal{D}|}\iint_{\mathcal{D}} f(z)\, dz= \miint{\mathcal{D}} f(z)\, dz
 \end{align*}
 as the integral average of $f$ over a measurable set $\mathcal{D} \subset \Omega_T$ with $0<|\mathcal{D}|<\infty.$ In particular, when $\mathcal{D}=Q_{\varrho, \lambda}(z_0), J_{\varrho, \lambda}(z_0), G_{\varrho, \lambda}(z_0),$ the integral average of $f$ is denoted by $(f)_{z_0, \varrho, \lambda}, (f)^q_{z_0, \varrho, \lambda}, (f)^{(p,q)}_{z_0, \varrho, \lambda},$ respectively.
\end{enumerate}

\vspace{0.5cm}

We also introduce the definition of the parabolic metric and its scaled version.

\begin{definition}[Parabolic metric]\label{parabolic metric}
Given any two points $z_1=(x_1, t_1)$ and $z_2=(x_2, t_2)$ in $\RR^{n+1},$ we define the parabolic metric $d$ as 
\begin{align*}
d(z_1, z_2):=\max\left\{|x_1-x_2|, \sqrt{|t_1-t_2|}\right\}.
\end{align*}
\end{definition}
\vspace{.3cm}
A suitable adaptation of the above definition in intrinsically scaled cylinders centered at $z_0\in \RR^{n+1}$ is as follows.
\begin{definition}[Scaled parabolic metric]\label{para metric}
Fix $z_0 \in \RR^{n+1}.$ Given any two points $z_1=(x_1, t_1), z_2=(x_2, t_2) \in \RR^{n+1}$ and $\lambda\geq 1$, we define the parabolic metric $d^{\la}_{z_0}$ as

\begin{align*}
\displaystyle
d^{\la}_{z_0}(z_1,z_2):=\left\{\begin{array}{l}
\max\left\{|x_1-x_2|, \sqrt{a(z_0)\lambda^{p-2}|t_1-t_2|}\right\},\quad \text{in $p$-intrinsic case,}\\
\max\left\{|x_1-x_2|, \sqrt{b(z_0)\lambda^{q-2}|t_1-t_2|}\right\},\quad \text{in $q$-intrinsic case,}\\
\max\left\{|x_1-x_2|, \sqrt{\Lambda \lambda^{-2}|t_1-t_2|}\right\},\quad\quad\,\,\,\,\, \text{in $(p,q)$-intrinsic case,}
\end{array}\right.
\end{align*}
where $\La:=a(z_0)\la^p+b(z_0)\la^q.$   
\end{definition}

\subsection{Basic results and useful lemmas}
In this subsection, we introduce several lemmas that will be used later.
We begin by recalling the notations \ref{general cylinder}, \ref{not5.5} and \ref{integral avarage}. Parabolic cylinders
$U_{r,\tau}(z_0)$ and $Q_r(z_0)$ are defined by
    $$ U_{r,\tau}(z_0) := B_r(x_0)\times (t_0-\tau,t_0+\tau), \qquad Q_r(z_0) := B_r(x_0)\times (t_0-r^2,t_0+r^2)$$
for $z_0=(x_0,t_0)$, and the average integral of $u$ over $U_{r, \tau}(z_0)$ is defined as
\[(u)_{U_{r, \tau}(z_0)}:=\miint{U_{r, \tau}(z_0)}u(z)\,dz.\] Furthermore, we recall that
\begin{align*}
    H(z,|Du|)=a(z)|Du|^p+b(z)|Du|^q.
\end{align*}

The following lemma presents a Gagliardo–Nirenberg inequality, which will play a key role in establishing reverse H\"{o}lder's inequalities.
\begin{lemma}[\cite{Hasto_2021}, Lemma 2.12]\label{lem : lemma 2.1}
For an open ball $B_{\rho}\subset \mathbb{R}^n$, let $\psi \in W^{1, s}(B_{\rho})$ and take $\sigma, s, r \in[1, \infty)$ and $\vartheta \in(0,1)$ such that
    $$
    -\frac{n}{\sigma} \leq \vartheta\left(1-\frac{n}{s}\right)-(1-\vartheta) \frac{n}{r}.
    $$
    Then we have
    $$
    \dashint_{B_{\rho}} \frac{|\psi|^\sigma}{\rho^\sigma} \, d x \leq c\left(\dashint_{B_{\rho}}\left(\frac{|\psi|^s}{\rho^s}+|D\psi|^s\right) d x\right)^{\frac{\vartheta \sigma}{s}}\left(\dashint_{B_{\rho}} \frac{|\psi|^r}{\rho^r} d x\right)^{\frac{(1-\vartheta) \sigma}{r}}
    $$
    for some constant $c=c(n, \sigma)$.
\end{lemma}

We next record a simple but fundamental iteration lemma from \cite[Lemma 1.1]{GG82}:
\begin{lemma}
	\label{iter_lemma}
	Let $0< r< R<\infty$ be given and $h : [r, R] \to \RR$ be a non-negative bounded function. Furthermore, let $\vartheta \in (0,1)$ and $A,B,\gamma, \geq 0$ be fixed constants and 
	suppose that
	$$
	h(\varrho_1) \leq \vartheta h(\varrho_2) + \frac{A}{(\varrho_2-\varrho_1)^{\gamma}} + B,
	$$
	holds for all $r \leq \varrho_1 < \varrho_2 \leq R$.
    Then the following conclusion holds:
	$$
	h(r) \apprle_{(\vartheta,\gamma)} \frac{A}{(R-r)^{\gamma}} + B.
	$$
\end{lemma} 

Now we recall the notation \ref{not6} from Subsection \ref{Subsec:Notation}. We define $H_1:\Omega_T\times  \mr^+\rightarrow \mr^+$ by
$$
H_1(z,\kappa):=H(z,\kappa)+1=a(z)\kappa^p+b(z)\kappa^q+1
$$
for $z\in\Omega_T$ and $\kappa\in\mr^+$.
In this paper, we use $H_1$ instead of $H$ for technical convenience as in the following lemma to establish the key inequalities and estimates necessary to prove the main result.

\begin{lemma} \label{lem : bounds of H_1}
     There exist positive constants $c_1 = c_1(\delta)$ and $c_2 = c_2(\|a\|_{L^\infty(\Omega_T)},\,\|b\|_{L^\infty(\Omega_T)})$ such that 
 \begin{equation*}\label{eq:bounds of H_1}
     c_1|\xi|^p \leq H_1(z,|\xi|) \leq c_2 (|\xi|^q+1)
 \end{equation*}
 for all $z \in \Omega_T$ and $\xi \in \mathbb{R}^n$.
\end{lemma}

\begin{proof}
    Let $z \in \Omega_T$ and $\xi \in \mathbb{R}^n$.
    When $|\xi| \geq 1$, we observe from \eqref{eq:coefficient_sum_condition} that
    \begin{align*}
        \delta |\xi|^p & \leq(a(z)+b(z))|\xi|^p \\
        &\leq a(z)|\xi|^p+b(z)|\xi|^q+1 = H_1(z,|\xi|) \\
        &\leq (a(z)+b(z)+1)|\xi|^q \\
        &\leq (\|a\|_{L^\infty(\Omega_T)}+\|b\|_{L^\infty(\Omega_T)}+1)|\xi|^q.
    \end{align*}
    When $|\xi| \leq 1$, we get
    \begin{align*}
        |\xi|^p\leq 1 &\leq a(z)|\xi|^p+b(z)|\xi|^q+1 = H_1(z,|\xi|) \\
         &\leq a(z)+b(z)+1\\
         &\leq \|a\|_{L^\infty(\Omega_T)}+\|b\|_{L^\infty(\Omega_T)}+1.
     \end{align*}
     Combining these inequalities, we conclude that for any $\xi \in \mathbb{R}^n$,
     \begin{equation*}
         \min\{\delta,1\}|\xi|^p\leq H_1(z,|\xi|)\leq (\|a\|_{L^\infty(\Omega_T)}+\|b\|_{L^\infty(\Omega_T)}+1)(|\xi|^q+1),
     \end{equation*}
     and the lemma follows.
\end{proof}

We now establish the following Caccioppoli inequality. 
\begin{lemma} \label{lem : the Caccioppoli inequality}
    Let $u$ be a weak solution to \eqref{eq: main equation}. Then, for $U_{R,S}(z_0)\subset \Omega_T$, $r\in [R/2,R)$ and $\tau\in[S/2^2,S)$, there exists a constant $c$ depending on $n,p,q,\nu$ and $L$ such that the following inequality holds: 
    \begin{align*}
        &\sup_{t\in(t_0-\tau, t_0+\tau)}\dashint_{B_r(x_0)} \frac{|u-(u)_{U_{r,\tau}(z_0)}|^2}{\tau}\; dx +\miint{U_{r,\tau}(z_0)}H(z,|Du|)\;dz\\
        &\qquad \leq c\miint{U_{R,S}(z_0)}   H\left(z,\frac{|u-(u)_{U_{R,S}(z_0)}|}{R-r}\right) dz\\
        &\quad\qquad + c\miint{U_{R,S}(z_0)}\frac{|u-(u)_{U_{R,S}(z_0)}|^2}{S-\tau}\;dz +c\miint{U_{R,S}(z_0)} H(z,|F|) \;dz.
    \end{align*}
\end{lemma}  
\begin{proof}
    Consider a cut-off function $\eta\in C^\infty_0(B_R(x_0))$ satisfying
    $$
    0\leq \eta \leq 1, \quad \eta \equiv 1 \ \text{ in }B_r(x_0),\quad \text{and}\quad \|D\eta\|_{L^\infty}\leq \frac{2}{R-r}.
    $$
    For $\tau \in [S/2^2,S)$, we choose a sufficiently small $h_0>0$ such that there exists a cut-off function $\zeta\in C^\infty_0(I_{S-h_0}(t_0))$ satisfying 
    $$
    0\leq \zeta \leq 1,\quad \zeta\equiv 1 \ \text{ in }I_\tau(t_0),\quad \text{and}\quad \|\partial_t\zeta\|_{L^\infty}\leq \frac{3}{S-\tau}.
    $$
    Additionally, select $t_*\in I_\tau(t_0)$ and $\upepsilon \in (0,h_0)$. Define $\zeta_\upepsilon$ as follows:
    $$
    \zeta_\upepsilon(t):=
    \begin{cases}
        1,&\quad t\in(-\infty,t_*-\upepsilon),\\
        1-\frac{t-t_*+\upepsilon}{\upepsilon},&\quad t\in [t_*-\upepsilon,t_*],\\
        0,&\quad t\in (t_*,\infty).
    \end{cases}
    $$

    Fix $h\in(0,h_0)$. We express \eqref{eq: main equation} using Steklov averages as
    \begin{equation}    \label{eq: main equation with Steklov averages}
        \partial_t[u-(u)_{U_{R, S}(z_0)}]_h-\operatorname*{div}[\mathcal{A}(z, u,D u)]_h=-\operatorname*{div}[\mathcal{B}(z,F)]_h
    \end{equation}
    in $B_R(x_0)\times I_{S-h}(t_0)$. Then the function $[u-(u)_{U_{R,S}(z_0)}]_h \eta^q\zeta^2\zeta_\upepsilon$ belongs to 
    $$
    W^{1,2}_0(I_{S-h};L^2(B_R(x_0)))\cap L^q(I_{S-h};W^{1,q}_0(B_R(x_0))).
    $$
    By using $\varphi=[u-(u)_{U_{R, S}(z_0)}]_h \eta^q \zeta^2 \zeta_\upepsilon$ as a test function in \eqref{eq: main equation with Steklov averages}, we obtain
    \begin{equation*}
        \begin{aligned}
            \mathrm{I}+\mathrm{II}&:=\miint{U_{R, S}(z_0)} \partial_t[u-(u)_{U_{R, S}(z_0)}]_h \; \varphi \; d z+\miint{U_{R, S}(z_0)}[\mathcal{A}(z,u, D u)]_h \cdot D \varphi \; d z \\
            &=\miint{U_{R, S}(z_0)}[\mathcal{B}(z,F) ]_h \cdot D \varphi\; d z=: \mathrm{III}.
        \end{aligned}
    \end{equation*}
    We deduce from $|D u| \in L^q(\Omega_T)$ and \eqref{eq: assumption of F} that $\mathrm{II}$ and $\mathrm{III}$ are finite. Indeed, 
    \begin{align}
      \nonumber \mathrm{II} & \leq L \miint{U_{R, S}(z_0)}(a |Du|^{p-1})_h(x,t)|D\varphi(x,t)|\,dx\,dt \\
      &\label{finiteness of II} \qquad +L \miint{U_{R, S}(z_0)}(b |Du|^{q-1})_h(x,t)|D\varphi(x,t)|\,dx\,dt. 
    \end{align}
We will estimate the first term on the right-hand side and the estimation of the second term follows similarly. Using the definition of Steklov averages, the first term can be written as
\begin{align*}
&\miint{U_{R, S}(z_0)}(a |Du|^{p-1})_h(x,t)|D\varphi(x,t)|\,dx\,dt\\
&=\miint{U_{R, S}(z_0)}\dashint_{t}^{t+h}(a(x,s))^{\frac{p-1}{p}}|Du(x,s)|^{p-1}(a(x,s))^{\frac{1}{p}}|D\varphi(x,t)|\,ds\,dx\,dt.
\end{align*}
Now using H\"{o}lder's inequality and properties of Steklov averages, we get
\begin{align*}
&\miint{U_{R, S}(z_0)}|(a |Du|^{p-1})_h(x,t)||D\varphi(x,t)|\,dx\,dt \\
&\leq c \left(\miint{U_{R,S}(z_0)}a(x,t)|Du(x,t)|^p\,dx\,dt\right)^{\frac{p-1}{p}}\left(\miint{U_{R,S}(z_0)}\dashint_{t}^{t+h}a(x,s)|D \varphi(x,t)|^p\,ds\,dx\,dt\right)^{\frac{1}{p}}\\
&=c \left(\miint{U_{R,S}(z_0)}a(x,t)|Du(x,t)|^p\,dx\,dt\right)^{\frac{p-1}{p}}\left(\miint{U_{R,S}(z_0)}a_h(x,t)|D \varphi(x,t)|^p\,dx\,dt\right)^{\frac{1}{p}}.
\end{align*}
Using $||a||_{L^{\infty}(\Omega_T)}<\infty$ and $|Du|\in L^q(\Omega_T),$ we conclude that the above inequality is finite. By the same argument, we see that the second term of \eqref{finiteness of II} is also finite. We then estimate each term separately as follows. Note that estimation of the term $\mathrm{I}$ goes same as \cite[Lemma 2.3]{2023_Gradient_Higher_Integrability_for_Degenerate_Parabolic_Double-Phase_Systems}. We repeat it here for the sake of completeness.

\noindent \textbf{Estimate of $\mathrm{I}$:} Integrating by parts, we find 
    \begin{align}\nonumber
        \mathrm{I}&=\miint{U_{R,S}(z_0)} \frac{1}{2}(\partial_t |[u-(u)_{U_{R, S}(z_0)}]_h|^2)\eta^q\zeta^2\zeta_\upepsilon \; dz\\\label{eq : estimate of I in Lemma 3.1}
        &=-\miint{U_{R,S}(z_0)} |[u-(u)_{U_{R, S}(z_0)}]_h|^2\eta^q\zeta\zeta_\upepsilon\partial_t\zeta \;dz\\\nonumber
        &\qquad-\miint{U_{R,S}(z_0)} \frac{1}{2}|[u-(u)_{U_{R, S}(z_0)}]_h|^2 \eta^q \zeta^2 \partial_t\zeta_\upepsilon \; dz.
    \end{align}
    For the first term on the right-hand side of \eqref{eq : estimate of I in Lemma 3.1}, it follows from the definition of $\zeta$ that
    $$
    -\miint{U_{R,S}(z_0)} |[u-(u)_{U_{R, S}(z_0)}]_h|^2\eta^q\zeta\zeta_\upepsilon\partial_t\zeta \;dz\geq -3 \miint{U_{R,S}(z_0)}\frac{|[u-(u)_{U_{R,S}(z_0)}]_h|^2}{S-\tau} \; dz.
    $$
    Next, by the definition of $\zeta_\upepsilon$, we get 
    \begin{align*}
        &-\miint{U_{R,S}(z_0)} \frac{1}{2} |[u-(u)_{U_{R, S}(z_0)}]_h|^2 \eta^q \zeta^2 \partial_t\zeta_\upepsilon \; dz\\
        &\qquad =\frac{1}{|U_{R, S}|}\dashint_{t_*-\upepsilon}^{t_*}\int_{B_R(x_0)}\frac{1}{2} |[u-(u)_{U_{R, S}(z_0)}]_h|^2 \eta^q \zeta^2 \;dx dt\\
        &\qquad \geq \frac{1}{|U_{R, S}|}\dashint_{t_*-\upepsilon}^{t_*}\int_{B_r(x_0)} \frac{1}{2} |[u-(u)_{U_{R, S}(z_0)}]_h|^2\; dx dt.
    \end{align*}
    Thus, we obtain
    \begin{equation}\label{eq: final estimate of I in Lemma 3.1}
        \begin{aligned}
            \lim_{h\rightarrow 0^+}\lim_{\upepsilon \rightarrow 0^+} \mathrm{I} &\geq -3 \miint{U_{R,S}(z_0)}\frac{|u-(u)_{U_{R,S}(z_0)}|^2}{S-\tau} \; dz\\
            &\quad +\frac{1}{2|U_{R,S}|}\int_{B_r(x_0)} |u(x,t_*)-(u)_{U_{R, S}(z_0)}|^2\; dx.
        \end{aligned}
    \end{equation}

\noindent \textbf{Estimate of $\mathrm{II}$:} It follows from the definition of $\varphi$ that
    \begin{align}\nonumber
        \mathrm{II}&=\miint{U_{R,S}(z_0)} [\mathcal{A}(z,u,Du)]_h\cdot [Du]_h\eta^q\zeta^2\zeta_\upepsilon \;dz\\ \label{eq : estimate of II in Lemma 3.1}
        &\qquad + q\miint{U_{R,S}(z_0)} [\mathcal{A}(z,u,Du)]_h\cdot [u-(u)_{U_{R, S}(z_0)}]_h D\eta \eta^{q-1}\zeta^2\zeta_\upepsilon \; dz.
    \end{align}
For the first term of \eqref{eq : estimate of II in Lemma 3.1}, we deduce from \eqref{eq: growth condition of A} that
    \begin{align*}
        &\lim_{h\rightarrow 0^+}\lim_{\upepsilon \rightarrow 0^+} \miint{U_{R,S}(z_0)} [\mathcal{A}(z,u,Du)]_h\cdot [Du]_h\eta^q\zeta^2\zeta_\upepsilon \;dz\\
        &\qquad \geq \frac{\nu}{|U_{R, S}|}\int_{I_S(t_0)\cap(-\infty,t_*)}\int_{B_R(x_0)} H(z,|Du|)\eta^q\zeta^2 \; dx dt.
    \end{align*}
    Next, to estimate the second term in \eqref{eq : estimate of II in Lemma 3.1}, we apply \eqref{eq: growth condition of A} and the definition of $\eta$ to find that 
    \begin{align*}
        &\lim_{h\rightarrow 0^+}\lim_{\upepsilon \rightarrow 0^+} q\miint{U_{R,S}(z_0)} [\mathcal{A}(z,u,Du)]_h\cdot [u-(u)_{U_{R, S}(z_0)}]_h D\eta \eta^{q-1}\zeta^2\zeta_\upepsilon \; dz\\
        &\quad \geq -\frac{Lq}{|U_{R, S}|}\int_{I_S(t_0)\cap(-\infty,t_*)}\int_{B_R(x_0)} a(z)|Du|^{p-1}\eta^{q-1}\zeta^2\frac{|u-(u)_{U_{R, S}(z_0)}|}{R-r}\; dx dt\\
        &\qquad -\frac{Lq}{|U_{R, S}|}\int_{I_S(t_0)\cap(-\infty,t_*)}\int_{B_R(x_0)} b(z)|Du|^{q-1}\eta^{q-1}\zeta^2\frac{|u-(u)_{U_{R, S}(z_0)}|}{R-r}\; dx dt.
    \end{align*}
    Young's inequality with $\varepsilon=\frac{\nu}{8Ls}$ implies that
    \begin{align*}
        &\lim_{h\rightarrow 0^+}\lim_{\upepsilon \rightarrow 0^+} q\miint{U_{R,S}(z_0)} [\mathcal{A}(z,u,Du)]_h\cdot [u-(u)_{U_{R, S}(z_0)}]_h D\eta \eta^{q-1}\zeta^2\zeta_\upepsilon\; dz\\
        &\quad \geq -\frac{\nu}{4|U_{R, S}|}\int_{I_S(t_0)\cap(-\infty,t_*)} \int_{B_R(x_0)} H(z,|Du|)\eta^q\zeta^2 \;dx dt\\
        &\quad\qquad -c\miint{U_{R,S}(z_0)} H\left(z,\frac{|u-(u)_{U_{R, S}(z_0)}|}{R-r}\right) dz
    \end{align*}
    for some constant $c>0$ depending on $p,q,\nu$ and $L$, since $0\leq \eta \leq 1$ and $2 \leq p \leq q < \infty$. 
    It follows that 
    \begin{equation}\label{eq: final estimate of II in Lemma 3.1}
        \begin{aligned}
            \lim_{h\rightarrow 0^+}\lim_{\upepsilon \rightarrow 0^+} \mathrm{II} &\geq \frac{3\nu}{4|U_{R, S}|} \int_{I_S(t_0)\cap(-\infty,t_*)} \int_{B_R(x_0)} H(z,|Du|)\eta^q\zeta^2 \;dx dt\\
            &\quad -c\miint{U_{R,S}(z_0)} H\left(z,\frac{|u-(u)_{U_{R, S}(z_0)}|}{R-r}\right) dz.
        \end{aligned}
    \end{equation}

   \noindent\textbf{Estimate of $\mathrm{III}$:} As above, using the growth condition \eqref{eq: growth condition of B}, we obtain from Young's inequality that
    \begin{align}\nonumber
        \lim_{h\rightarrow 0^+}\lim_{\upepsilon \rightarrow 0^+} \mathrm{III} & \leq \miint{U_{R,S}(z_0)} H(z,|F|) \; dz\\ \nonumber
        &\quad +\frac{\nu}{2|U_{R, S}|} \int_{I_S(t_0)\cap(-\infty,t_*)} \int_{B_R(x_0)} H(z,|Du|)\eta^q \zeta^2 \;dx dt\\ \label{eq: final estimate of III in Lemma 3.1}
        &\quad +c\miint{U_{R,S}(z_0)} H\left(z,\frac{|u-(u)_{U_{R, S}(z_0)}|}{R-r}\right) dz.
    \end{align}
    Combining \eqref{eq: final estimate of I in Lemma 3.1}, \eqref{eq: final estimate of II in Lemma 3.1} and \eqref{eq: final estimate of III in Lemma 3.1} gives
    \begin{align*}
        &\frac{1}{|U_{R, S}|}\int_{B_r(x_0)} |u(x,t_*)-(u)_{U_{R, S}(z_0)}|^2\; dx\\
        &\qquad +\frac{1}{|U_{R, S}|} \int_{I_S(t_0)\cap(-\infty,t_*)} \int_{B_R(x_0)} H(z,|Du|)\eta^q\zeta^2 \;dx dt\\
        &\quad \leq c\miint{U_{R,S}(z_0)} H\left(z,\frac{|u-(u)_{U_{R, S}(z_0)}|}{R-r}\right) dz\\
        &\qquad + c \miint{U_{R,S}(z_0)}\frac{|u-(u)_{U_{R,S}(z_0)}|^2}{S-\tau} \; dz +c\miint{U_{R,S}(z_0)} H(z,|F|) \; dz,
    \end{align*}
    where $c$ depends only on $p,q,\nu$ and $L$. As $t_*\in I_\tau(t_0)$ is arbitrary, $\frac{R}{2}\leq r <R$ and $\frac{S}{4}\leq \tau < S$, we have
    \begin{align*}
        &\sup_{t\in (t_0-\tau,t_0+\tau)}\dashint_{B_r(x_0)} \frac{|u-(u)_{U_{R, S}(z_0)}|^2}{\tau}\; dx + \miint{U_{r,\tau}(z_0)} H(z,|Du|) \;dz\\
        &\qquad \leq c\miint{U_{R,S}(z_0)} H\left(z,\frac{|u-(u)_{U_{R, S}(z_0)}|}{R-r}\right)\; dz\\
        &\qquad \qquad + c \miint{U_{R,S}(z_0)}\frac{|u-(u)_{U_{R,S}(z_0)}|^2}{S-\tau} \; dz +c\miint{U_{R,S}(z_0)} H(z,|F|) \; dz
    \end{align*}
    for some constant $c>0$ depending on $n,p,q,\nu$ and $L$. Note that
    \begin{equation*}
        \sup_{t\in (t_0-\tau,t_0+\tau)}\dashint_{B_r(x_0)} \frac{|u-(u)_{U_{r, \tau}(z_0)}|^2}{\tau}\; dx \leq 2\sup_{t\in (t_0-\tau,t_0+\tau)}\dashint_{B_r(x_0)} \frac{|u-(u)_{U_{R, S}(z_0)}|^2}{\tau}\; dx.
    \end{equation*}
    This completes the proof.
\end{proof}
\begin{lemma}    \label{lem: gluing lemma}
Let $u$ be a weak solution to \eqref{eq: main equation}, and let $\eta\in C_0^\infty(B_R(x_0))$ be a function such that
\begin{equation*}    \label{eq: definition of eta in Lemma 3.2}
\eta\geq 0,\quad \dashint_{B_R(x_0)} \eta \; dx =1 \quad \text{and} \quad \|\eta\|_{L^\infty}+R\|D\eta\|_{L^\infty} \leq c(n).
\end{equation*}
Then for $U_{R,S}(z_0)\subset \Omega_T$, there exists a constant $c=c(n, L)$ such that 
\begin{align*}
\sup_{t_1,t_2\in(t_0-S,t_0+S)}|(u\eta)_{x_0;R}(t_2)-(u\eta)_{x_0;R}(t_1)|
 \leq c\frac{S}{R}\miint{U_{R,S}(z_0)}[\widetilde{H}(z,|Du|)+\widetilde{H}(z,|F|)]\; dz,
\end{align*}
 where $\widetilde{H}:\Omega_T\times \mr^+ \rightarrow \mr^+$ is denoted by $\widetilde{H}(z,\kappa)=a(z)\kappa^{p-1}+b(z)\kappa^{q-1}$ for $z \in \Omega_T$ and $\kappa\in\mr^+$.
\end{lemma}
\begin{proof}
 The proof can be completed following \cite[Lemma 2.4]{2023_Gradient_Higher_Integrability_for_Degenerate_Parabolic_Double-Phase_Systems}.   
\end{proof}
The next lemma presents a parabolic Poincar\'{e} inequality.
\begin{lemma}\label{lem : parabolic poincare inequality in Section 3}
    Let $u$ be a weak solution to \eqref{eq: main equation}. Then for $U_{R,S}(z_0)\subset \Omega_T$, $m\in(1,q]$ and $\theta\in(1/m,1]$, there exists a constant $c$ depending on $n$, $N$, $m$ and $L$ such that 
    \begin{align}\label{eq : parabolic poincare inequality in Section 3}
        \miint{U_{R,S}(z_0)} &\frac{|u-(u)_{U_{R,S}(z_0)}|^{\theta m}}{R^{\theta m}}\, dz \nonumber \\ &\leq c\miint{U_{R,S}(z_0)} |Du|^{\theta m} \,dz +c\left(\frac{S}{R^2}\miint{U_{R,S}(z_0)} [\widetilde{H}(z,|Du|)+\widetilde{H}(z,|F|)]\, dz \right)^{\theta m}.
    \end{align}
\end{lemma}
\begin{proof}
    In the same way as \cite[\text{Lemma}~2.5]{2023_Gradient_Higher_Integrability_for_Degenerate_Parabolic_Double-Phase_Systems}, the inequality \eqref{eq : parabolic poincare inequality in Section 3} can be derived.
\end{proof}

\section{Reverse H\"{o}lder inequalities}\label{Sec:Reverse_H}

We prove reverse H\"{o}lder type inequalities by dividing it into three cases, namely, $p$-intrinsic case, $q$-intrinsic case and $(p,q)$-intrinsic case.
The constant $K=K(n,p, q, \alpha, \delta, \omega_a,  [b]_\alpha, M)>1$ that appears below will be determined later. Moreover, throughout this section we consider radii $\rho\in(0,\rho_0)$ of the $p$-, $q$-, and $(p,q)$-intrinsic cylinders, where given any number $\eta \in (0,1)$, we choose $\rho_0>0$ sufficiently small such that $\omega_a(\rho)\le \eta$ and $[b]_\alpha\,\rho^\alpha \le \eta$ for every $\rho\in(0,\rho_0)$.

\subsection{Reverse H\"older inequality for $p$-phase}

In this subsection, we prove a reverse H\"{o}lder inequality for the $p$-phase. In the $p$-intrinsic case, we consider the following assumptions:
\begin{equation}    \label{eq: p-phase condition}
    \left\{
    \begin{aligned}
        &Ka(z_0)\lambda^p \geq b(z_0)\lambda^q,\\
        & a(z_0)\geq \frac{\delta}{2} \quad \text{and}\quad \frac{a(z)}{3}\leq a(z_0)\leq 3a(z) \ \, \text{for all}\,\,\, z\in Q_{\rho, \la}(z_0),\\
        &\miint{Q_{\tau,\lambda}(z_0)} [H(z,|Du|)+H(z,|F|)] \; dz <a(z_0)\lambda^p \ \, \text{ for every }\tau\in(\rho,c_v(2\rho)],\\
        &\miint{Q_{\rho,\lambda}(z_0)} [H(z,|Du|)+H(z,|F|)] \; dz =a(z_0)\lambda^p.
    \end{aligned}
    \right.
\end{equation}

\begin{lemma} \label{lem:p-phase1}
Let $u$ be a weak solution to \eqref{eq: main equation}. Then there exists a constant $c=c(\textnormal{\texttt{data}}) \geq 1$ such that for any $Q_{4\rho,\lambda}(z_0)\subset \Omega_T$ with \eqref{eq: p-phase condition}, $\tau \in [2\rho,4\rho]$ and $\theta \in \left(\frac{q-1}{p},1\right]$,
\begin{align*}
\miint{Q_{\tau,\lambda}(z_0)}[\widetilde{H}(z,|Du|)+\widetilde{H}(z,|F|)]\; d z &\leq c \miint{Q_{\tau,\lambda}(z_0)}\inf_{\tilde{z}\in Q_{\tau, \lambda}(z_0)}a(\tilde{z})(|D u|+|F|)^{p-1}\; dz\\
&\quad +c \lambda^{-1+\frac{p}{q}} \miint{Q_{\tau,\lambda}(z_0)} b(z)^{\frac{q-1}{q}}(|Du|+|F|)^{q-1}\; dz \\
&\quad +c\lambda^{\frac{\alpha p}{n+2}}\left(\miint{Q_{\tau,\lambda}(z_0)}(|Du|+|F|)^{\theta p} \;d z\right)^{\frac{1}{\theta}\left(\frac{p-1}{p}-\frac{\alpha}{n+2}\right)}.
\end{align*}
\end{lemma}
\begin{proof}
We note that $q-1<p$. Since $b(\cdot)\in C^{\alpha,\frac{\alpha}{2}}(\Omega_T)$, we have
\begin{align*}
&\miint{Q_{\tau,\lambda}(z_0)}[\widetilde{H}(z,|Du|)+\widetilde{H}(z,|F|)]\;dz \\
&\leq c\miint{Q_{\tau,\lambda}(z_0)} \inf_{\tilde{z}\in Q_{\tau,\lambda}(z_0)}a(\tilde{z}) \, (|Du|+|F|)^{p-1} \; dz+c \omega_a(\tau) \miint{Q_{\tau,\lambda}(z_0)} [|Du|^{p-1}+|F|^{p-1}] \; dz\\
&\quad +c\miint{Q_{\tau,\lambda}(z_0)}\inf_{\tilde{z}\in Q_{\tau,\lambda}(z_0)}b(\tilde{z}) \, (|Du|+|F|)^{q-1}\;dz+c\tau^\alpha \miint{Q_{\tau,\lambda}(z_0)} [|Du|^{q-1}+|F|^{q-1}] \; dz\\
& \leq c\miint{Q_{\tau,\lambda}(z_0)} \inf_{\tilde{z}\in Q_{\tau,\lambda}(z_0)}a(\tilde{z}) \, (|Du|+|F|)^{p-1} \; dz+c (\delta) \omega_a(\tau) \miint{Q_{\tau,\lambda}(z_0)} [\widetilde{H}(z,|Du|)+\widetilde{H}(z,|F|)] \; dz\\
&\quad +c\miint{Q_{\tau,\lambda}(z_0)}\inf_{\tilde{z}\in Q_{\tau,\lambda}(z_0)}b(\tilde{z}) \, (|Du|+|F|)^{q-1}\;dz+c\tau^\alpha \miint{Q_{\tau,\lambda}(z_0)} [|Du|^{q-1}+|F|^{q-1}] \; dz
\end{align*}
for some constant $c \geq 1$ depending on $p$, $q$ and $[b]_\alpha$.
We consider $\tau>0$ so small that $\displaystyle c (\delta) \omega_a(\tau) \leq \frac{1}{2}$. Then we have
\begin{align*}
&\miint{Q_{\tau,\lambda}(z_0)}[\widetilde{H}(z,|Du|)+\widetilde{H}(z,|F|)]\;dz\\
&\leq c\miint{Q_{\tau,\lambda}(z_0)} \inf_{\tilde{z}\in Q_{\tau,\lambda}(z_0)}a(\tilde{z}) \, (|Du|+|F|)^{p-1} \; dz+c\miint{Q_{\tau,\lambda}(z_0)}\inf_{\tilde{z}\in Q_{\tau,\lambda}(z_0)}b(\tilde{z}) \, (|Du|+|F|)^{q-1}\;dz\\
&\quad +c\tau^\alpha \miint{Q_{\tau,\lambda}(z_0)} (|Du|^{q-1}+|F|^{q-1}) \; dz.
\end{align*}
We estimate the second term on the right-hand side as follows:
    \begin{align*}
        &\miint{Q_{\tau,\lambda}(z_0)}\inf_{\tilde{z}\in Q_{\tau,\lambda}(z_0)} b(\tilde{z})(|Du|+|F|)^{q-1} \; dz \\
        &\qquad\leq a(z_0)^{\frac{1}{q}} K^{\frac{1}{q}}\lambda^{\frac{p}{q}-1}\miint{Q_{\tau,\lambda}(z_0)} \inf_{\tilde{z}\in Q_{\tau,\lambda}(z_0)}b(\tilde{z})^{\frac{q-1}{q}} (|Du|+|F|)^{q-1}\; dz\\
        &\qquad \leq a(z_0)^{\frac{1}{q}} K^{\frac{1}{q}}\lambda^{\frac{p}{q}-1}\miint{Q_{\tau,\lambda}(z_0)} b(z)^{\frac{q-1}{q}} (|Du|+|F|)^{q-1}\; dz.
    \end{align*}
Note that $|Q_{\tau, \lambda}(z_0)|=c(n)\tau^{n+2}\frac{\lambda^{2-p}}{a(z_0)}$. Then it follows from H\"{o}lder inequality that
\begin{align*}
&\tau^\alpha \miint{Q_{\tau,\lambda}(z_0)}\left(|Du|+|F|\right)^{q-1}\; dz\\
&\; \leq \tau^\alpha\left(\miint{Q_{\tau,\lambda}(z_0)} \left(|Du|+|F|\right)^{\theta p}\; dz\right)^{\frac{q-1}{\theta p}}\\
&\;= \tau^\alpha\left(\miint{Q_{\tau,\lambda}(z_0)} \left(|Du|+|F|\right)^{\theta p}\; dz\right)^{\frac{1}{\theta}\frac{\gamma}{p}}\left(\miint{Q_{\tau,\lambda}(z_0)} \left(|Du|+|F|\right)^{\theta p}\; dz\right)^{\frac{1}{\theta}\frac{q-1-\gamma}{p}}\\
&\; \leq c\tau^{\alpha-\frac{(n+2)\gamma}{p}}\lambda^{\frac{(p-2)\gamma}{p}}\left(\iints{Q_{\tau,\lambda}(z_0)} \left(|Du|+|F|\right)^{p}\; dz\right)^{\frac{\gamma}{p}}\left(\miint{Q_{\tau,\lambda}(z_0)} \left(|Du|+|F|\right)^{\theta p}\; dz\right)^{\frac{1}{\theta}\frac{q-1-\gamma}{p}}.
\end{align*}
Now choosing $0<\gamma=\frac{\alpha p}{n+2}<p-1,$ and using the second and third conditions of \eqref{eq: p-phase condition}, we get
\begin{align*}
\tau^\alpha \miint{Q_{\tau,\lambda}(z_0)}\left(|Du|+|F|\right)^{q-1}\; dz & \leq c a(z_0)^{\frac{q-p}{p}} \lambda^{\frac{\alpha p}{n+2}}\left(\miint{Q_{\tau,\lambda}(z_0)} \left(|Du|+|F|\right)^{\theta p}\; dz\right)^{\frac{1}{\theta}\frac{p-1-\gamma}{p}} \\
& \leq c\lambda^{\frac{\alpha p}{n+2}}\left(\miint{Q_{\tau,\lambda}(z_0)} \left(|Du|+|F|\right)^{\theta p}\; dz\right)^{\frac{1}{\theta}\frac{p-1-\gamma}{p}}.
\end{align*}
This completes the proof.
\end{proof}

\begin{lemma} \label{lem:p-term of parabolic poincare inequality on p-intrinsic cylinders}
Let $u$ be a weak solution to \eqref{eq: main equation}. Then there exists a constant $c=c(\textnormal{\texttt{data}}) \geq 1$ such that for any $Q_{4\rho,\lambda}(z_0)\subset \Omega_T$ with \eqref{eq: p-phase condition}, $\tau \in [2\rho,4\rho]$ and $\theta \in (\frac{q-1}{p},1]$, 
\begin{align*}
& \miint{Q_{\tau,\lambda}(z_0)}\frac{|u-(u)_{z_0;\tau,\lambda}|^{\theta p}}{\tau^{\theta p}}\, dz \\
& \quad \leq c \miint{Q_{\tau,\lambda}(z_0)}H(z,|Du|)^{\theta}\, dz+c \lambda^{\left(2-p+\frac{\alpha p}{n+2}\right)\theta p} \left(\miint{Q_{\tau,\lambda}(z_0)} (|Du|+|F|)^{\theta p}\, dz\right)^{p-1-\frac{\alpha p}{n+2}} \\
&\qquad +c\left(\miint{Q_{\tau,\lambda}(z_0)}H(z,|F|) \, d z\right)^{\theta}.
\end{align*}
\end{lemma}

\begin{proof}
We first observe from the second condition of \eqref{eq: p-phase condition} that
\begin{equation} \label{eq:bounds_of_a}
\frac{\delta}{6} \leq a(z) \leq \Norm{a}_{L^{\infty}(\Omega_T)}, \quad \forall z \in Q_{4\rho,\lambda}(z_0).
\end{equation}
By \cref{lem : parabolic poincare inequality in Section 3} and \cref{lem:p-phase1}, there exists a constant $c$ depending on $\texttt{data}$ such that 
\begin{align*}
\miint{Q_{\tau,\lambda}(z_0)}\frac{|u-(u)_{z_0;\tau,\lambda}|^{\theta p}}{\tau^{\theta p}}\; dz&\leq c\miint{Q_{\tau,\lambda}(z_0)} |Du|^{\theta p} \; dz + c\underbrace{\left(\lambda^{2-p}\miint{Q_{\tau,\lambda}(z_0)} (|Du|+|F|)^{p-1}\;dz \right)^{\theta p}}_{\mathrm{I}}\\
&\quad + c\underbrace{\left(\lambda^{1-p+\frac{p}{q}}\miint{Q_{\tau,\lambda}(z_0)} b(z)^{\frac{q-1}{q}}(|Du|+|F|)^{q-1}\; dz \right)^{\theta p}}_{\mathrm{II}}\\
&\quad + c\lambda^{(2-p+\frac{\alpha p}{n+2})\theta p}\left(\miint{Q_{\tau,\lambda}(z_0)} (|Du|+|F|)^{\theta p}\; dz \right)^{p-1-\frac{\alpha p}{n+2}}.
\end{align*}

\noindent\textbf{Estimate of $\mathrm{I}$:}
Using H\"{o}lder's inequality, the second condition of  \eqref{eq: p-phase condition} and \eqref{eq:bounds_of_a}, we estimate
\begin{align*}
&\lambda^{(2-p)\theta p}\left(\miint{Q_{\tau,\lambda}(z_0)}(|Du|+|F|)^{p-1} \; dz\right)^{\theta p}\\
&\qquad \leq 3^{p-1}\frac{\lambda^{(2-p)\theta p}}{a(z_0)^{\theta(p-2)}}\left(\miint{Q_{\tau,\lambda}(z_0)}a(z)^{\theta}(|Du|+|F|)^{\theta p} \; dz\right)^{p-2} \miint{Q_{\tau,\lambda}(z_0)}\frac{a^{\theta}(z)}{a^{\theta}(z_0)}(|Du|+|F|)^{\theta p}\, dz\\
&\qquad\leq 3^{p-1}\left(\frac{2}{\delta}\right)^{\theta}\miint{Q_{\tau,\lambda}(z_0)}a(z)^{\theta}(|Du|+|F|)^{\theta p} \; dz\leq c \miint{Q_{\tau,\lambda}(z_0)}\left(H_1(z, |Du|)+H_1(z, |F|)\right)^{\theta}\, dz.
\end{align*}
\textbf{Estimate of $\mathrm{II}$:} Similarly, using H\"{o}lder's inequality and the third condition of  \eqref{eq: p-phase condition}, we get
\begin{align*}
&\lambda^{\left(1-p+\frac{p}{q}\right)\theta p}\left(\miint{Q_{\tau,\lambda}(z_0)} b(z)^{\frac{q-1}{q}}(|Du|+|F|)^{q-1}\; dz \right)^{\theta p}\\
&\leq \lambda^{\left(1-p+\frac{p}{q}\right)\theta p}\left(\miint{Q_{\tau,\lambda}(z_0)} b(z)^{\theta}(|Du|+|F|)^{\theta q}\; dz \right)^{\frac{p(q-1)}{q}}\\
&=\lambda^{\left(1-p+\frac{p}{q}\right)\theta p}\left(\miint{Q_{\tau,\lambda}(z_0)} b(z)^{\theta}(|Du|+|F|)^{\theta q}\; dz \right)^{\frac{p(q-1)-q}{q}} \miint{Q_{\tau,\lambda}(z_0)} b(z)^{\theta}(|Du|+|F|)^{\theta q}\; dz \\
&\overset{\eqref{eq: p-phase condition}}{\leq}\lambda^{\left(1-p+\frac{p}{q}\right)\theta p}\lambda^{-\left(1-p+\frac{p}{q}\right)\theta p}\miint{Q_{\tau,\lambda}(z_0)} b(z)^{\theta}(|Du|+|F|)^{\theta q}\; dz=\miint{Q_{\tau,\lambda}(z_0)} b(z)^{\theta}(|Du|+|F|)^{\theta q}\; dz,
\end{align*}
where we have used
\begin{align*}
   \frac{p+q}{pq}=\frac{1}{p}+\frac{1}{q} \leq 1 \implies p+q\leq pq \implies q\leq p(q-1)
\end{align*}
since $2 \leq p \leq q$.
Finally, using H\"{o}lder inequality, we get the desired estimate.
\end{proof}
\begin{lemma} \label{lem: s-term of parabolic poincare inequality on p-intrinsic cylinders}
Let $u$ be a weak solution to \eqref{eq: main equation}. Then there exists a constant $c=c(\textnormal{\texttt{data}}) \geq 1$ such that for any $Q_{4\rho,\lambda}(z_0)\subset \Omega_T$ with \eqref{eq: p-phase condition}, $\tau \in [2\rho,4\rho]$ and $\theta \in ((q-1)/p,1]$, 
\begin{align*}
&\miint{Q_{\tau,\lambda}(z_0)} \inf_{\tilde{z}\in Q_{\tau,\lambda}(z_0)}b(\tilde{z})^\theta\frac{|u-(u)_{z_0;\tau,\lambda}|^{\theta q}}{\tau^{\theta q}}\; dz \\
&\quad \leq c \miint{Q_{\tau,\lambda}(z_0)}(H(z,|Du|))^{\theta}\; dz +c\left(\miint{Q_{\tau,\lambda}(z_0)}H(z,|F|)\; d z\right)^{\theta} \\
&\qquad +c \lambda^{\left(2-p+\frac{\alpha p}{n+2}\right)\theta p} \left(\miint{Q_{\tau,\lambda}(z_0)} (|Du|+|F|)^{\theta p}\; dz\right)^{p-1-\frac{\alpha p}{n+2}}.
\end{align*}
\end{lemma}
\begin{proof}
    The proof follows from \cite[Lemma 3.3]{2023_Gradient_Higher_Integrability_for_Degenerate_Parabolic_Double-Phase_Systems}.
\end{proof}
\begin{lemma}\label{lem : Lemma 5.1}
    Let $u$ be a weak solution to \eqref{eq: main equation}. Then for $Q_{c_v(2 \rho),\lambda}(z_0) \subset \Omega_T$ with \eqref{eq: p-phase condition}, there exists a constant $c=c(\textnormal{\texttt{data}}) \geq 1$ such that
    $$
    S(u)_{z_0 ; 2 \rho,\lambda}=\sup _{I_{2 \rho,\lambda}(t_0)} \dashint_{B_{2 \rho}(x_0)} \frac{\left|u-(u)_{z_0 ; 2 \rho,\lambda}\right|^2}{(2 \rho)^2}\; dx \leq c \lambda^2 .
    $$
\end{lemma}
\begin{proof}
Let $2\rho \leq \rho_1<\rho_2 \leq 4\rho$. By \cref{lem : the Caccioppoli inequality}, there exists a constant $c$ depending on $n,p,q,\nu$ and $L$ such that 
\begin{align}
\lambda^{p-2}S(u)_{z_0;\rho_1,\lambda}&\leq \frac{c\rho_2^q}{(\rho_2-\rho_1)^q}\miint{Q_{\rho_2,\lambda}(z_0)}\left( a(z)\frac{|u-(u)_{z_0;\rho_2,\lambda}|^p}{\rho^p_2}+b(z)\frac{|u-(u)_{z_0;\rho_2,\lambda}|^q}{\rho^q_2} \right) dz \nonumber\\
&\quad + \frac{c\rho_2^2 \lambda^{p-2}}{(\rho_2-\rho_1)^2}\miint{Q_{\rho_2,\lambda}(z_0)}\frac{|u-(u)_{z_0;\rho_2,\lambda}|^2}{\rho_2^2}\; dz +c\miint{Q_{\rho_2,\lambda}(z_0)} H(z,|F|)\; dz. \label{eq : second equation in lemma 3.4}
\end{align}
 We note that the term involving $a(z)$ in the right-hand side can be estimated by using \cref{lem:p-term of parabolic poincare inequality on p-intrinsic cylinders} with $\theta=1,$
 \begin{align*}
     \miint{Q_{\rho_2,\lambda}(z_0)} a(z)\frac{|u-(u)_{z_0;\rho_2,\lambda}|^p}{\rho^p_2}\, dz \leq ||a||_{L^{\infty}(\Omega_T)}\miint{Q_{\rho_2,\lambda}(z_0)}\frac{|u-(u)_{z_0;\rho_2,\lambda}|^p}{\rho^p_2}\, dz \leq c||a||_{L^{\infty}(\Omega_T)}\la^p.
 \end{align*}
 The rest of the proof follows from \cite[Lemma 4.3]{2023_Gradient_Higher_Integrability_for_Degenerate_Parabolic_Double-Phase_Systems}. Indeed, using the H\"{o}lder regularity of $b(\cdot), $ we have
 \begin{align*}
     &\miint{Q_{\rho_2, \lambda}(z_0)}b(z)\frac{|u-(u)_{z_0;\rho_2,\lambda}|^q}{\rho^q_2}\; dz\\
     &\quad \leq \underbrace{\miint{Q_{\rho_2, \lambda}(z_0)} \inf_{\tilde{z}\in Q_{\rho_2, \lambda}(z_0)}b(\tilde{z})\frac{|u-(u)_{z_0;\rho_2,\lambda}|^q}{\rho^q_2}\,dz}_{\mathrm{I}}+ \underbrace{[b]_{\alpha}\rho^{\alpha}_2\miint{Q_{\rho_2, \lambda}(z_0)} \frac{|u-(u)_{z_0;\rho_2,\lambda}|^q}{\rho^q_2}\,dz}_{\mathrm{II}}.
 \end{align*}
 \textbf{Estimate of $\mathrm{I}$:} Directly using \cref{lem: s-term of parabolic poincare inequality on p-intrinsic cylinders} with $\theta =1$ and the second condition of \eqref{eq: p-phase condition}, we obtain
 \begin{align*}
  \miint{Q_{\rho_2, \lambda}(z_0)} \inf_{\tilde{z}\in Q_{\rho_2, \lambda}(z_0)}b(\tilde{z})\frac{|u-(u)_{z_0;\rho_2,\lambda}|^q}{\rho^q_2}\,dz \leq c \lambda^p . 
 \end{align*}
 \textbf{Estimate of $\mathrm{II}$:} Using \cref{lem : lemma 2.1}, we estimate
 \begin{align*}
  [b]_{\alpha}\rho^{\alpha}_2\miint{Q_{\rho_2, \lambda}(z_0)} \frac{|u-(u)_{z_0;\rho_2,\lambda}|^q}{\rho^q_2}\,dz \leq c\lambda^p .  
 \end{align*}
 Moreover, following \cite[Lemma 4.3]{2023_Gradient_Higher_Integrability_for_Degenerate_Parabolic_Double-Phase_Systems}, we estimate the second term of \eqref{eq : second equation in lemma 3.4} as follows:
 \begin{align*}
\miint{Q_{\rho_2,\lambda}(z_0)}\frac{|u-(u)_{z_0;\rho_2,\lambda}|^2}{\rho_2^2}\; dz \leq c \lambda \left(S(u)_{z_0, \rho_2, \lambda}\right)^{\frac{1}{2}}.
 \end{align*}
Combining all the estimates, using the third condition of \eqref{eq: p-phase condition} and Young's inequality, we obtain
\begin{align*}
 S(u)_{z_0, \rho_1, \lambda}\leq \frac{1}{2}S(u)_{z_0, \rho_2, \lambda}+c \left[\frac{\rho^q_2}{(\rho_2-\rho_1)^q}+\frac{\rho^4_2}{(\rho_2-\rho_1)^4}\right]\lambda^2.   
\end{align*}
Finally, applying \cref{iter_lemma}, we conclude the proof.
\end{proof}
\begin{lemma}\label{lem : lemma 5.2}
    Let $u$ be a weak solution to \eqref{eq: main equation}. Then for $Q_{c_v(2 \rho),\lambda}\left(z_0\right) \subset \Omega_T$ with \eqref{eq: p-phase condition}, there exist constants $c=c(\textnormal{\texttt{data}}) \geq 1$ and $\theta_0=\theta_0(n, p, q) \in(0,1)$ such that for any $\theta \in\left(\theta_0, 1\right)$,
    $$
    \begin{aligned}
    &\miint{Q_{2 \rho,\lambda}\left(z_0\right)} H\left(z,\frac{\left|u-(u)_{z_0 ; 2 \rho,\lambda}\right|}{2 \rho}\right) d z \\
    & \quad \leq c \miint{Q_{2 \rho,\lambda}\left(z_0\right)}\left(\inf_{Q_{2\rho, \lambda}(z_0)}a(\tilde{z})^{\theta}\frac{\left|u-(u)_{z_0 ; 2 \rho,\lambda}\right|^{\theta p}}{(2 \rho)^{\theta p}}+\inf_{Q_{2\rho, \lambda}(z_0)}a(\tilde{z})^{\theta}|D u|^{\theta p}\right) d z\;\left(S(u)_{z_0 ; 2 \rho,\lambda}\right)^{\frac{(1-\theta) p}{2}} \\
    &\qquad+c \miint{Q_{2 \rho,\lambda}\left(z_0\right)}\left(\inf _{Q_{2 \rho,\lambda}\left(z_0\right)} b(\tilde{z})^\theta \frac{\left|u-(u)_{z_0 ; 2 \rho,\lambda}\right|^{\theta q}}{(2 \rho)^{\theta q}}+\inf _{Q_{2 \rho,\lambda}\left(z_0\right)} b(\tilde{z})^\theta|D u|^{\theta q}\right) dz \\
    & \qquad \quad \times \lambda^{(p-q)(1-\theta)}\;\left(S(u)_{z_0 ; 2 \rho,\lambda}\right)^{\frac{(1-\theta) q}{2}}.
    \end{aligned}
    $$
\end{lemma}
\begin{proof}
 Using the uniform continuity of $a(\cdot)$ and H\"{o}lder continuity of $b(\cdot), $ we have
\begin{align*}
&\miint{Q_{2 \rho,\lambda}\left(z_0\right)} H\left(z,\frac{\left|u-(u)_{z_0 ; 2 \rho,\lambda}\right|}{2 \rho}\right) dz\\
&\leq \omega_a(2\rho)\miint{Q_{2 \rho,\lambda}\left(z_0\right)}\frac{\left|u-(u)_{z_0 ; 2 \rho,\lambda}\right|^p}{(2 \rho)^p}\, dz + \miint{Q_{2 \rho,\lambda}\left(z_0\right)} \inf_{Q_{2 \rho,\lambda}\left(z_0\right)} a(\tilde{z}) \frac{\left|u-(u)_{z_0 ; 2 \rho,\lambda}\right|^p}{(2 \rho)^p}\, dz\\
&\quad + \miint{Q_{2 \rho,\lambda}\left(z_0\right)} \inf_{Q_{2 \rho,\lambda}\left(z_0\right)} b(\tilde{z})\frac{\left|u-(u)_{z_0 ; 2 \rho,\lambda}\right|^q}{(2 \rho)^q}\, dz +[b]_{\alpha}(2\rho)^{\alpha}\miint{Q_{2 \rho,\lambda}\left(z_0\right)}\frac{\left|u-(u)_{z_0 ; 2 \rho,\lambda}\right|^q}{(2 \rho)^q}\, dz\\
&\overset{\eqref{eq: p-phase condition}}{\leq} c(\delta)\omega_a(2\rho)\miint{Q_{2 \rho,\lambda}\left(z_0\right)} H\left(z,\frac{\left|u-(u)_{z_0 ; 2 \rho,\lambda}\right|}{2 \rho}\right) dz + \miint{Q_{2 \rho,\lambda}\left(z_0\right)} \inf_{Q_{2 \rho,\lambda}\left(z_0\right)} a(\tilde{z}) \frac{\left|u-(u)_{z_0 ; 2 \rho,\lambda}\right|^p}{(2 \rho)^p}\, dz\\
&\quad + \miint{Q_{2 \rho,\lambda}\left(z_0\right)} \inf_{Q_{2 \rho,\lambda}\left(z_0\right)} b(\tilde{z})\frac{\left|u-(u)_{z_0 ; 2 \rho,\lambda}\right|^q}{(2 \rho)^q}\, dz +[b]_{\alpha}(2\rho)^{\alpha}\miint{Q_{2 \rho,\lambda}\left(z_0\right)}\frac{\left|u-(u)_{z_0 ; 2 \rho,\lambda}\right|^q}{(2 \rho)^q}\, dz.
\end{align*}
We choose $\rho>0$ so small that $c(\delta)\omega_a(2\rho)\leq \frac{1}{2}$.
Then it follows that
\begin{align*}
\miint{Q_{2 \rho,\lambda}\left(z_0\right)} H\left(z,\frac{\left|u-(u)_{z_0 ; 2 \rho,\lambda}\right|}{2 \rho}\right) dz &\leq c \underbrace{\miint{Q_{2 \rho,\lambda}\left(z_0\right)} \inf_{Q_{2 \rho,\lambda}\left(z_0\right)} a(\tilde{z}) \frac{\left|u-(u)_{z_0 ; 2 \rho,\lambda}\right|^p}{(2 \rho)^p}\, dz}_{\mathrm{I}}\\
&\quad +c\miint{Q_{2 \rho,\lambda}\left(z_0\right)} \inf_{Q_{2 \rho,\lambda}\left(z_0\right)} b(\tilde{z})\frac{\left|u-(u)_{z_0 ; 2 \rho,\lambda}\right|^q}{(2 \rho)^q}\, dz\\
&\quad +c[b]_{\alpha}(2\rho)^{\alpha}\miint{Q_{2 \rho,\lambda}\left(z_0\right)}\frac{\left|u-(u)_{z_0 ; 2 \rho,\lambda}\right|^q}{(2 \rho)^q}\, dz.
\end{align*}
\textbf{Estimate of $\mathrm{I}$:} Choosing $\sigma=p, s=\theta p, r=2, \vartheta=\theta,$ in  \cref{lem : lemma 2.1}, we estimate
\begin{align*}
&\miint{Q_{2 \rho,\lambda}\left(z_0\right)} \inf_{Q_{2 \rho,\lambda}\left(z_0\right)} a(\tilde{z}) \frac{\left|u-(u)_{z_0 ; 2 \rho,\lambda}\right|^p}{(2 \rho)^p}\, dz \\
&\leq \miint{Q_{2 \rho,\lambda}\left(z_0\right)} \left(\inf_{Q_{2 \rho,\lambda}\left(z_0\right)} a(\tilde{z})^\theta \frac{\left|u-(u)_{z_0 ; 2 \rho,\lambda}\right|^{\theta p}}{(2 \rho)^{\theta p}}+ \inf_{Q_{2 \rho,\lambda}\left(z_0\right)} a(\tilde{z})^\theta |Du|^{\theta p} \right) dz \\
&\quad \times \inf_{Q_{2 \rho,\lambda}\left(z_0\right)} a(\tilde{z})^{1-\theta}\left(S(u)_{z_0, 2\rho, \lambda}\right)^{\frac{(1-\theta)p}{2}}\\
&\leq c \miint{Q_{2 \rho,\lambda}\left(z_0\right)} \left(\inf_{Q_{2 \rho,\lambda}\left(z_0\right)} a(\tilde{z})^\theta \frac{\left|u-(u)_{z_0 ; 2 \rho,\lambda}\right|^{\theta p}}{(2 \rho)^{\theta p}}+ \inf_{Q_{2 \rho,\lambda}\left(z_0\right)} a(\tilde{z})^\theta |Du|^{\theta p} \right) dz \left(S(u)_{z_0, 2\rho, \lambda}\right)^{\frac{(1-\theta)p}{2}},
\end{align*}
where we have used $\displaystyle{\inf_{Q_{2 \rho,\lambda}\left(z_0\right)}} a(\tilde{z})^{1-\theta} \leq \max\left\{1, ||a||^{1-\theta_0}_{L^{\infty}(\Omega_T)}\right\}.$ Now following \cite[Lemma 4.4]{2023_Gradient_Higher_Integrability_for_Degenerate_Parabolic_Double-Phase_Systems}, we have

\begin{align*}
&\miint{Q_{2 \rho,\lambda}\left(z_0\right)} \inf_{Q_{2 \rho,\lambda}\left(z_0\right)} b(\tilde{z})\frac{\left|u-(u)_{z_0 ; 2 \rho,\lambda}\right|^q}{(2 \rho)^q}\, dz\\
&\leq c\miint{Q_{2 \rho,\lambda}\left(z_0\right)} \left(\inf_{Q_{2 \rho,\lambda}\left(z_0\right)} b(\tilde{z})^\theta \frac{\left|u-(u)_{z_0 ; 2 \rho,\lambda}\right|^{\theta q}}{(2 \rho)^{\theta q}}+ \inf_{Q_{2 \rho,\lambda}\left(z_0\right)} b(\tilde{z})^\theta |Du|^{\theta q} \right) dz \, \lambda^{(p-q)(1-\theta)} \left(S(u)_{z_0, 2\rho, \lambda}\right)^{\frac{(1-\theta)p}{2}}
\end{align*}
and 
\begin{align*}
&(2\rho)^{\alpha}\miint{Q_{2 \rho,\lambda}\left(z_0\right)}\frac{\left|u-(u)_{z_0 ; 2 \rho,\lambda}\right|^q}{(2 \rho)^q}\, dz\\
&\leq c \miint{Q_{2 \rho,\lambda}\left(z_0\right)} \left(\inf_{Q_{2 \rho,\lambda}\left(z_0\right)} a(\tilde{z})^\theta \frac{\left|u-(u)_{z_0 ; 2 \rho,\lambda}\right|^{\theta p}}{(2 \rho)^{\theta p}}+ \inf_{Q_{2 \rho,\lambda}\left(z_0\right)} a(\tilde{z})^\theta |Du|^{\theta p} \right) dz \left(S(u)_{z_0, 2\rho, \lambda}\right)^{\frac{(1-\theta)p}{2}},
\end{align*}
where we also used the second condition of \eqref{eq: p-phase condition}. Combining the estimates above proves the lemma.
\end{proof}
Now we are ready to prove the reverse H\"{o}lder inequality for $p$-phase.

\begin{lemma}\label{lem : lemma 5.3}
    Let $u$ be a weak solution to \eqref{eq: main equation}. Then for $Q_{c_v(2 \rho),\lambda}\left(z_0\right) \subset \Omega_T$ with \eqref{eq: p-phase condition}, there exist constants $c=c(\textnormal{\texttt{data}}) \geq 1$ and $\theta_0=\theta_0(n,p,q) \in (0,1)$ such that for any $\theta\in (\theta_0,1),$
    \begin{align*}
        \miint{Q_{\rho,\lambda}(z_0)} H(z,|Du|)\; dz  \leq c \left(\miint{Q_{2\rho,\lambda}(z_0)} [H(z,|Du|)]^\theta\;dz\right)^\frac{1}{\theta}+c\miint{Q_{2\rho,\lambda}(z_0)} H(z,|F|)\; dz.
    \end{align*}
\end{lemma}
\begin{proof}
From the energy estimate,  \cref{lem : the Caccioppoli inequality}, we have
\begin{align*}
\miint{Q_{\rho, \lambda}(z_0)}H(z,|Du|)\;dz
& \leq c\miint{Q_{2\rho, \lambda}(z_0)}  H\left(z,\frac{|u-(u)_{z_0, 2\rho, \lambda}|}{2\rho}\right) dz\\
&\quad + c\la^{p-2}\miint{Q_{2\rho, \lambda}(z_0)}\frac{|u-(u)_{z_0, 2\rho, \lambda}|^2}{(2\rho)^2}\;dz +c\miint{Q_{2\rho, \lambda}(z_0)} H(z,|F|) \;dz.
\end{align*}  
Now applying \cref{lem : Lemma 5.1} and  \cref{lem : lemma 5.2}, we estimate the right-hand side of the above inequality as follows:
\begin{align*}
&\miint{Q_{2\rho, \lambda}(z_0)}  H\left(z,\frac{|u-(u)_{z_0, 2\rho, \lambda}|}{2\rho}\right) dz \leq c \la^{(1-\theta)p} \miint{Q_{2\rho,\lambda}(z_0)} H(z, |Du|)^{\theta}\, dz\\
& \qquad \qquad + \la^{(1-\theta)p} \miint{Q_{2\rho,\lambda}(z_0)} \left(\inf_{Q_{2\rho, \lambda}(z_0)}a(\tilde{z})^{\theta}\frac{\left|u-(u)_{z_0 ; 2 \rho,\lambda}\right|^{\theta p}}{(2 \rho)^{\theta p}}+ \inf _{Q_{2 \rho,\lambda}\left(z_0\right)} b(\tilde{z})^\theta \frac{\left|u-(u)_{z_0 ; 2 \rho,\lambda}\right|^{\theta q}}{(2 \rho)^{\theta q}} \right) dz.   
\end{align*}
The rest of the proof follows from \cite[Lemma 4.5]{2023_Gradient_Higher_Integrability_for_Degenerate_Parabolic_Double-Phase_Systems} by applying \cref{lem:p-term of parabolic poincare inequality on p-intrinsic cylinders}, \cref{lem: s-term of parabolic poincare inequality on p-intrinsic cylinders}, Young's inequality and the fourth condition of \eqref{eq: p-phase condition}. For the sake of completeness, we include it here.

Indeed, by \cref{lem:p-term of parabolic poincare inequality on p-intrinsic cylinders} and \cref{lem: s-term of parabolic poincare inequality on p-intrinsic cylinders}, we have
\begin{align}\label{EQQ3.4}
&\miint{Q_{2\rho,\lambda}(z_0)} H\left(z,\frac{|u-(u)_{z_0, 2\rho, \lambda}|}{2\rho}\right) dz\nonumber\\
&\quad \leq c\lambda^{p(1-\beta \theta)}\left(\miint{Q_{2\rho, \lambda}(z_0)}H(z, |Du|)^{\theta}\, dz\right)^{\beta}+c\lambda^{p(1-\beta \theta)}\left(\miint{Q_{2\rho, \lambda}(z_0)}H(z, |F|)\,dz\right)^{\theta \beta}
\end{align}
and 
\begin{align}\label{Equation3.5}
\la^{p-2}\miint{Q_{2\rho, \lambda}(z_0)}\frac{|u-(u)_{z_0, 2\rho, \lambda}|^2}{(2\rho)^2}\;dz&\leq c\lambda^{p-\beta}\left(\miint{Q_{2\rho, \lambda}(z_0)}H(z, |Du|)^{\theta}\, dz\right)^{\frac{\beta}{\theta p}}\nonumber\\
&\quad +c\lambda^{p-\beta}\left(\miint{Q_{2\rho, \lambda}(z_0)}H(z, |F|)\, dz\right)^{\frac{\beta}{p}}    
\end{align}
for $$\beta:=\min\left\{p-1-\frac{\alpha p}{n+2}, \frac{1}{2}\right\}.$$ Now applying Young's inequality in \eqref{EQQ3.4} with $\left(\frac{1}{1-\beta \theta}, \frac{1}{\beta \theta}\right)$ and in \eqref{Equation3.5} with $\left(\frac{p}{p-\beta}, \frac{p}{\beta}\right)$, we get
\begin{align*}
&\miint{Q_{2\rho,\lambda}(z_0)} H\left(z,\frac{|u-(u)_{z_0, 2\rho, \lambda}|}{2\rho}\right) dz \\
&\qquad \leq \varepsilon \lambda^p+c(\varepsilon)\left[\left(\miint{Q_{2\rho, \lambda}(z_0)}H(z, |Du|)^{\theta}\, dz\right)^{\frac{1}{\theta}} + \miint{Q_{2\rho, \lambda}(z_0)}H(z, |F|)\, dz\right].    
\end{align*}
Finally, using the fourth condition in \eqref{eq: p-phase condition}, we obtain the desired estimate.
\end{proof}

\subsection{Reverse H\"{o}lder inequality for $q$-phase} In this subsection, we prove reverse H\"{o}lder inequality for $q$-phase. In the $q$-intrinsic case, we consider the following assumptions:\begin{equation}\label{eq: q-phase condition}
\left\{
\begin{aligned}
&Ka(z_0)\lambda^p \leq b(z_0)\lambda^q,\\
& b(z_0)\geq \frac{\delta}{2} \quad \text{and}\quad \frac{b(z)}{3}\leq b(z_0)\leq 3b(z) \ \, \text{for all}\,\,\, z\in J_{\rho, \la}(z_0),\\
&\miint{J_{\tau,\lambda}(z_0)} \left[H(z,|Du|)+H(z,|F|)\right] dz <b(z_0)\lambda^q \ \, \text{for every }\tau\in(\rho,c_v(2\rho)],\\
&\miint{J_{\rho,\lambda}(z_0)} \left[H(z,|Du|)+H(z,|F|)\right] dz =b(z_0)\lambda^q.
\end{aligned}
\right.
\end{equation}
We note that, from the second and third conditions above, we have
\begin{align*}
\miint{J_{\tau, \lambda}(z_0)}b(z_0)\left(|Du|+|F|\right)^q\, dz\leq \miint{J_{\tau, \lambda}(z_0)}3 b(z)\left(|Du|+|F|\right)^q\, dz< c\, b(z_0)\lambda^q.
\end{align*}
This implies that
\begin{align}\label{EQQ3.5}
\miint{J_{\tau, \lambda}(z_0)}\left(|Du|+|F|\right)^q\, dz \leq c\lambda^q.    
\end{align}
\begin{lemma} \label{lem:q-phase1}
Let $u$ be a weak solution to \eqref{eq: main equation}. Then there exists a constant $c=c(\textnormal{\texttt{data}}) \geq 1$ such that for any $J_{4\rho,\lambda}(z_0)\subset \Omega_T$ with \eqref{eq: q-phase condition}, $\tau \in [2\rho,4\rho]$ and $\theta \in (q-1/p,1]$, 
\begin{align*}
&\miint{J_{\tau,\lambda}(z_0)} \left[ \widetilde{H}(z,|Du|)+\widetilde{H}(z,|F|)\right] d z \\
&\quad \leq c \lambda^{-1+\frac{q}{p}} \miint{J_{\tau,\lambda}(z_0)}a(z)^{\frac{p-1}{p}}(|D u|+|F|)^{p-1}\; dz +c  \miint{J_{\tau,\lambda}(z_0)} b(z)(|Du|+|F|)^{q-1}\; dz \\
&\qquad +c\omega_a(\tau) \miint{J_{\tau,\lambda}(z_0)} (|Du|+|F|)^{p-1} \; dz.
\end{align*}
\end{lemma}
\begin{proof}
First, using the uniform continuity of $a(\cdot)$ and H\"{o}lder continuity of $b(\cdot),$ we observe that
\begin{align*}
&\miint{J_{\tau,\lambda}(z_0)} \left[ \widetilde{H}(z,|Du|)+\widetilde{H}(z,|F|)\right] dz\\
&\leq c\miint{J_{\tau,\lambda}(z_0)} \inf_{\tilde{z}\in J_{\tau,\lambda}(z_0)}a(\tilde{z}) \, (|Du|+|F|)^{p-1} \; dz+c \omega_a(\tau) \miint{J_{\tau,\lambda}(z_0)} \left(|Du|^{p-1}+|F|^{p-1}\right) dz\\
&\quad +c\miint{J_{\tau,\lambda}(z_0)}\inf_{\tilde{z}\in J_{\tau,\lambda}(z_0)}b(\tilde{z}) \, (|Du|+|F|)^{q-1}\;dz+c\tau^\alpha \miint{J_{\tau,\lambda}(z_0)} \left(|Du|^{q-1}+|F|^{q-1}\right) dz\\
& \leq c\miint{Q_{\tau,\lambda}(z_0)} \inf_{\tilde{z}\in J_{\tau,\lambda}(z_0)}a(\tilde{z}) \, (|Du|+|F|)^{p-1} \; dz+c \omega_a(\tau) \miint{J_{\tau,\lambda}(z_0)} \left(|Du|^{p-1}+|F|^{p-1}\right) dz\\
&\quad +c\miint{J_{\tau,\lambda}(z_0)}\inf_{\tilde{z}\in J_{\tau,\lambda}(z_0)}b(\tilde{z}) \, (|Du|+|F|)^{q-1}\;dz+c(\delta) \tau^\alpha \miint{J_{\tau,\lambda}(z_0)} \left[\widetilde{H}(z,|Du|)+\widetilde{H}(z,|F|) \right] dz
\end{align*}
for some constant $c \geq 1$ depending on $p$, $q$, $[b]_\alpha$ and $\delta$.
We consider $\tau>0$ so small that $\displaystyle c (\delta) \tau^\alpha \leq \frac{1}{2}$. Thus, we obtain
\begin{align*}
&\miint{J_{\tau,\lambda}(z_0)} \left[\widetilde{H}(z,|Du|)+\widetilde{H}(z,|F|) \right]dz\\
& \leq c\miint{J_{\tau,\lambda}(z_0)} \inf_{\tilde{z}\in J_{\tau,\lambda}(z_0)}a(\tilde{z}) \, (|Du|+|F|)^{p-1} \; dz+c\miint{J_{\tau,\lambda}(z_0)}\inf_{\tilde{z}\in J_{\tau,\lambda}(z_0)}b(\tilde{z}) \, (|Du|+|F|)^{q-1}\;dz\\
&\quad +c\omega_a(\tau) \miint{J_{\tau,\lambda}(z_0)} \left(|Du|^{p-1}+|F|^{p-1} \right) dz.
\end{align*}
Using the first condition of \eqref{eq: q-phase condition}, we get
\begin{align*}
   &\miint{J_{\tau,\lambda}(z_0)} \left[\widetilde{H}(z,|Du|)+\widetilde{H}(z,|F|) \right]dz\\ 
   &\leq K^{-\frac{1}{p}}b(z_0)^{\frac{1}{p}}\lambda^{-1+\frac{q}{p}}\miint{J_{\tau,\lambda}(z_0)} a(z)^{\frac{p-1}{p}} \, (|Du|+|F|)^{p-1} \; dz+c\miint{J_{\tau,\lambda}(z_0)}b(z) \, (|Du|+|F|)^{q-1}\;dz\\
   &\quad +c\omega_a(\tau) \miint{J_{\tau,\lambda}(z_0)} \left(|Du|+|F|\right)^{p-1} dz.
\end{align*}
This completes the proof.
\end{proof}
 \begin{lemma} \label{lem:q-term of parabolic poincare inequality on q-intrinsic cylinders_1}
    Let $u$ be a weak solution to \eqref{eq: main equation}. Then for $J_{4\rho,\lambda}(z_0)\subset \Omega_T$ with \eqref{eq: q-phase condition}, $\tau \in [2\rho,4\rho]$ and $\theta \in (\frac{q-1}{p},1]$, there exists a constant $c=c(\textnormal{\texttt{data}}) \geq 1$ such that 
    \begin{align*}
        \miint{J_{\tau,\lambda}(z_0)}\frac{|u-(u)^q_{z_0;\tau,\lambda}|^{\theta q}}{\tau^{\theta q}}\, dz \leq c \miint{J_{\tau,\lambda}(z_0)}H(z,|Du|)^{\theta}\, dz
         +c\left(\miint{J_{\tau,\lambda}(z_0)}H(z,|F|) \, d z\right)^{\theta}.
    \end{align*} 
\end{lemma}
\begin{proof}
From \cref{eq : parabolic poincare inequality in Section 3} of \cref{lem : parabolic poincare inequality in Section 3}, we have
\begin{align*}
\miint{J_{\tau,\lambda}(z_0)}\frac{|u-(u)^q_{z_0;\tau,\lambda}|^{\theta q}}{\tau^{\theta q}}\; dz &\leq c\miint{J_{\tau,\lambda}(z_0)} |Du|^{\theta q} \; dz + c\underbrace{\left(\lambda^{2-q}\miint{J_{\tau,\lambda}(z_0)}b(z) (|Du|+|F|)^{q-1}\;dz \right)^{\theta q}}_{\mathrm{I}}\\
&\quad + c\underbrace{\left(\lambda^{1-q+\frac{q}{p}}\miint{J_{\tau,\lambda}(z_0)} a(z)^{\frac{p-1}{p}}(|Du|+|F|)^{p-1}\; dz \right)^{\theta q}}_{\mathrm{II}}\\
&\quad + c\underbrace{\lambda^{(2-q)\theta q}\left(\miint{J_{\tau,\lambda}(z_0)} (|Du|+|F|)^{p-1}\; dz \right)^{\theta q}}_{\mathrm{III}}.
\end{align*}
We estimate $\mathrm{I}, \mathrm{II}, \mathrm{III}$ as follows.

\noindent\textbf{Estimate of $\mathrm{I}$:} Using the second condition of \eqref{eq: q-phase condition}, that is $\frac{b(z)}{3}\leq b(z_0)\leq 3 b(z),$ and H\"{o}lder's inequality, we get
\begin{align*}
&\left(\lambda^{2-q}\miint{J_{\tau,\lambda}(z_0)}b(z) (|Du|+|F|)^{q-1}\;dz \right)^{\theta q}\\
&\leq c \lambda^{(2-q)\theta q} b(z_0)^{\theta} \left(\miint{J_{\tau, \lambda}(z_0)}b(z)^{\frac{q-1}{q}}\left(|Du|+|F|\right)^{q-1}\right)^{\theta q}\\
&\leq c\lambda^{(2-q)\theta q}b(z_0)^{\theta}\left(\miint{J_{\tau, \lambda}(z_0)}b(z)^{\theta}(|Du|+|F|)^{\theta q}\, dz\right)^{q-1}\\
&\leq c\lambda^{(2-q)\theta q}b(z_0)^{\theta}b(z_0)^{\theta (q-2)}\lambda^{\theta q(q-2)} \miint{J_{\tau, \lambda}(z_0)}b(z)^{\theta}(|Du|+|F|)^{\theta q}\, dz.
\end{align*}

\noindent\textbf{Estimate of $\mathrm{II}$:} Since $q(p-1)-p\geq 0,$ using H\"{o}lder's inequality and the third condition of \eqref{eq: q-phase condition}, we get
\begin{align*}
&\left(\lambda^{1-q+\frac{q}{p}}\miint{J_{\tau,\lambda}(z_0)} a(z)^{\frac{p-1}{p}}(|Du|+|F|)^{p-1}\; dz \right)^{\theta q}\\
&\leq c\lambda^{(1-q+\frac{q}{p})\theta q}\left(\miint{J_{\tau,\lambda}(z_0)} a(z)^\theta(|Du|+|F|)^{\theta p}\; dz \right)^{\frac{q(p-1)}{p}}\\
&=c\lambda^{(1-q+\frac{q}{p})\theta q}\left(\miint{J_{\tau,\lambda}(z_0)} a(z)^\theta(|Du|+|F|)^{\theta p}\; dz \right)^{\frac{q(p-1)-p}{p}} \miint{J_{\tau,\lambda}(z_0)} a(z)^\theta(|Du|+|F|)^{\theta p}\; dz \\
&\overset{\eqref{eq: q-phase condition}}{\leq} c \lambda^{\left(\frac{q(1-p)}{p}+\frac{q(p-1)}{p}\right)\theta q} \miint{J_{\tau,\lambda}(z_0)} a(z)^\theta(|Du|+|F|)^{\theta p}\; dz=c\miint{J_{\tau,\lambda}(z_0)} a(z)^\theta(|Du|+|F|)^{\theta p}\; dz.
\end{align*}

\noindent\textbf{Estimate of $\mathrm{III}$:} Here, we use the second and third conditions of \eqref{eq: q-phase condition}. Indeed, 
\begin{align*}
\lambda^{(2-q)\theta q}\left(\miint{J_{\tau,\lambda}(z_0)} (|Du|+|F|)^{p-1}\; dz \right)^{\theta q}&\leq \lambda^{(2-q)\theta q}\left(\miint{J_{\tau,\lambda}(z_0)} (|Du|+|F|)^{\theta q}\; dz \right)^{p-1}\\
&\overset{\eqref{eq: q-phase condition}}{\leq}\frac{\lambda^{(p-q)\theta q}}{b(z_0)^{\theta}} \miint{J_{\tau, \lambda}(z_0)}b(z_0)^{\theta}(|Du|+|F|)^{\theta q}\, dz\\
&\overset{\eqref{eq: q-phase condition}}{\leq} c\left(\frac{2}{\delta}\right)^\theta \miint{J_{\tau, \lambda}(z_0)}b(z)^{\theta}(|Du|+|F|)^{\theta q}\, dz.
\end{align*}
Finally, using H\"{o}lder's inequality and combining the above estimates, we obtain the desired estimate in the statement. This completes the proof.
\end{proof}
\begin{lemma} \label{lem: q-term of parabolic poincare inequality on q-intrinsic cylinders_2}
Let $u$ be a weak solution to \eqref{eq: main equation}. Then there exists a constant $c=c(\textnormal{\texttt{data}}) \geq 1$ such that for any $J_{4\rho,\lambda}(z_0)\subset \Omega_T$ with \eqref{eq: q-phase condition}, $\tau \in [2\rho,4\rho]$ and $\theta \in ((q-1)/p,1]$, we have
\begin{align*}
&\miint{J_{\tau,\lambda}(z_0)} \inf_{\tilde{z}\in J_{\tau,\lambda}(z_0)}a(\tilde{z})^\theta\frac{|u-(u)^q_{z_0;\tau,\lambda}|^{\theta p}}{\tau^{\theta p}}\; dz\\
&\quad \leq c \miint{J_{\tau,\lambda}(z_0)}(H(z,|Du|))^{\theta}\; dz +c\left(\miint{J_{\tau,\lambda}(z_0)}H(z,|F|)\; d z\right)^{\theta}.
\end{align*}
\begin{proof}
 From \cref{eq : parabolic poincare inequality in Section 3} of \cref{lem : parabolic poincare inequality in Section 3} and \cref{lem:q-phase1}, we find
\begin{align*}
&\miint{J_{\tau,\lambda}(z_0)}\inf_{\tilde{z}\in J_{\tau, \lambda}(z_0)}a(\tilde{z})^{\theta}\frac{|u-(u)^q_{z_0;\tau,\lambda}|^{\theta p}}{\tau^{\theta p}}\; dz\\
&\leq c\miint{J_{\tau,\lambda}(z_0)} \inf_{\tilde{z}\in J_{\tau, \lambda}(z_0)}a(\tilde{z})^{\theta}|Du|^{\theta p} \; dz + c\underbrace{\inf_{\tilde{z}\in J_{\tau, \lambda}(z_0)}a(\tilde{z})^{\theta}\left(\lambda^{2-q}\miint{J_{\tau,\lambda}(z_0)}b(z) (|Du|+|F|)^{q-1}\;dz \right)^{\theta p}}_{\mathrm{I}}\\
&\quad + c\underbrace{\inf_{\tilde{z}\in J_{\tau, \lambda}(z_0)}a(\tilde{z})^{\theta}\left(\lambda^{1-q+\frac{q}{p}}\miint{J_{\tau,\lambda}(z_0)} a(z)^{\frac{p-1}{p}}(|Du|+|F|)^{p-1}\; dz \right)^{\theta p}}_{\mathrm{II}}\\
&\quad + c\underbrace{\inf_{\tilde{z}\in J_{\tau, \lambda}(z_0)}a(\tilde{z})^{\theta}\lambda^{(2-q)\theta p}\left(\miint{J_{\tau,\lambda}(z_0)} (|Du|+|F|)^{p-1}\; dz \right)^{\theta p}}_{\mathrm{III}}.
\end{align*} 

\noindent \textbf{Estimate of $\mathrm{I}$:}   We use the first and second conditions of  \eqref{eq: q-phase condition}, \eqref{EQQ3.5} and $b(z)\leq ||b||_{L^{\infty}(\Omega_T)}$ to estimate
\begin{align*}
 &\inf_{\tilde{z}\in J_{\tau, \lambda}(z_0)}a(\tilde{z})^{\theta}\left(\lambda^{2-q}\miint{J_{\tau,\lambda}(z_0)}b(z) (|Du|+|F|)^{q-1}\;dz \right)^{\theta p}\\
 &\overset{\eqref{eq: q-phase condition}}{\leq} K^{-\theta}b(z_0)^\theta \lambda^{(q-p)\theta} \lambda^{(2-q)\theta p}||b||_{L^{\infty}}\left(\miint{J_{\tau, \lambda}(z_0)}\left(|Du|+|F|\right)^{\theta q}\,dz\right)^{\frac{(q-1)p}{q}}\\
 &=c(K, ||b||_{L^{\infty}}, \theta)\lambda^{(q-p+2p-pq)\theta}\left(\miint{J_{\tau, \lambda}(z_0)}\left(|Du|+|F|\right)^{\theta q}\,dz\right)^{\frac{(q-1)p-q}{q}} \miint{J_{\tau, \lambda}(z_0)}\left(|Du|+|F|\right)^{\theta q}\,dz\\
 &\overset{\eqref{EQQ3.5}}{\leq} c(K, ||b||_{L^{\infty}}, \theta)\lambda^{(q-p+2p-pq+(q-1)p-q)\theta} \miint{J_{\tau, \lambda}(z_0)}\left(|Du|+|F|\right)^{\theta q}\,dz\\
 &=c(K, ||b||_{L^{\infty}}, \theta) \miint{J_{\tau, \lambda}(z_0)}\left(|Du|+|F|\right)^{\theta q}\,dz\\
 &\overset{\eqref{eq: q-phase condition}}{\leq}\frac{c(K, ||b||_{L^{\infty}}, \theta)}{b(z_0)^{\theta}} \miint{J_{\tau, \lambda}(z_0)}b(z)^{\theta}\left(|Du|+|F|\right)^{\theta q}\,dz\\
 &\leq c(K, ||b||_{L^{\infty}}, \theta, \delta)\miint{J_{\tau, \lambda}(z_0)}(H(z, |Du|)+H(z,|F|))^{\theta}\, dz.
\end{align*}
Here we note that, since $2\leq p\leq q,$
\begin{align*}
   \frac{p+q}{pq} = \frac{1}{p}+\frac{1}{q} \leq 1 \implies p+q\leq pq \implies q\leq p(q-1).
\end{align*}
\textbf{Estimate of $\mathrm{II}$:} We use the first and third conditions of \eqref{eq: q-phase condition} and H\"{o}lder's inequality to deduce
\begin{align*}
&\inf_{\tilde{z}\in J_{\tau, \lambda}(z_0)}a(\tilde{z})^{\theta}\left(\lambda^{1-q+\frac{q}{p}}\miint{J_{\tau,\lambda}(z_0)} a(z)^{\frac{p-1}{p}}(|Du|+|F|)^{p-1}\; dz \right)^{\theta p}\\
&\leq K^{-\theta}b(z_0)^{\theta}\lambda^{(q-p)\theta}\lambda^{\left(1-q+\frac{q}{p}\right)\theta p}\left(\miint{J_{\tau, \lambda}(z_0)}a(z)^{\theta}(|Du|+|F|)^{\theta p}\,dz\right)^{p-1}\\
&=c\lambda^{(q-p)\theta}\lambda^{\left(1-q+\frac{q}{p}\right)\theta p}\left(\miint{J_{\tau, \lambda}(z_0)}a(z)^{\theta}(|Du|+|F|)^{\theta p}\,dz\right)^{p-2} \miint{J_{\tau, \lambda}(z_0)}a(z)^{\theta}(|Du|+|F|)^{\theta p}\,dz\\
&\overset{\eqref{eq: q-phase condition}}{\leq} c(K, \theta, ||b||_{L^{\infty}})\lambda^{(q-p+p-pq+q+pq-2q)\theta} \miint{J_{\tau, \lambda}(z_0)}a(z)^{\theta}(|Du|+|F|)^{\theta p}\,dz\\
&\leq c(K, \theta, ||b||_{L^{\infty}}) \miint{J_{\tau, \lambda}}(H(z, |Du|)+H(z, |F|))^{\theta}\, dz\\
&\leq c \left[ \miint{J_{\tau, \lambda}(z_0)}H(z, |Du|)^{\theta}\,dz+\left(\miint{J_{\tau, \lambda}(z_0)}H(z, |F|)\,dz\right)^{\theta}\right].
\end{align*}

\noindent\textbf{Estimate of $\mathrm{III}$:} Similarly, using H\"{o}lder's inequality and \eqref{EQQ3.5}, we can estimate
\begin{align*}
&\inf_{\tilde{z}\in J_{\tau, \lambda}(z_0)}a(\tilde{z})^{\theta}\lambda^{(2-q)\theta p}\left(\miint{J_{\tau,\lambda}(z_0)} \omega_a(\tau) (|Du|+|F|)^{p-1}\; dz \right)^{\theta p}\\
&\leq c\lambda^{(2-q)\theta p}\inf_{\tilde{z}\in J_{\tau, \lambda}(z_0)}a(\tilde{z})^{\theta}\left(\miint{J_{\tau,\lambda}(z_0)} (|Du|+|F|)^{\theta p}\; dz\right)^{p-1}\\
&\leq c\lambda^{(2-q)\theta p}\left(\miint{J_{\tau,\lambda}(z_0)} (|Du|+|F|)^{ q}\; dz\right)^{\frac{(p-2)\theta p}{q}} \miint{J_{\tau,\lambda}(z_0)} \inf_{\tilde{z}\in J_{\tau, \lambda}(z_0)}a(\tilde{z})^{\theta}(|Du|+|F|)^{\theta p}\; dz\\
&\overset{\eqref{EQQ3.5}}{\leq} c\lambda^{(2-q)\theta p}\lambda^{(p-2)\theta p} \miint{J_{\tau,\lambda}(z_0)} a(z)^{\theta}(|Du|+|F|)^{\theta p}\; dz\\
&\leq c\lambda^{(p-q)\theta p}\left[ \miint{J_{\tau, \lambda}(z_0)}H(z, |Du|)^{\theta}\,dz+\left(\miint{J_{\tau, \lambda}(z_0)}H(z, |F|)\,dz\right)^{\theta}\right].
\end{align*}
Combining the three estimates above, we conclude the proof of this lemma.
\end{proof}
\end{lemma}
\begin{lemma}\label{lem : Lemma 3.10}
    Let $u$ be a weak solution to \eqref{eq: main equation}. Then for $J_{c_v(2 \rho),\lambda}(z_0) \subset \Omega_T$ with \eqref{eq: q-phase condition}, there exists a constant $c=c(\textnormal{\texttt{data}}) \geq 1$ such that
    $$
    S(u)^q_{z_0 ; 2 \rho,\lambda}=\sup _{I^q_{2 \rho,\lambda}(t_0)} \dashint_{B_{2 \rho}(x_0)} \frac{\left|u-(u)^q_{z_0 ; 2 \rho,\lambda}\right|^2}{(2 \rho)^2}\; dx \leq c \lambda^2 .
    $$
\end{lemma}
\begin{proof}
Let $2\rho \leq \rho_1<\rho_2 \leq 4\rho$. By \cref{lem : the Caccioppoli inequality}, there exists a constant $c$ depending on $n,p,q,\nu$ and $L$ such that 
\begin{align}
\lambda^{q-2}S(u)^q_{z_0;\rho_1,\lambda}&\leq \frac{c\rho_2^q}{(\rho_2-\rho_1)^q}\miint{J_{\rho_2,\lambda}(z_0)}\left( a(z)\frac{|u-(u)^q_{z_0;\rho_2,\lambda}|^p}{\rho^p_2}+b(z)\frac{|u-(u)^q_{z_0;\rho_2,\lambda}|^q}{\rho^q_2} \right) dz \nonumber\\
&\quad + \frac{c\rho_2^2 \lambda^{q-2}}{(\rho_2-\rho_1)^2}\miint{J_{\rho_2,\lambda}(z_0)}\frac{|u-(u)^q_{z_0;\rho_2,\lambda}|^2}{\rho_2^2}\; dz +c\miint{J_{\rho_2,\lambda}(z_0)} H(z,|F|)\; dz.\label{eq : second equation in lemma 5.1}
\end{align}
Firstly, using \cref{lem:q-term of parabolic poincare inequality on q-intrinsic cylinders_1} with $\theta=1$ and $b(z)\leq ||b||_{L^{\infty}}$, we get
\begin{align*}
\miint{J_{\rho_2,\lambda}(z_0)}b(z)\frac{|u-(u)_{z_0;\rho_2,\lambda}|^q}{\rho^q_2} \, dz \leq c b(z_0) \miint{J_{\tau, \lambda}(z_0)} [H(z,|Du|)+H(z, |F|)]\,dz
\overset{\eqref{eq: q-phase condition}}{\leq} c||b||^2_{L^{\infty}(\Omega_T)}\lambda^q.   
\end{align*}
Next, we estimate
\begin{align*} 
\miint{J_{\tau, \lambda}(z_0)}a(z)\frac{|u-(u)_{z_0, \rho_2, \lambda}|^p}{\rho^p_2}\,dz&\leq \miint{J_{\tau, \lambda}(z_0)}\inf_{\tilde{z}\in J_{\tau, \lambda}(z_0)}a(\tilde{z})\frac{|u-(u)_{z_0, \rho_2, \lambda}|^p}{\rho^p_2}\,dz\\
&\quad + \omega_{a}(\tau)\miint{J_{\tau, \lambda}(z_0)}\frac{|u-(u)_{z_0, \rho_2, \lambda}|^p}{\rho^p_2}\,dz.
\end{align*}
Now the first term to the above inequality can be estimated by \cref{lem: q-term of parabolic poincare inequality on q-intrinsic cylinders_2} with $\theta =1.$ Indeed, using the third condition of \eqref{eq: q-phase condition}, we get
\begin{align*}
 \miint{J_{\tau, \lambda}(z_0)}\inf_{\tilde{z}\in J_{\tau, \lambda}(z_0)}a(\tilde{z})\frac{|u-(u)_{z_0, \rho_2, \lambda}|^p}{\rho^p_2}\,dz \leq b(z_0)\lambda^q.  
\end{align*}
Now we estimate the second term. Again, from \cref{eq : parabolic poincare inequality in Section 3} and \cref{lem:q-phase1}, we have
\begin{align}\label{EQQQ3.9}
        &\miint{J_{\tau,\lambda}(z_0)}\omega_a(\tau)\frac{|u-(u)_{z_0;\tau,\lambda}|^{p}}{\tau^{p}}\; dz\\
        &\quad \leq c\underbrace{\miint{J_{\tau,\lambda}(z_0)} \omega_a(\tau)|Du|^{p} \; dz}_{\mathrm{I}} + c\omega_a(\tau)\underbrace{\left(\lambda^{2-q}\miint{J_{\tau,\lambda}(z_0)}b(z) (|Du|+|F|)^{q-1}\;dz \right)^{p}}_{\mathrm{II}}\nonumber\\
        &\quad\qquad + c\omega_a(\tau)\underbrace{\left(\lambda^{1-q+\frac{q}{p}}\miint{J_{\tau,\lambda}(z_0)} a(z)^{\frac{p-1}{p}}(|Du|+|F|)^{p-1}\; dz \right)^{p}}_{\mathrm{III}}\nonumber\\
        &\quad\qquad + c\omega_a(\tau)\underbrace{\lambda^{(2-q) p}\left(\miint{J_{\tau,\lambda}(z_0)} (|Du|+|F|)^{p-1}\; dz \right)^{p}}_{\mathrm{IV}}\nonumber.
    \end{align} 
\textbf{Estimate of $\mathrm{I}$:} We simply use H\"{o}lder's inequality and \eqref{EQQ3.5}:
\begin{align*}
\miint{J_{\tau, \lambda}(z_0)}\omega_{a}(\tau)|Du|^p\,dz\leq c \left(\miint{J_{\tau, \lambda}(z_0)}|Du|^q\,dz\right)^{\frac{p}{q}}\overset{\eqref{EQQ3.5}}{\leq} c\lambda^p \leq c\lambda^q.
\end{align*}
\textbf{Estimate of $\mathrm{II}$:} We use H\"{o}lder's inequality, \eqref{EQQ3.5} and $b(z)\leq ||b||_{L^{\infty}}$ to deduce
\begin{align*}
 &\left(\lambda^{2-q}\miint{J_{\tau,\lambda}(z_0)}b(z) (|Du|+|F|)^{q-1}\;dz \right)^{p}\\
 &\quad \leq  \lambda^{(2-q)p}||b||_{L^{\infty}}^p\left(\miint{J_{\tau, \lambda}(z_0)}\left(|Du|+|F|\right)^{q}\,dz\right)^{\frac{(q-1)p}{q}} \overset{\eqref{EQQ3.5}}{\leq} c\lambda^p\leq c\lambda^q.
\end{align*}
\textbf{Estimate of $\mathrm{III}$:} Using H\"{o}lder's inequality and third condition of \eqref{eq: q-phase condition}, we get
\begin{align*}
&\left(\lambda^{1-q+\frac{q}{p}}\miint{J_{\tau,\lambda}(z_0)} a(z)^{\frac{p-1}{p}}(|Du|+|F|)^{p-1}\; dz \right)^{p}\\
&\leq \lambda^{\left(1-q+\frac{q}{p}\right)p}\left(\miint{J_{\tau, \lambda}(z_0)}a(z)(|Du|+|F|)^{ p}\,dz\right)^{p-1} \leq c\lambda^{\left(1-q+\frac{q}{p}\right)p}\lambda^{q(p-1)}=c\lambda^p\leq c\lambda^q.
\end{align*}
\textbf{Estimate of $\mathrm{IV}$:} Again, we simply use H\"{o}lder's inequality and \eqref{EQQ3.5} to obtain
\begin{align*}
    &\lambda^{(2-q) p}\left(\miint{J_{\tau,\lambda}(z_0)} (|Du|+|F|)^{p-1}\; dz \right)^{ p}\\
    &\leq \lambda^{(2-q)p}\left(\miint{J_{\tau,\lambda}(z_0)} (|Du|+|F|)^{q}\; dz\right)^{\frac{(p-1)p}{q}}\overset{\eqref{EQQ3.5}}{\leq} c\lambda^{2p-pq+p^2-p}\leq c\lambda^{q}.
\end{align*}
In the last inequality above,  we used
\begin{align*}
    0\geq (p+1)(p-q)=p^2-pq+p-q=p^2-pq+2p-p-q\implies 2p-pq+p^2-p\leq q.
\end{align*}
Now we estimate the second term of \eqref{eq : second equation in lemma 5.1}. Using \cref{lem : lemma 2.1} with $\sigma =2, s=q, r=2, \vartheta=\frac{1}{2},$ \cref{lem:q-term of parabolic poincare inequality on q-intrinsic cylinders_1} and \cref{eq: q-phase condition}, we estimate the following
\begin{align*}
&\miint{J_{\rho_2,\lambda}(z_0)}\frac{|u-(u)^q_{z_0;\rho_2,\lambda}|^2}{\rho_2^2}\; dz \\
&\leq c\dashint_{I_{\rho_2}(t_0)}\left(\dashint_{B_{\rho_2}(x_0)}\left(\frac{|u-(u)^q_{z_0;\rho_2,\lambda}|^q}{\rho_2^q}+|Du|^q\right)\,dx\right)^{\frac{1}{q}}dt \, \left(S(u)^q_{z_0, \rho_2, \lambda}\right)^{\frac{1}{2}}\\
&\leq c (\delta)\lambda \left(S(u)^q_{z_0, \rho_2, \lambda}\right)^{\frac{1}{2}}.
\end{align*}
Note that, in the last step of the above estimate we also use the second condition of \eqref{eq: q-phase condition}, specifically, $b(z_0)\geq \frac{\delta}{2}$ and $\frac{b(z)}{3}\leq b(z_0)\leq 3b(z)$. Since the third condition of \eqref{eq: q-phase condition} implies that
\begin{align*}
    \miint{J_{\rho_2, \lambda}(z_0)}H(z, |F|)\, dz \apprle \lambda^q,
\end{align*}
combining all the estimates above, we have
\begin{align*}
    S(u)^q_{z_0, \rho_1, \lambda}\leq c\frac{\rho^q_2}{(\rho_2-\rho_1)^q}\lambda^2+c\frac{\rho^2_2}{(\rho_2-\rho_1)^2}\lambda \left(S(u)^q_{z_0, \rho_2, \lambda}\right)^{\frac{1}{2}}.
\end{align*}
Now applying Young's inequality in the second term on the right-hand side above, we obtain
\begin{align*}
S(u)^q_{z_0, \rho_1, \lambda}\leq \frac{1}{2} S(u)^q_{z_0, \rho_2, \lambda}+c\left[\frac{\rho^q_2}{(\rho_2-\rho_1)^q}+\frac{\rho^4_2}{(\rho_2-\rho_1)^4}\right] \lambda^2.
\end{align*}
Now we use iteration  lemma, \cref{iter_lemma} to complete the proof.
\end{proof}
\begin{remark}
It is interesting to note that the term in \eqref{EQQQ3.9} admits a sharper estimate. Indeed, we have
\begin{align*}
\miint{J_{\tau,\lambda}(z_0)}\omega_a(\tau)\frac{|u-(u)_{z_0;\tau,\lambda}|^{p}}{\tau^{p}}\; dz\leq c(\delta)\omega_a(\tau)\lambda^p.    
\end{align*}
This shows that an improved bound can be derived. Regarding the estimate of the fourth term (see Estimate of $\mathrm{IV}$ above), we observe that the condition $p\leq q$ implies $p(1-q+p)\leq p,$ which subsequently yields the bound $c(\delta)\lambda^p$.
\end{remark}
\begin{lemma}\label{Lem : lemma 3.11}
    Let $u$ be a weak solution to \eqref{eq: main equation}. Then for $J_{c_v(2 \rho),\lambda}\left(z_0\right) \subset \Omega_T$ with \eqref{eq: q-phase condition}, there exist constants $c=c(\textnormal{\texttt{data}}) \geq 1$ and $\theta_0=\theta_0(n, p, q, s) \in(0,1)$ such that for any $\theta \in\left(\theta_0, 1\right)$, we have
    $$
    \begin{aligned}
    &\miint{J_{2 \rho,\lambda}\left(z_0\right)} H\left(z,\frac{\left|u-(u)_{z_0 ; 2 \rho,\lambda}\right|}{2 \rho}\right) dz \\
    &\leq c \miint{J_{2 \rho,\lambda}\left(z_0\right)}\left(\inf_{\tilde{z} \in J_{2\rho, \lambda}(z_0)}a(\tilde{z})^{\theta}\frac{\left|u-(u)_{z_0 ; 2 \rho,\lambda}\right|^{\theta p}}{(2 \rho)^{\theta p}}+\inf_{J_{2\rho, \lambda}(z_0)}a(\tilde{z})^{\theta}|D u|^{\theta p}\right) dz\;\\
    &\qquad \times \lambda^{(q-p)(1-\theta)}\left[S(u)_{z_0 ; 2 \rho,\lambda}\right]^{\frac{(1-\theta) p}{2}} \\
    &\quad +c \miint{J_{2 \rho,\lambda}\left(z_0\right)}\left( \frac{\left|u-(u)_{z_0 ; 2 \rho,\lambda}\right|^{\theta q}}{(2 \rho)^{\theta q}}+|D u|^{\theta q}\right) dz \;\left(S(u)^q_{z_0 ; 2 \rho,\lambda}\right)^{\frac{(1-\theta) q}{2}}\\
    &\quad + \omega_a(2\rho)\miint{J_{2\rho, \lambda}(z_0)}\frac{\left|u-(u)_{z_0 ; 2 \rho,\lambda}\right|^p}{(2 \rho)^p}\,dz.
    \end{aligned}
    $$
\end{lemma}
\begin{proof} Using the second condition of \eqref{eq: q-phase condition} and uniform continuity of $a(\cdot)$, we have
\begin{align*}
&\miint{J_{2 \rho,\lambda}\left(z_0\right)} H\left(z,\frac{\left|u-(u)_{z_0 ; 2 \rho,\lambda}\right|}{2 \rho}\right) d z \leq c\miint{J_{2 \rho,\lambda}\left(z_0\right)} a(z)\frac{|u-(u)_{z_0, 2\rho, \lambda}|^p}{(2\rho)^p}+ b(z_0) \frac{|u-(u)_{z_0, 2\rho, \lambda}|^q}{(2\rho)^q}\, dz\\
&\leq c\underbrace{\miint{J_{2 \rho,\lambda}\left(z_0\right)} \inf_{\tilde{z}\in J_{2\rho, \lambda}(z_0)}a(\tilde{z})\frac{|u-(u)_{z_0, 2\rho, \lambda}|^p}{(2\rho)^p}\,dz}_{\mathrm{I}}+c\underbrace{\miint{J_{2\rho, \lambda}(z_0)}b(z_0) \frac{|u-(u)_{z_0, 2\rho, \lambda}|^q}{(2\rho)^q}\, dz}_{\mathrm{II}}\\
&\quad +c\omega_a(2\rho)\miint{J_{2 \rho,\lambda}\left(z_0\right)} \frac{|u-(u)_{z_0, 2\rho, \lambda}|^p}{(2\rho)^p}\,dz.
\end{align*}
We estimate the terms $\mathrm{I}$ and $\mathrm{II}$ below.

\noindent\textbf{Estimate of $\mathrm{I}$:} We use \cref{lem : lemma 2.1} with $\sigma=p, s=\theta p, \vartheta=\theta, r=2$ to obtain
\begin{align*}
&\miint{J_{2 \rho,\lambda}\left(z_0\right)} \inf_{\tilde{z}\in J_{2\rho, \lambda}(z_0)}a(\tilde{z})\frac{|u-(u)_{z_0, 2\rho, \lambda}|^p}{(2\rho)^p}\,dz\\
&\leq c \left(\miint{J_{2\rho, \lambda}(z_0)}\inf_{\tilde{z}\in J_{2\rho, \lambda}(z_0)} a(\tilde{z})^{\theta}\frac{|u-(u)_{z_0, 2\rho, \lambda}|^{\theta p}}{(2\rho)^{\theta p}}+\inf_{\tilde{z}\in J_{2\rho, \lambda}(z_0)} a(\tilde{z})^{\theta}|Du|^{\theta p}\, dz\right)\\
&\qquad \times \inf_{\tilde{z}\in J_{2\rho, \lambda}(z_0)} a(\tilde{z})^{1-\theta}[S(u)_{z_0, 2\rho, \lambda}]^{\frac{(1-\theta)p}{2}}\\
&\leq c\left(\miint{J_{2\rho, \lambda}(z_0)}\inf_{\tilde{z}\in J_{2\rho, \lambda}(z_0)} a(\tilde{z})^{\theta}\frac{|u-(u)_{z_0, 2\rho, \lambda}|^{\theta p}}{(2\rho)^{\theta p}}+\inf_{\tilde{z}\in J_{2\rho, \lambda}(z_0)} a(\tilde{z})^{\theta}|Du|^{\theta p}\, dz\right)\\
&\qquad \times \lambda^{(q-p)(1-\theta)}\left(S(u)^q_{z_0, 2\rho, \lambda}\right)^{\frac{(1-\theta)p}{2}}.
\end{align*}
\textbf{Estimate of $\mathrm{II}$:} Similarly, using \cref{lem : lemma 2.1} with $\sigma=q, s=\theta q, \vartheta=\theta, r=2,$ we get
\begin{align*}
&\miint{J_{2\rho, \lambda}(z_0)}b(z_0) \frac{|u-(u)_{z_0, 2\rho, \lambda}|^q}{(2\rho)^q}\, dz\\
&\quad \leq c\miint{J_{2 \rho,\lambda}\left(z_0\right)}\left( \frac{\left|u-(u)_{z_0 ; 2 \rho,\lambda}\right|^{\theta q}}{(2 \rho)^{\theta q}}+|D u|^{\theta q}\,\right) dz \;\left(S(u)^q_{z_0 ; 2 \rho,\lambda}\right)^{\frac{(1-\theta) q}{2}}.
\end{align*}
This completes the proof.
\end{proof}
Now we are ready to prove the reverse H\"{o}lder inequality for $q$-phase.

\begin{lemma}\label{lem : lemma 3.12}
    Let $u$ be a weak solution to \eqref{eq: main equation}. Then for $J_{c_v(2 \rho),\lambda}\left(z_0\right) \subset \Omega_T$ with \eqref{eq: q-phase condition}, there exist constants $c=c(\textnormal{\texttt{data}}) \geq 1$ and $\theta_0=\theta_0(n,p,q) \in (0,1)$ such that for any $\theta\in (\theta_0,1),$
    \begin{align*}
        \miint{J_{\rho,\lambda}(z_0)} H(z,|Du|)\; dz\leq c \left(\miint{J_{2\rho,\lambda}(z_0)} [H(z,|Du|)]^\theta\;dz\right)^\frac{1}{\theta}+c\miint{J_{2\rho,\lambda}(z_0)} H(z,|F|)\; dz.
    \end{align*}
\end{lemma}
\begin{proof}
From the energy estimate, \cref{lem : the Caccioppoli inequality}, we have
\begin{align*}
\miint{J_{\rho, \lambda}(z_0)}H(z,|Du|)\;dz
& \leq c\underbrace{\miint{J_{2\rho, \lambda}(z_0)}  H\left(z,\frac{|u-(u)^q_{z_0, 2\rho, \lambda}|}{2\rho}\right) dz}_{\mathrm{I}}\\
&\quad + c \underbrace{\la^{q -2}\miint{J_{2\rho, \lambda}(z_0)}\frac{|u-(u)^q_{z_0, 2\rho, \lambda}|^2}{(2\rho)^2}\;dz}_{\mathrm{II}} +c\miint{J_{2\rho, \lambda}(z_0)} H(z,|F|) \;dz.   
\end{align*}
\textbf{Estimate of $\mathrm{I}$:} It follows from  \cref{Lem : lemma 3.11} and \cref{lem : Lemma 3.10} that
\begin{align*}
&\miint{J_{2\rho, \lambda}(z_0)}  H\left(z,\frac{|u-(u)^q_{z_0, 2\rho, \lambda}|}{2\rho}\right) dz \\
&\leq c\lambda^{(1-\theta)q} \miint{J_{2\rho, \tau}(z_0)}\left(\inf_{\tilde{z} \in J_{2\rho, \lambda}(z_0)}a(\tilde{z})^{\theta}\frac{\left|u-(u)^q_{z_0 ; 2 \rho,\lambda}\right|^{\theta p}}{(2 \rho)^{\theta p}}+ \frac{\left|u-(u)^q_{z_0 ; 2 \rho,\lambda}\right|^{\theta q}}{(2 \rho)^{\theta q}} \right) dz \\
&\quad +\lambda^{(1-\theta)q}\miint{J_{2\rho, \lambda}(z_0)}\left(\inf_{\tilde{z}\in J_{2\rho, \tau}(z_0)}a(\tilde{z})^{\theta}|Du|^{\theta p}+|Du|^{\theta q}\right) dz+\omega_a(2\rho)\miint{J_{2\rho, \lambda}(z_0)}\frac{\left|u-(u)^q_{z_0 ; 2 \rho,\lambda}\right|^p}{(2 \rho)^p}\,dz.
\end{align*}
To estimate the first term of the above inequality, we use \cref{lem: q-term of parabolic poincare inequality on q-intrinsic cylinders_2} and  \cref{lem:q-term of parabolic poincare inequality on q-intrinsic cylinders_1}. Thus, we obtain
\begin{align*}
&\miint{J_{2\rho, \lambda}(z_0)}  H\left(z,\frac{|u-(u)_{z_0, 2\rho, \lambda}|}{2\rho}\right) dz\\
&\leq c \lambda^{(1-\theta)q}\left[c \miint{J_{2\rho,\lambda}(z_0)}H(z,|Du|)^{\theta}\; dz +c\left(\miint{J_{2\rho,\lambda}(z_0)}H(z,|F|)\; d z\right)^{\theta}\right] \\
&\quad +\omega_a(2\rho)\miint{J_{2\rho, \lambda}(z_0)}\frac{\left|u-(u)_{z_0 ; 2 \rho,\lambda}\right|^p}{(2 \rho)^p}\,dz.
\end{align*}
Moreover, from the estimates of  \cref{lem : Lemma 3.10}, we also have
\begin{align*}
    \omega_a(2\rho)\miint{J_{2\rho, \lambda}(z_0)}\frac{\left|u-(u)_{z_0 ; 2 \rho,\lambda}\right|^p}{(2 \rho)^p}\,dz \leq c \omega_a(2\rho)\lambda^q\leq \varepsilon_1 \lambda^q\,\,\,\, \text{for small}\,\,\, \rho>0.
\end{align*}
Next, an application of Young's inequality with $\left(\frac{1}{\theta}, \frac{1}{1-\theta}\right)$ gives
\begin{align*}
  &\lambda^{(1-\theta)q}\left[c \miint{J_{2\rho,\lambda}(z_0)} H(z,|Du|)^{\theta}\; dz +c\left(\miint{J_{2\rho,\lambda}(z_0)}H(z,|F|)\; d z\right)^{\theta}\right]\\
  &\quad \leq \varepsilon_2\lambda^q + c(\varepsilon_2)\left[\left(\miint{J_{2\rho, \lambda}(z_0)}H(z, |Du|)^{\theta}\, dz \right)^{\frac{1}{\theta}}+ \miint{J_{2\rho, \lambda}(z_0)}H(z, |F|)\, dz\right].
\end{align*}
Thus, combining the above estimates, we finally get
\begin{align*}
 \miint{J_{2\rho, \lambda}(z_0)}  H\left(z,\frac{|u-(u)_{z_0, 2\rho, \lambda}|}{2\rho}\right)\; dz \leq \varepsilon \lambda^q&+ c\left(\miint{J_{2\rho, \lambda}(z_0)}H(z, |Du|)^{\theta}\, dz \right)^{\frac{1}{\theta}}\\
 &+ c\miint{J_{2\rho, \lambda}(z_0)}H(z, |F|)\, dz,  
\end{align*}
where $\varepsilon = \max\left\{\varepsilon_1, \varepsilon_2\right\}.$

\noindent\textbf{Estimate of $\mathrm{II}$:} Choosing $\sigma =2, s=\theta q, r=2,\vartheta=\frac{1}{2},$ in \cref{lem : lemma 2.1}, we get
\begin{align*}
&\lambda^{q-2}\miint{J_{2\rho,\lambda}(z_0)}\frac{|u-(u)^q_{z_0;2\rho,\lambda}|^2}{(2\rho)^2}\; dz \\
&\leq c\lambda^{q-2} \dashint_{I_{2\rho}(t_0)}\left(\dashint_{B_{2\rho}(x_0)}\left(\frac{|u-(u)^q_{z_0;2\rho,\lambda}|^{\theta q}}{(2\rho)^{\theta q}}+|Du|^{\theta q}\right)\,dx\right)^{\frac{1}{\theta q}} dt \, \left(S(u)^q_{z_0, 2\rho, \lambda}\right)^{\frac{1}{2}}\\
&\leq c\lambda^{q-1} \left[\miint{J_{\tau,\lambda}(z_0)}H(z,|Du|)^{\theta}\, dz
         +\left(\miint{J_{\tau,\lambda}(z_0)}H(z,|F|) \, d z\right)^{\theta}\right]^{\frac{1}{\theta q}}.
\end{align*}
Note that, in the last step, we used \cref{lem : Lemma 3.10}, \cref{lem:q-term of parabolic poincare inequality on q-intrinsic cylinders_1} and the second condition of \eqref{eq: q-phase condition}. Now applying Young's inequality with $\left(q, \frac{q}{q-1}\right)$, we get
\begin{align*}
\lambda^{q-2}\miint{J_{2\rho,\lambda}(z_0)}\frac{|u-(u)_{z_0;2\rho,\lambda}|^2}{(2\rho)^2}\; dz \leq \varepsilon \lambda^q + c\left(\miint{J_{2\rho, \lambda}(z_0)}H(z, |Du|)^{\theta}\, dz \right)^{\frac{1}{\theta}}+ c\miint{J_{2\rho, \lambda}(z_0)}H(z, |F|)\, dz.
\end{align*}
Hence, combining all the above estimates and absorbing $\varepsilon \lambda^q$ by the last condition of \eqref{eq: q-phase condition} in the left-hand side, we complete the proof.
\end{proof}
\subsection{Reverse H\"{o}lder inequality for $(p, q)$-phase} In this subsection, we prove reverse H\"{o}lder inequality for $(p,q)$-phase. In the $(p, q)$-intrinsic case, we consider the following assumptions:                                                  
\begin{equation}\label{Eq 3.8}    
    \left\{
    \begin{aligned}
        &Ka(z_0)\lambda^p \leq b(z_0)\lambda^q,\\
        &\frac{a(z)}{3}\leq a(z_0)\leq 3a(z) \quad \text{and}\quad \frac{b(z)}{3}\leq b(z_0)\leq 3b(z) \ \, \text{for all}\,\,\, z\in G_{\rho, \la}(z_0),\\
        &\miint{G_{\tau,\lambda}(z_0)} [H(z,|Du|)+H(z,|F|)] \; dz <a(z_0)\lambda^p+b(z_0)\lambda^q \ \, \text{for every }\tau\in(\rho,c_v(2\rho)],\\
        &\miint{G_{\rho,\lambda}(z_0)} [H(z,|Du|)+H(z,|F|)] \; dz =a(z_0)\lambda^p+b(z_0)\lambda^q,
    \end{aligned}
    \right.
\end{equation}
or
\begin{equation}\label{Eq 3.9}    
    \left\{
    \begin{aligned}
        &Ka(z_0)\lambda^p \geq b(z_0)\lambda^q,\\
        &\frac{a(z)}{3}\leq a(z_0)\leq 3a(z) \quad \text{and}\quad \frac{b(z)}{3}\leq b(z_0)\leq 3b(z) \ \, \text{for all}\,\,\, z\in G_{\rho, \la}(z_0),\\
        &\miint{G_{\tau,\lambda}(z_0)} [H(z,|Du|)+H(z,|F|)] \; dz <a(z_0)\lambda^p+b(z_0)\lambda^q \ \, \text{for every }\tau\in(\rho,c_v(2\rho)],\\
        &\miint{G_{\rho,\lambda}(z_0)} [H(z,|Du|)+H(z,|F|)] \; dz =a(z_0)\lambda^p+b(z_0)\lambda^q.
    \end{aligned}
    \right.
\end{equation}
\begin{remark}\label{REM3.1}
Note that the second assumption of the above conditions \eqref{Eq 3.8}--\eqref{Eq 3.9} is related to the strict positivity of the coefficients, i.e.,
$$\inf_{z\in G_{\tau, \lambda}(z_0)}a(z)>0 \quad \text{and} \quad \inf_{z\in G_{\tau, \lambda}(z_0)}b(z)>0$$
in the $(p, q)$-intrinsic case. However, in contrast to the $p$-intrinsic or the $q$-intrinsic case, we cannot guarantee a strictly positive uniform lower bound for these quantities (see Subsection \ref{deduction of pq case} for further details). Therefore, in this case, it is tempting not to impose any restriction on the simultaneous smallness of the competing coefficients $a(\cdot)$ and $b(\cdot)$. On the other hand, if $0<b(\cdot)<\varepsilon$ for some $\varepsilon >0,$ from assumption $a(\cdot)+b(\cdot)\geq \delta,$ we get $a(\cdot)\geq \delta-\varepsilon.$ Heuristically speaking, when $b(\cdot)$ is sufficiently small, the system exhibits $p$-growth behavior, and similarly, when $a(\cdot)$ is sufficiently small, it exhibits $q$-growth behavior. However, the conditions in \eqref{Eq 3.8}--\eqref{Eq 3.9} cannot be directly reduced to either the $p$-intrinsic or $q$-intrinsic case and require separate treatment.
\end{remark}



\begin{lemma}\label{Lemma 3.13}
Let $u$ be a weak solution to \eqref{eq: main equation}. Then there exists a constant $c=c(\textnormal{\texttt{data}}) \geq 1$ such that for any $G_{4\rho,\lambda}(z_0)\subset \Omega_T$ with \eqref{Eq 3.8}--\eqref{Eq 3.9}, $\tau \in [2\rho,4\rho]$ and $\theta \in ((q-1)/p,1]$, we have
\begin{align}\label{Eq 3.10}
&\miint{G_{\tau,\lambda}(z_0)} \frac{|u-(u)^{(p,q)}_{z_0;\tau,\lambda}|^{\theta p}}{\tau^{\theta p}}\; dz \leq c \miint{G_{\tau,\lambda}(z_0)} \left(|Du|+|F|\right)^{\theta p}\, dz
\end{align}
and 
\begin{align}\label{Eq: 3.11}
&\miint{G_{\tau,\lambda}(z_0)} \frac{|u-(u)^{(p,q)}_{z_0;\tau,\lambda}|^{\theta q}}{\tau^{\theta q}}\; dz \leq c \miint{G_{\tau,\lambda}(z_0)} \left(|Du|+|F|\right)^{\theta q}\, dz.
\end{align}
\end{lemma}
\begin{proof}
We observe from \eqref{eq : parabolic poincare inequality in Section 3} of \cref{lem : parabolic poincare inequality in Section 3} that
\begin{align}\label{EQQ3.12}
    &\miint{G_{\tau, \lambda}(z_0)}\frac{|u-(u)^{(p,q)}_{z_0, \tau, \lambda}|^{\theta p}}{\tau^{\theta p}}\,dz\nonumber\\
    &\leq c\miint{G_{\tau, \lambda}(z_0)}|Du|^{\theta p}\, dz \nonumber\\
    &\quad + c\left(\frac{\lambda^2}{a(z_0)\lambda^p+b(z_0)\lambda^q}\miint{G_{\tau, \lambda}(z_0)}a(z)\left(|Du|^{p-1}+|F|^{p-1}\right)+b(z)\left(|Du|^{q-1}+|F|^{q-1}\right) dz\right)^{\theta p}\nonumber\\
    &\leq c\miint{G_{\tau, \lambda}(z_0)}|Du|^{\theta p}\, dz+ c\underbrace{\left(\frac{\lambda^2}{a(z_0)\lambda^p+b(z_0)\lambda^q}\miint{G_{\tau, \lambda}(z_0)}a(z)\left(|Du|+|F|\right)^{p-1} dz\right)^{\theta p}}_{\mathrm{I}}\nonumber\\
    &\quad +c\underbrace{\left(\frac{\lambda^2}{a(z_0)\lambda^p+b(z_0)\lambda^q}\miint{G_{\tau, \lambda}(z_0)}b(z)\left(|Du|+|F|\right)^{q-1} dz\right)^{\theta p}}_{\mathrm{II}}.
\end{align}
\textbf{Case I:} In this case, we estimate $\mathrm{I}$ and $\mathrm{II}$ of \eqref{EQQ3.12} considering the assumption \eqref{Eq 3.8}. First we observe, since $b(z_0)\neq 0,$ from the first, second and the third conditions of \eqref{Eq 3.8}, we get
\begin{align*}
 \miint{G_{\tau, \lambda}(z_0)} b(z_0)\left(|Du|+|F|\right)^q\,dz \leq c \miint{G_{\tau, \lambda}(z_0)} b(z)\left(|Du|+|F|\right)^q\,dz&\leq c \left(a(z_0)\lambda^p+b(z_0)\lambda^q\right)\\
 &\leq c\, b(z_0)\lambda^q.
\end{align*}
Thus, we have
\begin{align}\label{bound by q}
    \miint{G_{\tau, \lambda}(z_0)}\left(|Du|+|F|\right)^q\, dz \leq c\lambda^q.
\end{align}

\noindent\textbf{Estimate of $\mathrm{I}$:} Using the second condition of \eqref{Eq 3.8} and H\"{o}lder's inequality, we obtain
\begin{align*}
    &\left(\frac{\lambda^2}{a(z_0)\lambda^p+b(z_0)\lambda^q}\miint{G_{\tau, \lambda}(z_0)}a(z)\left(|Du|+|F|\right)^{p-1}\,dz\right)^{\theta p}\\
    &\leq \left(\frac{\lambda^2}{a(z_0)\lambda^p+b(z_0)\lambda^q}\right)^{\theta p}a(z_0)^{\theta}\left(\miint{G_{\tau, \lambda}(z_0)}a(z)^{\frac{p-1}{p}}\left(|Du|+|F|\right)^{p-1}\,dz\right)^{\theta p}\\
    &\leq \left(\frac{\lambda^2}{a(z_0)\lambda^p+b(z_0)\lambda^q}\right)^{\theta p}a(z_0)^{\theta}\left(\miint{G_{\tau, \lambda}(z_0)}a(z)^\theta \left(|Du|+|F|\right)^{\theta p}\,dz\right)^{p-1}\\
    &\leq c\left(\frac{\lambda^2}{a(z_0)\lambda^p+b(z_0)\lambda^q}\right)^{\theta p}(a(z_0)\lambda^p+b(z_0)\lambda^q)^{\theta (p-2)}a(z_0)^{2\theta} \miint{G_{\tau, \lambda}(z_0)} \left(|Du|+|F|\right)^{\theta p}\,dz.
\end{align*}
Now using $a(z_0)\leq \frac{1}{K} b(z_0)\lambda^{q-p}$ from the assumption \eqref{Eq 3.8}, we have
\begin{align*}
 &\left(\frac{\lambda^2}{a(z_0)\lambda^p+b(z_0)\lambda^q}\right)^{\theta p}(a(z_0)\lambda^p+b(z_0)\lambda^q)^{\theta (p-2)}a(z_0)^{2\theta}\\
 &\quad \leq \left(\frac{\lambda^p}{a(z_0)\lambda^p+b(z_0)\lambda^q}\right)^{2\theta}b(z_0)^{2\theta}\lambda^{2\theta(q-p)}=\left(\frac{b(z_0)\lambda^{q}}{a(z_0)\lambda^p+b(z_0)\lambda^q}\right)^{2\theta} \leq 1.
\end{align*}
Thus we finally obtain
\begin{align*}
\left(\frac{\lambda^2}{a(z_0)\lambda^p+b(z_0)\lambda^q}\miint{G_{\tau, \lambda}(z_0)}a(z)\left(|Du|+|F|\right)^{p-1}\,dz\right)^{\theta p} \leq c \miint{G_{\tau, \lambda}(z_0)} \left(|Du|+|F|\right)^{\theta p}\,dz.    
\end{align*}

\noindent\textbf{Estimate of $\mathrm{II}$:} Similarly, using the second condition of \eqref{Eq 3.8} and H\"{o}lder's inequality, we get
\begin{align*}
&\left(\frac{\lambda^2}{a(z_0)\lambda^p+b(z_0)\lambda^q}\miint{G_{\tau, \lambda}(z_0)}b(z)\left(|Du|+|F|\right)^{q-1}\,dz\right)^{\theta p}\\
&\leq \left(\frac{\lambda^2 b(z_0)}{a(z_0)\lambda^p+b(z_0)\lambda^q}\right)^{\theta p}\left(\miint{G_{\tau, \lambda}(z_0)}\left(|Du|+|F|\right)^q\, dz\right)^{\frac{(q-2)\theta p}{q}} \miint{G_{\tau, \lambda}(z_0)}\left(|Du|+|F|\right)^{\theta p}\, dz\\
&\leq \left(\frac{\lambda^2 b(z_0)}{a(z_0)\lambda^p+b(z_0)\lambda^q}\right)^{\theta p} \lambda^{(q-2)\theta p} \miint{G_{\tau, \lambda}(z_0)}\left(|Du|+|F|\right)^{\theta p}\, dz\\
&\leq \miint{G_{\tau, \lambda}(z_0)}\left(|Du|+|F|\right)^{\theta p}\, dz.
\end{align*}
Thus, combining the above estimates, we obtain \eqref{Eq 3.10}.

\noindent\textbf{Case II:} In this case, we estimate $\mathrm{I}$ and $\mathrm{II}$ of \eqref{EQQ3.12} considering the assumption \eqref{Eq 3.9}. Since $a(z_0)\neq 0,$ using the first, second and the third assumptions of \eqref{Eq 3.9}, we get
\begin{align*}
    a(z_0)\miint{G_{\tau, \lambda}(z_0)}\left(|Du|+|F|\right)^p\,dz\leq 3 \miint{G_{\tau, \lambda}(z_0)}a(z)\left(|Du|+|F|\right)^p\,dz&\leq a(z_0)\lambda^p+b(z_0)\lambda^q\\&\leq c a(z_0)\lambda^p,
\end{align*}
and thus we conclude that
\begin{align}\label{bound by lambda-p}
    \miint{G_{\tau, \lambda}(z_0)}\left(|Du|+|F|\right)^p\,dz \leq c \lambda^p.
\end{align}
The estimates in this case are similar to those in Case I.

\noindent \textbf{Estimate of $\mathrm{I}$:} Using the second condition of \eqref{Eq 3.9}, \eqref{bound by lambda-p} and H\"{o}lder's inequality, we obtain
\begin{align*}
    &\left(\frac{\lambda^2}{a(z_0)\lambda^p+b(z_0)\lambda^q}\miint{G_{\tau, \lambda}(z_0)}a(z)\left(|Du|+|F|\right)^{p-1}\,dz\right)^{\theta p}\\
    & \leq c \left(\frac{a(z_0)\lambda^2}{a(z_0)\lambda^p+b(z_0)\lambda^q}\right)^{\theta p}\lambda^{\theta p (p-2)} \miint{G_{\tau, \lambda}(z_0)}\left(|Du|+|F|\right)^{\theta p}\, dz\\
    &=c\left(\frac{a(z_0)\lambda^p}{a(z_0)\lambda^p+b(z_0)\lambda^q}\right)^{\theta p} \miint{G_{\tau, \lambda}(z_0)}\left(|Du|+|F|\right)^{\theta p}\, dz \leq c\miint{G_{\tau, \lambda}(z_0)}\left(|Du|+|F|\right)^{\theta p}\, dz.
\end{align*}
\textbf{Estimate of $\mathrm{II}$:} Similarly, we estimate
\begin{align*}
 &\left(\frac{\lambda^2}{a(z_0)\lambda^p+b(z_0)\lambda^q}\miint{G_{\tau, \lambda}(z_0)}b(z)\left(|Du|+|F|\right)^{q-1}\,dz\right)^{\theta p}\\
 &\leq c \left(\frac{\lambda^2 b(z_0)}{a(z_0)\lambda^p+b(z_0)\lambda^q}\right)^{\theta p}\left(\miint{G_{\tau, \lambda}(z_0)}\left(|Du|+|F|\right)^{p}\,dz\right)^{\theta (q-2)} \miint{G_{\tau, \lambda}(z_0)}\left(|Du|+|F|\right)^{\theta p}\,dz\\
 &\leq c \left(\frac{\lambda^q b(z_0)}{a(z_0)\lambda^p+b(z_0)\lambda^q}\right)^{\theta p} \miint{G_{\tau, \lambda}(z_0)}\left(|Du|+|F|\right)^{\theta p}\,dz \leq \miint{G_{\tau, \lambda}(z_0)}\left(|Du|+|F|\right)^{\theta p}\,dz.
\end{align*}
Thus, combining the above estimates, we obtain \eqref{Eq 3.10}.

Next we prove the estimate \eqref{Eq: 3.11}. Again, it follows from \eqref{eq : parabolic poincare inequality in Section 3} of \cref{lem : parabolic poincare inequality in Section 3} that
\begin{align}\label{EQQ3.14}
&\miint{G_{\tau, \lambda}(z_0)}\frac{|u-(u)^{(p,q)}_{z_0, \tau, \lambda}|^{\theta q}}{\tau^{\theta q}}\,dz\nonumber\\
&\quad \leq c\miint{G_{\tau, \lambda}(z_0)}|Du|^{\theta q}\, dz+ c\underbrace{\left(\frac{\lambda^2}{a(z_0)\lambda^p+b(z_0)\lambda^q}\miint{G_{\tau, \lambda}(z_0)}a(z)\left(|Du|+|F|\right)^{p-1}\,dz\right)^{\theta q}}_{\mathrm{I}}\nonumber\\
&\qquad +c\underbrace{\left(\frac{\lambda^2}{a(z_0)\lambda^p+b(z_0)\lambda^q}\miint{G_{\tau, \lambda}(z_0)}b(z)\left(|Du|+|F|\right)^{q-1}\,dz\right)^{\theta q}}_{\mathrm{II}}.
\end{align}

\noindent \textbf{Case I:} In this case, we estimate $\mathrm{I}$ and $\mathrm{II}$ of \eqref{EQQ3.14} considering the assumption \eqref{Eq 3.8}.

\noindent \textbf{Estimate of $\mathrm{I}$:} Using the second condition of \eqref{Eq 3.8}, H\"{o}lder's inequality and \eqref{bound by q}, we see that
\begin{align*}
&\left(\frac{\lambda^2}{a(z_0)\lambda^p+b(z_0)\lambda^q}\miint{G_{\tau, \lambda}(z_0)}a(z)\left(|Du|+|F|\right)^{p-1}\,dz\right)^{\theta q}\\
&\leq c\left(\frac{\lambda^2 a(z_0)}{a(z_0)\lambda^p+b(z_0)\lambda^q}\right)^{\theta q}\left(\miint{G_{\tau, \lambda}(z_0)}\left(|Du|+|F|\right)^{\theta q}\,dz\right)^{p-1}\\
&\leq c \left(\frac{\lambda^2 a(z_0)}{a(z_0)\lambda^p+b(z_0)\lambda^q}\right)^{\theta q}\left(\miint{G_{\tau, \lambda}(z_0)}\left(|Du|+|F|\right)^q\, dz\right)^{\theta(p-2)} \miint{G_{\tau, \lambda}(z_0)}\left(|Du|+|F|\right)^{\theta q}\, dz\\
&\leq c\left(\frac{a(z_0)\lambda^p}{a(z_0)\lambda^p+b(z_0)\lambda^q}\right)^{\theta q} \miint{G_{\tau, \lambda}(z_0)}\left(|Du|+|F|\right)^{\theta q}\, dz \\
&\leq c \miint{G_{\tau, \lambda}(z_0)}\left(|Du|+|F|\right)^{\theta q}\, dz.
\end{align*}
\textbf{Estimate of $\mathrm{II}$:} Similarly, we can estimate
\begin{align*}
&\left(\frac{\lambda^2}{a(z_0)\lambda^p+b(z_0)\lambda^q}\miint{G_{\tau, \lambda}(z_0)}b(z)\left(|Du|+|F|\right)^{q-1}\,dz\right)^{\theta q}\\
&\leq c\left(\frac{\lambda^2 b(z_0)}{a(z_0)\lambda^p+b(z_0)\lambda^q}\right)^{\theta q}\left(\miint{G_{\tau, \lambda}(z_0)}\left(|Du|+|F|\right)^q\,dz\right)^{\theta (q-2)} \miint{G_{\tau, \lambda}(z_0)}\left(|Du|+|F|\right)^{\theta q}\,dz\\
&\leq c \left(\frac{\lambda^q b(z_0)}{a(z_0)\lambda^p+b(z_0)\lambda^q}\right)^{\theta q} \miint{G_{\tau, \lambda}(z_0)}\left(|Du|+|F|\right)^{\theta q}\,dz\\
&\leq c \miint{G_{\tau, \lambda}(z_0)}\left(|Du|+|F|\right)^{\theta q}\,dz.
\end{align*}
Combining the above estimates completes the proof of \eqref{Eq: 3.11}.

\noindent \textbf{Case II:} We estimate $\mathrm{I}$ and $\mathrm{II}$ of \eqref{EQQ3.14} by considering the assumption \eqref{Eq 3.9}.

\noindent\textbf{Estimate of $\mathrm{I}$:} Using \eqref{bound by lambda-p}, H\"{o}lder's inequality and the second condition of \eqref{Eq 3.9}, we get
\begin{align*}
&\left(\frac{\lambda^2}{a(z_0)\lambda^p+b(z_0)\lambda^q}\miint{G_{\tau, \lambda}(z_0)}a(z)\left(|Du|+|F|\right)^{p-1}\,dz\right)^{\theta q}\\
&\leq c\left(\frac{a(z_0)\lambda^2}{a(z_0)\lambda^p+b(z_0)\lambda^q}\right)^{\theta q}\left(\miint{G_{\tau, \lambda}(z_0)}\left(|Du|+|F|\right)^{\theta p}\,dz\right)^{\frac{(p-1)q}{p}}\\
&=c\left(\frac{a(z_0)\lambda^2}{a(z_0)\lambda^p+b(z_0)\lambda^q}\,\right)^{\theta q}\left(\miint{G_{\tau, \lambda}(z_0)}\left(|Du|+|F|\right)^{\theta p}\,dz\right)^{\frac{(p-2)q}{p}}\left(\miint{G_{\tau, \lambda}(z_0)}\left(|Du|+|F|\right)^{\theta p}\,dz\right)^{\frac{q}{p}}\\
&\leq c\left(\frac{a(z_0)\lambda^2}{a(z_0)\lambda^p+b(z_0)\lambda^q}\right)^{\theta q}\lambda^{\theta q(p-2)} \miint{G_{\tau, \lambda}(z_0)}\left(|Du|+|F|\right)^{\theta q} dz\\
&=c\left(\frac{a(z_0)\lambda^p}{a(z_0)\lambda^p+b(z_0)\lambda^q}\right) \miint{G_{\tau, \lambda}(z_0)}\left(|Du|+|F|\right)^{\theta q}\,dz\\
&\leq \miint{G_{\tau, \lambda}(z_0)}\left(|Du|+|F|\right)^{\theta q} dz.
\end{align*}

\noindent \textbf{Estimate of $\mathrm{II}$:} Similarly, we estimate
\begin{align*}
&\left(\frac{\lambda^2}{a(z_0)\lambda^p+b(z_0)\lambda^q}\miint{G_{\tau, \lambda}(z_0)}b(z)\left(|Du|+|F|\right)^{q-1}\,dz\right)^{\theta q}\\
&\leq c\left(\frac{\lambda^2}{a(z_0)\lambda^p+b(z_0)\lambda^q}\right)^{\theta q}b(z_0)^{\theta}\left(\miint{G_{\tau, \lambda}(z_0)}b(z)^{\frac{q-1}{q}}\left(|Du|+|F|\right)^{q-1}\, dz\right)^{\theta q}\\
&\leq c \left(\frac{\lambda^2}{a(z_0)\lambda^p+b(z_0)\lambda^q}\right)^{\theta q}b(z_0)^{\theta}\left(\miint{G_{\tau, \lambda}(z_0)}b(z)^{\theta}\left(|Du|+|F|\right)^{\theta q}dz\right)^{q-1}\\
&\leq c \left(\frac{\lambda^2}{a(z_0)\lambda^p+b(z_0)\lambda^q}\right)^{\theta q}(a(z_0)\lambda^p+b(z_0)\lambda^q)^{\theta (q-2)}b(z_0)^{2\theta} \miint{G_{\tau, \lambda}(z_0)} \left(|Du|+|F|\right)^{\theta q}dz.
\end{align*}
As in the previous case, using the first condition of \eqref{Eq 3.9}, i.e., $Ka(z_0)\lambda^p\geq b(z_0)\lambda^q,$ we obtain
\begin{align*}
 \left(\frac{\lambda^2}{a(z_0)\lambda^p+b(z_0)\lambda^q}\right)^{\theta q}(a(z_0)\lambda^p+b(z_0)\lambda^q)^{\theta (q-2)}b(z_0)^{2\theta} \leq 1.   
\end{align*}
It follows that
\begin{align*}
    \left(\frac{\lambda^2}{a(z_0)\lambda^p+b(z_0)\lambda^q}\miint{G_{\tau, \lambda}(z_0)}b(z)\left(|Du|+|F|\right)^{q-1}\,dz\right)^{\theta q} \leq c \miint{G_{\tau, \lambda}(z_0)} \left(|Du|+|F|\right)^{\theta q} dz.
\end{align*}
Thus, combining the above estimates, we obtain \eqref{Eq: 3.11} for this case.
\end{proof}
Next, we prove the following:
\begin{lemma}\label{Lemma 3.14}
Let $u$ be a weak solution to \eqref{eq: main equation}. Then for $G_{c_v(2 \rho),\lambda}(z_0) \subset \Omega_T$ with \eqref{Eq 3.8}--\eqref{Eq 3.9}, there exists a constant $c=c(\textnormal{\texttt{data}}) \geq 1$ such that
\begin{align*}
S(u)^{(p,q)}_{z_0 ; 2 \rho,\lambda}=\sup _{I^{(p,q)}_{2 \rho,\lambda}(t_0)} \dashint_{B_{2 \rho}(x_0)} \frac{\left|u-(u)^{(p,q)}_{z_0 ; 2 \rho,\lambda}\right|^2}{(2 \rho)^2}\; dx \leq c \lambda^2 .
\end{align*}
\end{lemma}
\begin{proof}
Let $2\rho \leq \rho_1<\rho_2 \leq 4\rho$. By \cref{lem : the Caccioppoli inequality}, there exists a constant $c$ depending on $n,p,q,\nu$ and $L$ such that 
\begin{align}
&\left(a(z_0)\lambda^{p-2}+b(z_0)\lambda^{q-2}\right)S(u)^{(p,q)}_{z_0;\rho_1,\lambda}\nonumber\\
&\leq \frac{c\rho_2^q}{(\rho_2-\rho_1)^q}\miint{G_{\rho_2,\lambda}(z_0)}\left( a(z)\frac{|u-(u)^{(p,q)}_{z_0;\rho_2,\lambda}|^p}{\rho^p_2}+b(z)\frac{|u-(u)^{(p,q)}_{z_0;\rho_2,\lambda}|^q}{\rho^q_2} \right) dz \nonumber\\
&\quad + \frac{c\rho_2^2 \left(a(z_0)\lambda^{p-2}+b(z_0)\lambda^{q-2}\right)}{(\rho_2-\rho_1)^2}\miint{G_{\rho_2,\lambda}(z_0)}\frac{|u-(u)^{(p,q)}_{z_0;\rho_2,\lambda}|^2}{\rho_2^2}\; dz\nonumber\\\label{eq : 3.12}
&\quad +c\miint{G_{\rho_2,\lambda}(z_0)} H(z,|F|)\; dz.
\end{align}
Using \cref{Lemma 3.13} with $\theta =1$, we estimate the first term of \eqref{eq : 3.12} as follows:
\begin{align*}
 &\miint{G_{\rho_2,\lambda}(z_0)}\left( a(z)\frac{|u-(u)^{(p,q)}_{z_0;\rho_2,\lambda}|^p}{\rho^p_2}+b(z)\frac{|u-(u)^{(p,q)}_{z_0;\rho_2,\lambda}|^q}{\rho^q_2} \right) dz\\
 &\overset{\eqref{Eq 3.8}\,\, \text{or} \,\, \eqref{Eq 3.9}}{\leq} c\miint{G_{\rho_2,\lambda}(z_0)} \left(a(z_0)\frac{|u-(u)^{(p,q)}_{z_0;\rho_2,\lambda}|^p}{\rho^p_2}+b(z_0)\frac{|u-(u)^{(p,q)}_{z_0;\rho_2,\lambda}|^q}{\rho^q_2} \right) dz\\
 &\overset{\cref{Lemma 3.13}}{\leq} c\miint{G_{\rho_2,\lambda}(z_0)}a(z_0)\left(|Du|+|F|\right)^p+b(z_0)\left(|Du|+|F|\right)^q\,dz\\
 &\overset{\eqref{Eq 3.8}\,\, \text{or} \,\, \eqref{Eq 3.9}}{\leq} c \miint{G_{\rho_2,\lambda}(z_0)}a(z)\left(|Du|+|F|\right)^p+b(z)\left(|Du|+|F|\right)^q\,dz\overset{\eqref{Eq 3.8}\,\, \text{or} \,\, \eqref{Eq 3.9}}{\leq}c\left(a(z_0)\lambda^p+b(z_0)\lambda^q\right).
\end{align*}
For the case \eqref{Eq 3.8},  using \cref{lem : lemma 2.1} with $\sigma =2, s=q, r=2, \vartheta=\frac{1}{2},$ we get
\begin{align*}
   \miint{G_{\rho_2,\lambda}(z_0)}\frac{|u-(u)^{(p,q)}_{z_0;\rho_2,\lambda}|^2}{\rho_2^2}\; dz \leq c\left(\miint{G_{\rho_2, \lambda}(z_0)}\left(\frac{|u-(u)^{(p,q)}_{z_0;\rho_2,\lambda}|^q}{\rho_2^q}+|Du|^q\right)dz\right)^{\frac{1}{q}}\left(S(u)^{(p,q)}_{z_0, \rho_2, \lambda}\right)^{\frac{1}{2}}.
\end{align*}
Now using \eqref{Eq: 3.11} of \cref{Lemma 3.13} and \eqref{bound by q}, we obtain
\begin{align*}
\left(\miint{G_{\rho_2, \lambda}(z_0)}\left(\frac{|u-(u)^{(p,q)}_{z_0;\rho_2,\lambda}|^q}{\rho_2^q}+|Du|^q\right)dz\right)^{\frac{1}{q}} \leq c \lambda.    
\end{align*}
On the other hand, for the case \eqref{Eq 3.9}, similarly,  we derive
\begin{align*}
   \miint{G_{\rho_2,\lambda}(z_0)}\frac{|u-(u)_{z_0;\rho_2,\lambda}|^2}{\rho_2^2}\; dz \leq c\left(\miint{G_{\rho_2, \lambda}(z_0)}\left(\frac{|u-(u)^{(p,q)}_{z_0;\rho_2,\lambda}|^p}{\rho_2^p}+|Du|^p\right)dz\right)^{\frac{1}{p}}\left(S(u)^{(p,q)}_{z_0, \rho_2, \lambda}\right)^{\frac{1}{2}}.
\end{align*}
Now using \eqref{Eq 3.10} of \cref{Lemma 3.13} and \eqref{bound by lambda-p}, we obtain
\begin{align*}
    \left(\miint{G_{\tau, \lambda}(z_0)}\left(\frac{|u-(u)_{z_0;\rho_2,\lambda}|^p}{\rho_2^p}+|Du|^p\right)dz\right)^{\frac{1}{p}} \leq c \lambda.
\end{align*}
Thus, together with the above estimates, we get
    \begin{align*}
     \miint{G_{\rho_2,\lambda}(z_0)}\frac{|u-(u)_{z_0;\rho_2,\lambda}|^2}{\rho_2^2}\; dz \leq \lambda \left(S(u)^{(p,q)}_{z_0, 2\rho, \lambda}\right)^{\frac{1}{2}}.
    \end{align*}
    Now applying Young's inequality and iteration lemma, \cref{iter_lemma}, we complete the proof.
\end{proof}
\begin{lemma}\label{Lemma 3.15}
Let $u$ be a weak solution to \eqref{eq: main equation}. Then for $G_{c_v(2 \rho),\lambda}\left(z_0\right) \subset \Omega_T$ with \eqref{Eq 3.8}--\eqref{Eq 3.9}, there exist constants $c=c(\textnormal{\texttt{data}}) \geq 1$ and $\theta_0=\theta_0(n, p, q) \in(0,1)$ such that for any $\theta \in\left(\theta_0, 1\right)$, we have
\begin{align}\label{EQQ3.15}
&\miint{G_{2 \rho,\lambda}\left(z_0\right)} H\left(z,\frac{\left|u-(u)^{(p, q)}_{z_0 ; 2 \rho,\lambda}\right|}{2 \rho}\right) d z \nonumber\\
&\leq c\left(a(z_0)\lambda^p+b(z_0)\lambda^q\right)^{1-\theta}\left[\miint{G_{2 \rho,\lambda}\left(z_0\right)} H(z, |Du|)^{\theta}\, dz+\left(\miint{G_{2\rho, \lambda}(z_0)}H(z, |F|)\, dz\right)^{\theta}\right]. 
\end{align}
\end{lemma}
\begin{proof}
 We use \cref{lem : lemma 2.1} with $\sigma=p, s=\theta p, \vartheta= \theta, r=2$ to deduce
 \begin{align*}
    &\miint{G_{2\rho, \lambda}(z_0)}a(z)\frac{\left|u-(u)^{(p,q)}_{z_0 ; 2 \rho,\lambda}\right|^p}{(2 \rho)^p}\, dz\\
     &\quad \leq c\left(\miint{G_{2\rho, \lambda}(z_0)}a(z_0)^{\theta}\frac{\left|u-(u)^{(p,q)}_{z_0 ; 2 \rho,\lambda}\right|^{\theta p}}{(2 \rho)^{\theta p}}+a(z_0)^{\theta }|Du|^{\theta p}\, dz\right)a(z_0)^{1-\theta}\left(S(u)^{(p,q)}_{z_0, 2\rho, \lambda}\right)^{\frac{(1-\theta)p}{2}}.
 \end{align*}
 Now we use \cref{Lemma 3.13} and \cref{Lemma 3.14} to get
 \begin{align*}
&\left(\miint{G_{2\rho, \lambda}(z_0)}a(z_0)^{\theta}\frac{\left|u-(u)^{(p,q)}_{z_0 ; 2 \rho,\lambda}\right|^{\theta p}}{(2 \rho)^{\theta p}}+a(z_0)^{\theta }|Du|^{\theta p}\, dz\right)a(z_0)^{1-\theta}\left(S(u)^{(p,q)}_{z_0, 2\rho, \lambda}\right)^{\frac{(1-\theta)p}{2}}\\
&\quad \leq c (a(z_0)\lambda^p)^{1-\theta}\miint{G_{2\rho, \lambda}(z_0)}a(z)^{\theta}\left(|Du|+|F|\right)^{\theta p}\, dz.
 \end{align*}
 Hence, it follows that
 \begin{align*}
 \miint{G_{2\rho, \lambda}(z_0)}a(z)\frac{\left|u-(u)^{(p,q)}_{z_0 ; 2 \rho,\lambda}\right|^p}{(2 \rho)^p}\, dz\leq c (a(z_0)\lambda^p)^{1-\theta}\miint{G_{2\rho, \lambda}(z_0)}a(z)^{\theta}\left(|Du|+|F|\right)^{\theta p}\, dz.    
 \end{align*}
 Similarly, we have
 \begin{align*}
 \miint{G_{2\rho, \lambda}(z_0)}b(z)\frac{\left|u-(u)^{(p,q)}_{z_0 ; 2 \rho,\lambda}\right|^q}{(2 \rho)^q}\, dz\leq c (b(z_0)\lambda^q)^{1-\theta}\miint{G_{2\rho, \lambda}(z_0)}b(z)^{\theta}\left(|Du|+|F|\right)^{\theta q}\, dz.        
 \end{align*}
 Combining the above two estimates and using H\"{o}lder's inequality, we get the desired estimate \eqref{EQQ3.15}.
\end{proof}
\begin{lemma}
    Let $u$ be a weak solution to \eqref{eq: main equation}. Then for $G_{c_v(2 \rho),\lambda}\left(z_0\right) \subset \Omega_T$ with \eqref{Eq 3.8}--\eqref{Eq 3.9}, there exist constants $c=c(\textnormal{\texttt{data}}) \geq 1$ and $\theta_0=\theta_0(n,p,q) \in (0,1)$ such that for any $\theta\in (\theta_0,1),$
    \begin{align*}
        &\miint{G_{\rho,\lambda}(z_0)} [H(z,|Du|)+H(z, |F|)]\; dz\\
        &\qquad \leq c \left(\miint{G_{2\rho,\lambda}(z_0)} [H(z,|Du|)]^\theta\;dz\right)^\frac{1}{\theta}+c\miint{G_{2\rho,\lambda}(z_0)} H(z,|F|)\; dz.
    \end{align*}
\end{lemma}
\begin{proof}
    From the energy estimate, \cref{lem : the Caccioppoli inequality}, we have
\begin{align*}
&\miint{G_{\rho, \lambda}(z_0)}H(z,|Du|)\;dz\\
& \leq c \underbrace{\miint{G_{2\rho, \lambda}(z_0)}  H\left(z,\frac{|u-(u)^{(p,q)}_{z_0, 2\rho, \lambda}|}{2\rho}\right) dz}_{\mathrm{I}}+ c\underbrace{\frac{a(z_0)\lambda^p+b(z_0)\lambda^q}{\lambda^2}\miint{G_{2\rho, \lambda}(z_0)}\frac{|u-(u)^{(p,q)}_{z_0, 2\rho, \lambda}|^2}{(2\rho)^2}\;dz}_{\mathrm{II}}\\
&\quad +c\miint{G_{2\rho, \lambda}(z_0)} H(z,|F|) \;dz.
\end{align*}
\textbf{Estimate of $\mathrm{I}$:} From \cref{Lemma 3.15}, applying Young's inequality with $\left(\frac{1}{1-\theta}, \frac{1}{\theta}\right)$ yields
\begin{align*}
&\miint{G_{2\rho, \lambda}(z_0)}  H\left(z,\frac{|u-(u)_{z_0, 2\rho, \lambda}|}{2\rho}\right) dz\\
&\leq c\left(a(z_0)\lambda^p+b(z_0)\lambda^q\right)^{1-\theta}\left[\miint{G_{2 \rho,\lambda}\left(z_0\right)} H(z, |Du|)^{\theta}\, dz+\left(\miint{G_{2\rho, \lambda}(z_0)}H(z, |F|)\, dz\right)^{\theta}\right]\\
&\leq \varepsilon \left(a(z_0)\lambda^p+b(z_0)\lambda^q\right)+ c(\varepsilon) \left[\left(\miint{G_{2 \rho,\lambda}\left(z_0\right)} H(z, |Du|)^{\theta}\, dz\right)^{\frac{1}{\theta}}+\miint{G_{2\rho, \lambda}(z_0)}H(z, |F|)\, dz \right].
\end{align*}
\textbf{Estimate of $\mathrm{II}$:} We estimate $\mathrm{II}$ by considering it in two separate parts. We proceed as follows:
\begin{align*}
   \underbrace{ a(z_0)\lambda^{p-2}\miint{G_{2\rho, \lambda}(z_0)}\frac{|u-(u)_{z_0, 2\rho, \lambda}|^2}{(2\rho)^2}\;dz}_{\mathrm{J}_1} +\underbrace{b(z_0)\lambda^{q-2}\miint{G_{2\rho, \lambda}(z_0)}\frac{|u-(u)_{z_0, 2\rho, \lambda}|^2}{(2\rho)^2}\;dz}_{\mathrm{J}_2}.
\end{align*}
First, we estimate $\mathrm{J}_1$ by using \cref{lem : lemma 2.1} with $\sigma =2, s=\theta p, \vartheta=\frac{1}{2}, r=2.$ Indeed, 
\begin{align*}
 \mathrm{J}_1=&\lambda^{p-2}\miint{G_{2\rho, \lambda}(z_0)}a(z_0)\frac{|u-(u)^{(p,q)}_{z_0, 2\rho, \lambda}|^2}{(2\rho)^2}\;dz\\
 &\overset{\cref{lem : lemma 2.1}}{\leq} c \lambda^{p-2} \left(\miint{G_{2\rho, \lambda}(z_0)}a(z_0)^{\theta}\frac{|u-(u)_{z_0, 2\rho, \lambda}|^{\theta p}}{(2\rho)^{\theta p}}+a(z_0)^{\theta}|Du|^{\theta p}\,dz\right)^{\frac{1}{\theta p}}a(z_0)^{\frac{p-1}{p}}\left(S(u)^{(p,q)}_{z_0, 2\rho, \lambda}\right)^{\frac{1}{2}}\\
 &\overset{\cref{Lemma 3.14}}{\leq} c \left(a(z_0)\lambda^p\right)^{\frac{p-1}{p}} \left(\miint{G_{2\rho, \lambda}(z_0)}a(z_0)^{\theta}\left(|Du|+|F|\right)^{\theta p}\,dz\right)^{\frac{1}{\theta p}}\\
 &\qquad\leq c \left(a(z_0)\lambda^p\right)^{\frac{p-1}{p}}\left(\miint{G_{2\rho, \lambda}(z_0)}a(z)^{\theta}\left(|Du|+|F|\right)^{\theta p}\,dz\right)^{\frac{1}{\theta p}}.
\end{align*}
Now applying Young's inequality with $(p, \frac{p}{p-1}),$ we obtain
\begin{align*}
    \mathrm{J}_1 &\leq \varepsilon a(z_0)\lambda^{p}+ c(\varepsilon)\left(\miint{G_{2\rho, \lambda}(z_0)}a(z)^{\theta}\left(|Du|+|F|\right)^{\theta p}\,dz\right)^{\frac{1}{\theta }}\\
    &\leq \varepsilon  a(z_0)\lambda^p+ c(\varepsilon)\left[\left(\miint{G_{2 \rho,\lambda}\left(z_0\right)} H(z, |Du|)^{\theta}\, dz\right)^{\frac{1}{\theta}}+\miint{G_{2\rho, \lambda}(z_0)}H(z, |F|)\, dz\right].
\end{align*}
Similarly, we can estimate $\mathrm{J}_2$ as follows:
\begin{align*}
    \mathrm{J}_2 \leq \varepsilon b(z_0)\lambda^q + c(\varepsilon)\left[\left(\miint{G_{2 \rho,\lambda}\left(z_0\right)} H(z, |Du|)^{\theta}\, dz\right)^{\frac{1}{\theta}}+\miint{G_{2\rho, \lambda}(z_0)}H(z, |F|)\, dz\right].
\end{align*}
Thus, combining the above estimates gives
\begin{align*}
 &\frac{a(z_0)\lambda^p+b(z_0)\lambda^q}{\lambda^2}\miint{G_{2\rho, \lambda}(z_0)}\frac{|u-(u)_{z_0, 2\rho, \lambda}|^2}{(2\rho)^2}\;dz \\
 &\quad \leq \varepsilon (a(z_0)\lambda^p+b(z_0)\lambda^q)+c\left[\left(\miint{G_{2 \rho,\lambda}\left(z_0\right)} H(z, |Du|)^{\theta}\, dz\right)^{\frac{1}{\theta}}+\miint{G_{2\rho, \lambda}(z_0)}H(z, |F|)\, dz\right] . 
\end{align*}
Finally, using the fourth condition of \eqref{Eq 3.8}-\eqref{Eq 3.9} and absorbing $\varepsilon (a(z_0)\lambda^p+b(z_0)\lambda^q)$ to the left-hand side, we complete the proof of this lemma.
\end{proof}

\section{Gradient higher integrability}\label{Sec:Higher_Integrability} In this section, our goal is to prove \cref{main_thm}. Note that \cref{main_thm} is stated in terms of $H_1$, as $H_1$ exhibits $(p,q)$-growth, as established in \cref{lem : bounds of H_1}. However, so far, the reverse Hölder inequalities have only been derived for $H$. Extending these inequalities to $H_1$ is straightforward, which is the content of our next lemma. Below, we present the proof of the $p$-intrinsic case, while the $q$-intrinsic and $(p,q)$-intrinsic cases follow in the same way.

\begin{lemma}[Reverse H\"{o}lder inequality in terms of $H_1$]\label{lem: reverse holder H_1}
Let $u$ be a weak solution to \eqref{eq: main equation}. Then for $Q_{c_v(2 \rho),\lambda}\left(z_0\right) \subset \Omega_T$ with \eqref{eq: p-phase condition}, there exist constants $c=c(\textnormal{\texttt{data}}) \geq 1$ and $\theta_0=\theta_0(n,p,q) \in (0,1)$ such that for any $\theta\in (\theta_0,1),$
\begin{align}\label{r holder h_1}
\miint{Q_{\rho,\lambda}(z_0)} H_1(z,|Du|)\; dz \leq c \left(\miint{Q_{2\rho,\lambda}(z_0)} [H_1(z,|Du|)]^\theta\;dz\right)^\frac{1}{\theta}+c\miint{Q_{2\rho,\lambda}(z_0)} H_1(z,|F|)\; dz.
\end{align}
\end{lemma}
\begin{proof}
 From \cref{lem : lemma 5.3}, we have 
 \begin{align*}
\miint{Q_{\rho,\lambda}(z_0)} H(z,|Du|)\; dz \leq c \left(\miint{Q_{2\rho,\lambda}(z_0)} [H(z,|Du|)]^\theta\;dz\right)^\frac{1}{\theta}+c\miint{Q_{2\rho,\lambda}(z_0)} H(z,|F|)\; dz.
\end{align*}
Adding $1$ on both sides of the above inequality and using the fact that $$0\leq H(z, |Du|)\leq H_1(z, |Du|),$$ we get the desired estimate \eqref{r holder h_1}.
\end{proof}
\subsection{Stopping time argument}\label{subsec:stopping_time}
Recall that $$H_1(z,\kappa) = a(z)\kappa^p+b(z)\kappa^q+1$$ and we set
\begin{equation}    \label{eq: definition of lambda_0}
    \lambda_0^2:=\miint{Q_{2r}(z_0)} [H_1(z,|Du|)+H_1(z,|F|)] \; dz
\end{equation}
and
\begin{equation}    \label{eq: definition of Lambda_0}
    \Lambda_0:=\sup_{z\in Q_{2r}(z_0)}a(z)\lambda_0^p+\sup_{z\in Q_{2r}(z_0)}b(z)\lambda_0^q+1.
\end{equation}
Also, for a fixed $\delta>0$ satisfying $\delta \leq a(z)+b(z),$ we set
\begin{align}\label{defn of K}
    K:=2+\frac{120[b]_{\alpha}}{\delta} \left(\frac{1}{2|B_1|}\iint_{\Omega_T} [H(z,|Du|)+H(z,|F|)]\, dz\right)^{\frac{\gamma}{p}}
\end{align}
with $$\gamma:=\frac{\alpha p}{n+2}.$$ 
We define $c_v\geq c_*(\texttt{data})$ by 
\begin{equation}    \label{eq: definition of c_nu}
    c_v:=c_*(\texttt{data})K,
\end{equation}
where $c_*(\texttt{data})$ will be determined later.
For $\Lambda_1:=\Lambda+1, $ we define the superlevel sets
\begin{align}\label{Defn E_lambda}
E(\Lambda_1)=\left\{z \in \Omega_T:\, H_1\left(z, |Du(z)|\right)>\Lambda_1\right\}
\end{align}
and
\begin{align}\label{Defn Phi_lambda}
\Phi(\Lambda_1)=\left\{z \in \Omega_T:\, H_1 \left(z, |F(z)|\right)>\Lambda_1\right\}.
\end{align}
For $E(\Lambda_1),\; \Phi(\Lambda_1)$ and $\rho\in [r,2r]$, we write
\begin{align}\label{defn_E(Lambda, rho)}
E(\Lambda_1,\rho):=E(\Lambda_1)\cap Q_{\rho}(z_0)=\left\{z \in Q_{\rho}(z_0):\, H_1(z, |Du(z)|)>\Lambda_1)\right\}
\end{align}
and
\begin{align*}
\Phi(\Lambda_1,\rho):=\Phi(\Lambda_1)\cap Q_\rho(z_0)=\left\{z \in Q_{\rho}(z_0):\, H_1 (z, |F(z)|)>\Lambda_1)\right\}.
\end{align*}

Now, we discuss the stopping time argument. We want to mention that our plan differs slightly from that in \cite[Section 5.1]{2023_Gradient_Higher_Integrability_for_Degenerate_Parabolic_Double-Phase_Systems}. Here, we need to treat the $p$-intrinsic and $q$-intrinsic cases independently. To be  more precise, we obtain a stopping time radius $\rho_{\mathfrak z}>0$ from the $p$-intrinsic case and another stopping time radius $\rho_{q;\mathfrak z}>0$ from the $q$-intrinsic case. Finally, we find a smaller radius $\rho_{p,q;\mathfrak z}< \min\{\rho_{\mathfrak z}, \rho_{q;\mathfrak z}\}$ as a stopping time for the $(p,q)$-intrinsic case. 

Let $r\leq r_1< r_2\leq 2r$, and let
\begin{equation}    \label{eq: definition of Lambda}
    \Lambda_1 > \left(\frac{4c_v r}{r_2-r_1}\right)^\frac{q(n+2)}{2}\Lambda_0.
\end{equation}
For each $\mathfrak{z} \in E(\Lambda_1,r_1)$, we take $\lambda_\mathfrak{z}>0$ such that 
\begin{equation}    \label{eq: definition of lambda z}
\Lambda_1=a(\mathfrak{z})\lambda_\mathfrak{z}^p+b(\mathfrak{z})\lambda_\mathfrak{z}^q+1.
\end{equation}
For $\mathfrak{z} \in E(\Lambda_1,r_1)$ and $K>1$, we have the following possibilities:
\begin{align}\label{possibilities depending on lambda zeta}
    (1)\,\,\,\, Ka(\mathfrak z)\lambda^p_{\mfz}\geq b(\mfz)\lambda^q_{\mfz}, \quad (2)\,\,\,\,Ka(\mathfrak z)\lambda^p_{\mfz}\leq b(\mfz)\lambda^q_{\mfz}.
\end{align}
\subsubsection{Existence of $p$-intrinsic case \eqref{eq: p-phase condition}} We consider the first possibility of \eqref{possibilities depending on lambda zeta}. We claim that 
\begin{equation}    \label{eq: lambda bigger than 1}
    \lambda_\mathfrak{z}>\left(\frac{4c_v r}{r_2-r_1}\right)^\frac{n+2}{2}\lambda_0\; (\geq 1).
\end{equation}
Suppose not, then we have
\begin{align*}    
\Lambda_1=a(\mathfrak{z})\lambda_\mathfrak{z}^p+b(\mathfrak{z})\lambda_\mathfrak{z}^q+1
&\leq \left(\frac{4c_v r}{r_2-r_1}\right)^\frac{q(n+2)}{2}(a(\mathfrak{z})\lambda_0^p+b(\mathfrak{z})\lambda_0^q+1)\\
&\leq \left(\frac{4c_v r}{r_2-r_1}\right)^\frac{q(n+2)}{2}\Lambda_0,
\end{align*}
which contradicts \eqref{eq: definition of Lambda}. Thus, for $\tau\in[(r_2-r_1)/(2c_v),r_2-r_1)$, we have 
\begin{align}\nonumber
    &\miint{Q_{\tau,\lambda_\mathfrak{z}}(\mathfrak{z})} [H_1(z,|Du|)+H_1(z,|F|)]\; dz\\ \nonumber
    &\quad \leq a(\mathfrak{z})\lambda_\mathfrak{z}^{p-2}\left(\frac{2r}{\tau}\right)^{n+2}\miint{Q_{2r}(z_0)} [H_1(z,|Du|)+H_1(z,|F|)]\; dz\\ \label{eq : basic inequality 1 in section 5}
    &\quad \leq \left(\frac{4c_v r}{r_2-r_1}\right)^{n+2} a(\mathfrak{z})\lambda_\mathfrak{z}^{p-2} \lambda_0^2 \overset{\eqref{eq: lambda bigger than 1}}{<}a(\mathfrak{z})\lambda_{\mathfrak{z}}^p <a(\mathfrak{z})\lambda_{\mathfrak{z}}^p+2.
\end{align}
The facts $\mathfrak{z}\in E(\Lambda_1,r_1)$ and \eqref{eq: definition of lambda z} imply that $\mathfrak{z}\in E(a(\mathfrak{z})\lambda_\mathfrak{z}^p+1,r_1)$. We define a continuous function $\Phi: (0, r_2-r_1]\to \mathbb{R}$ as
\begin{align*}
  \Phi (\tau)=  \miint{Q_{\tau,\lambda_\mathfrak{z}}(\mathfrak{z})} [H_1(z,|Du|)+H_1(z,|F|)]\; dz.
\end{align*}
By the Lebesgue differentiation theorem, we have 
\begin{align*}
    \lim_{\tau \to 0}\Phi(\tau)=H_1(\mathfrak z, |Du(\mathfrak z)|)+H_1(\mathfrak z, |F(\mathfrak z)|)>a(\mathfrak z)\lambda^p_{\mathfrak z}+2.
\end{align*}
Therefore, by intermediate value property of continuous function, there exists $\rho_\mathfrak{z}\in (0,(r_2-r_1)/(2c_v))$ such that 
\begin{equation}    \label{eq: choice condition 1 of rho}
    \miint{Q_{\rho_\mathfrak{z},\lambda_\mathfrak{z}}(\mathfrak{z})} [H_1(z,|Du|)+H_1(z,|F|)]\; dz=a(\mathfrak{z})\lambda_\mathfrak{z}^p+2
\end{equation}
holds. Moreover, from \eqref{eq : basic inequality 1 in section 5} we conclude
\begin{equation}    \label{eq: choice condition 2 of rho}
    \miint{Q_{\tau,\lambda_\mathfrak{z}}(\mathfrak{z})} [H_1(z,|Du|)+H_1(z,|F|)]\; dz< a(\mathfrak{z})\lambda_\mathfrak{z}^p+2
\end{equation}
for every $\tau\in(\rho_\mathfrak{z},r_2-r_1)$. From \eqref{eq: choice condition 2 of rho} and \eqref{eq: choice condition 1 of rho}, we get the third and fourth conditions of \eqref{eq: p-phase condition}.
Now for this $\rho_{\mfz}>0,$ we have four possibilities:
\begin{itemize}
\item [(i)] $\,\,\,a(\mathfrak{z})\geq 2\omega_a(\rho_{\mathfrak{z}}),\,\,\, b(\mathfrak{z})\geq 2[b]_{\alpha}\left(10 \rho_{\mathfrak{z}}\right)^{\alpha},$\\
\item [(ii)]  $\,\,\, a(\mathfrak{z})\geq 2\omega_a(\rho_{\mathfrak{z}}),\,\,\, b(\mathfrak{z})\leq 2[b]_{\alpha}\left(10 \rho_{\mathfrak{z}}\right)^{\alpha},$\\
\item [(iii)]  $ \,\,\, a(\mathfrak{z})\leq 2\omega_a(\rho_{\mathfrak{z}}),\,\,\, b(\mathfrak{z})\geq 2[b]_{\alpha}\left(10 \rho_{\mathfrak{z}}\right)^{\alpha},$\\
\item [(iv)]  $\,\,\, a(\mathfrak{z})\leq 2\omega_a(\rho_{\mathfrak{z}}),\,\,\, b(\mathfrak{z})\leq 2[b]_{\alpha}\left(10 \rho_{\mathfrak{z}}\right)^{\alpha}.$\\
\end{itemize}
We show that the case (ii) implies $p$-intrinsic case \eqref{eq: p-phase condition}. For small $\rho_{\mathfrak{z}},$ using the second condition of (ii), we get
\begin{align*}
    \delta \leq a(\mathfrak{z})+b(\mathfrak{z})\leq a(\mathfrak{z})+2[b]_{\alpha}(10 \rho_{\mathfrak{z}})^{\alpha}<a(\mathfrak{z})+\frac{\delta}{2},
\end{align*}
we conclude that $a(\mathfrak{z})\geq \frac{\delta}{2}$ which implies first part of the second condition of \eqref{eq: p-phase condition} replacing the center point. We note that, the smallness of  $\rho_{\mathfrak{z}}$ will be guaranteed later, for instance, we can choose large $c_v$ given by \eqref{eq: definition of c_nu} to get $\rho_{\mathfrak{z}}$ as small as we want. Moreover, from the first condition of (ii), for any $z \in Q_{\rho_{\mathfrak{z}}, \lambda_{\mathfrak{z}}}(\mathfrak{z})$ we have
\begin{align*}
    a(z)\leq |a(z)-a(\mathfrak{z})|+a(\mathfrak{z})\leq \omega_a(\rho_{\mathfrak{z}})+a(\mathfrak{z})\leq \frac{1}{2}a(\mathfrak{z})+a(\mathfrak{z})\leq 3 a(\mathfrak{z}).
\end{align*}
Similarly, we get $a(\mathfrak{z})\leq 3 a(z)$ for any $z \in  Q_{\rho_{\mathfrak{z}}, \lambda_{\mathfrak{z}}}(\mathfrak{z}).$ This implies the second part of the second condition of \eqref{eq: p-phase condition} replacing the center and the radius by $\mathfrak z$ and $\rho_{\mathfrak z}$ respectively. 

It is easy to see from the condition $\delta \leq a(\mathfrak{z})+b(\mathfrak{z})$ that (iv) cannot occur.  Also, (iii) is not possible for the following reason: in this case $a(\mathfrak{z})\leq 2\omega_a(\rho_{\mathfrak{z}})$ and for small $\rho_{\mathfrak{z}},$ it can be assumed that $\omega_a(\rho_{\mathfrak{z}})<\frac{\delta}{4K}$. Thus, we have $a(\mathfrak{z}) < \frac{\delta}{2K}$. From this, we see that $b(\mathfrak{z}) \geq \frac{2K-1}{2K}\delta$, and hence $$\frac{(2K-1)\delta}{2K}\lambda_{\mathfrak{z}}^q \leq b(\mathfrak{z})\lambda_{\mathfrak{z}}^q\leq Ka(\mathfrak{z})\lambda_{\mathfrak{z}}^p\leq 2K\omega_a(\rho_{\mathfrak{z}})\lambda_{\mathfrak{z}}^p.$$ 
Since for small $\rho_{\mathfrak{z}} >0,$ we have $\omega_a(\rho_{\mathfrak{z}})< \frac{\delta}{4K}, $ then we arrive at $$\frac{(2K-1)\delta}{2K}< \frac{\delta}{2}\lambda_{\mathfrak{z}}^{p-q}\leq \frac{\delta}{2},$$ which is a contradiction as this gives $K<1.$ This completes the existence of  $p$-intrinsic case.

\subsubsection{Existence of $q$-intrinsic case \eqref{eq: q-phase condition}} 
We consider the second possibility of \eqref{possibilities depending on lambda zeta}. Since \eqref{eq: lambda bigger than 1} holds for $\tau\in[(r_2-r_1)/(2c_v),r_2-r_1)$, we have 
\begin{align*}\nonumber
    &\miint{J_{\tau,\lambda_\mathfrak{z}}(\mathfrak{z})} [H_1(z,|Du|)+H_1(z,|F|)]\; dz\\ \nonumber
    &\quad \leq b(\mathfrak{z})\lambda_\mathfrak{z}^{q-2}\left(\frac{2r}{\tau}\right)^{n+2}\miint{Q_{2r}(z_0)} [H_1(z,|Du|)+H_1(z,|F|)]\; dz\\ \label{eq : basic inequality 1 in section 5}
    &\quad \leq \left(\frac{4c_v r}{r_2-r_1}\right)^{n+2} b(\mathfrak{z})\lambda_\mathfrak{z}^{q-2} \lambda_0^2 \overset{\eqref{eq: lambda bigger than 1}}{<}b(\mathfrak{z})\lambda_{\mathfrak{z}}^q<b(\mathfrak{z})\lambda_{\mathfrak{z}}^q+2.
\end{align*}
The facts $\mathfrak{z}\in E(\Lambda_1,r_1)$ and \eqref{eq: definition of lambda z} imply that $\mathfrak{z}\in E(b(\mathfrak{z})\lambda_\mathfrak{z}^q+1,r_1)$. By the Lebesgue differentiation theorem, there exists $\rho_{q; \, \mathfrak{z}}\in (0,(r_2-r_1)/(2c_v))$ such that 
\begin{equation}    \label{eq: choice condition 1 of q phase}
    \miint{J_{\rho_{q; \mathfrak{z}},\lambda_\mathfrak{z}}(\mathfrak{z})} [H_1(z,|Du|)+H_1(z,|F|)]\; dz=b(\mathfrak{z})\lambda_\mathfrak{z}^q+2
\end{equation}
and
\begin{equation}    \label{eq: choice condition 2 of q phase}
    \miint{J_{\tau,\lambda_\mathfrak{z}}(\mathfrak{z})} [H_1(z,|Du|)+H_1(z,|F|)]\; dz< b(\mathfrak{z})\lambda_\mathfrak{z}^q+2
\end{equation}
for every $\tau\in(\rho_{q; \mathfrak{z}},r_2-r_1).$ From \eqref{eq: choice condition 1 of q phase} and \eqref{eq: choice condition 2 of q phase}, we get the third and fourth conditions of \eqref{eq: q-phase condition} by replacing the center and radius with $\mathfrak{z}$ and $\rho_{q; \mathfrak{z}}.$

Now that the radius $\rho_{q; \mfz}>0$ is defined, we have four possibilities.
\begin{itemize}
\item [(i)] $\,\,\,a(\mathfrak{z})\geq 2\omega_a(\rho_{q; \mathfrak{z}}),\,\,\, b(\mathfrak{z})\geq 2[b]_{\alpha}\left(10 \rho_{q; \mathfrak{z}}\right)^{\alpha},$\\
\item [(ii)]  $\,\,\, a(\mathfrak{z})\geq 2\omega_a(\rho_{q; \mathfrak{z}}),\,\,\, b(\mathfrak{z})\leq 2[b]_{\alpha}\left(10 \rho_{q; \mathfrak{z}}\right)^{\alpha},$\\
\item [(iii)]  $ \,\,\, a(\mathfrak{z})\leq 2\omega_a(\rho_{q; \mathfrak{z}}),\,\,\, b(\mathfrak{z})\geq 2[b]_{\alpha}\left(10 \rho_{q; \mathfrak{z}}\right)^{\alpha},$\\
\item [(iv)]  $\,\,\, a(\mathfrak{z})\leq 2\omega_a(\rho_{q; \mathfrak{z}}),\,\,\, b(\mathfrak{z})\leq 2[b]_{\alpha}\left(10 \rho_{q; \mathfrak{z}}\right)^{\alpha}.$\\
\end{itemize}
Now we show that (iii) implies $q$-intrinsic case \eqref{eq: q-phase condition}. Since, for small $\rho_{q; \mathfrak{z}}$ we may assume $\omega_a(\rho_{q; \mathfrak{z}})\leq  \frac{\delta}{4},$ from the first condition of (iii), we get
\begin{align*}
    \delta \leq a(\mathfrak{z})+b(\mathfrak{z})\leq 2\omega_a(\rho_{q; \mathfrak{z}})+b(\mathfrak{z})\leq \frac{\delta}{2}+b(\mathfrak{z}),
\end{align*}
which gives us the lower bound of $b(z_0)$ in \eqref{eq: q-phase condition} replacing the center with $\mathfrak{z}.$ The second condition of (iii) implies the second part of the second condition of \eqref{eq: q-phase condition}. Indeed, for any $z \in Q_{10 \rho_{q; \mathfrak{z}}}(\mathfrak{z}),$ we have
\begin{align*}
b(z)\leq |b(z)-b(\mathfrak{z})|+b(\mathfrak{z})&\leq [b]_{\alpha}(10 \rho_{q; \mathfrak{z}})^{\alpha}+b(\mathfrak{z})\\
&\leq \frac{b(\mathfrak{z})}{2}+b(\mathfrak{z})\leq 2 b(\mathfrak{z}).
\end{align*}
We note that (iv) cannot occur because $a(\mfz)+b(\mfz)\geq \delta>0.$ We show (ii) cannot occur. First, we note that using the second condition of (ii), we get $a(\mathfrak{z})\geq \frac{\delta}{2}$ as before. It follows from \eqref{eq: choice condition 1 of q phase} that

\begin{align*}
b(\mathfrak{z})\lambda^q_{\mathfrak{z}}=\miint{J_{\rho_{q; \mathfrak{z}},\lambda_\mathfrak{z}}(\mathfrak{z})} [H(z,|Du|)+H(z,|F|)]\; dz < \left(\frac{K \delta}{120[b]_{\alpha}}\right)^{\frac{n+2}{\alpha}}\frac{1}{|J_{\rho_{q; \mathfrak{z}},\lambda_\mathfrak{z}}|}
=\left(\frac{K \delta}{120[b]_{\alpha}}\right)^{\frac{n+2}{\alpha}}\frac{b(\mathfrak{z})\lambda^{q-2}_{\mathfrak z}}{\rho^{n+2}_{q; \mathfrak z}}.
\end{align*}

This implies
\begin{align*}
    \rho^{\alpha}_{q; \mathfrak{z}}\lambda^q_{\mathfrak{z}} < \frac{K \delta}{120[b]_{\alpha}}\lambda^{q-\frac{2\alpha}{n+2}}_{\mathfrak{z}},
\end{align*}
and  using the range of $q$, given in \eqref{def_pq}, i.e., $q\leq p+\frac{2\alpha}{n+2}$, we get
\begin{align}\label{estimate on rho alpha}
    \rho^{\alpha}_{q; \mathfrak{z}}\lambda^q_{\mathfrak{z}}< \frac{K}{40[b]_{\alpha}}\frac{\delta}{3}\lambda^{p}_{\mathfrak{z}}.
\end{align}
Now using second possibility of \eqref{possibilities depending on lambda zeta} and second condition of (ii), we conclude
\begin{align}\label{EQQ4.17}
    Ka(\mathfrak{z})\lambda^p_{\mathfrak{z}}\leq b(\mathfrak{z})\lambda^q_{\mathfrak{z}}\leq 2[b]_{\alpha}(10\rho_{q; \mathfrak{z}})^{\alpha}\lambda^q_{\mathfrak{z}}< \frac{K10^\alpha}{20}\frac{\delta}{3}\lambda^p_{\mathfrak{z}}.
\end{align}
On the other hand, we have
\begin{align}\label{EQQ4.18}
    K a(\mathfrak{z})\lambda^p_{\mathfrak{z}}\geq \frac{K\delta}{2}\lambda^p_{\mathfrak{z}}.
\end{align}
Combining the inequalities \eqref{EQQ4.17} and \eqref{EQQ4.18} gives
\begin{align*}
    \frac{K \delta}{2}\lambda^p_{\mathfrak{z}}< \frac{K\delta}{6}\lambda^p_{\mathfrak{z}},
\end{align*}
which is a contradiction.

Note that so far we have not considered the possibilities
\begin{align*}
 \,\,\,a(\mathfrak{z})\geq 2\omega_a(\rho_{\mathfrak{z}}),\,\,\, b(\mathfrak{z})\geq 2[b]_{\alpha}\left(10 \rho_{\mathfrak{z}}\right)^{\alpha},\quad \text{and} \quad a(\mathfrak{z})\geq 2\omega_a(\rho_{q; \mathfrak{z}}),\,\,\, b(\mathfrak{z})\geq 2[b]_{\alpha}\left(10 \rho_{q; \mathfrak{z}}\right)^{\alpha},  
\end{align*}
appearing in $p$- and $q$-intrinsic cases and now these possibilities are useful to define $(p, q)$-intrinsic case. Define $$\widetilde \rho_{ \mathfrak z}=\min\{\rho_{\mathfrak z}, \rho_{q; \mathfrak z}\}.$$ From the above two possibilities, we have 
\begin{align*}
a(\mathfrak{z})\geq 2\omega_a(\widetilde \rho_{\mathfrak{z}}),\quad  b(\mathfrak{z})\geq 2[b]_{\alpha}\left(10 \widetilde \rho_{\mathfrak{z}}\right)^{\alpha}.   
\end{align*}

\subsubsection{Existence of $(p,q)$-intrinsic case \eqref{Eq 3.8}--\eqref{Eq 3.9}} \label{deduction of pq case}
In this case, we consider, any of the possibilities (1)
or (2) of \eqref{possibilities depending on lambda zeta}, i.e., 
\begin{align*}
 Ka(\mathfrak z)\lambda^p_{\mfz}\geq b(\mfz)\lambda^q_{\mfz}, \quad \text{or}\quad Ka(\mathfrak z)\lambda^p_{\mfz}\leq b(\mfz)\lambda^q_{\mfz}   
\end{align*}
and 
\begin{align}\label{eq: pq phase assmp on ab}
a(\mathfrak{z})\geq 2\omega_a(\widetilde \rho_{\mathfrak{z}}),\quad b(\mathfrak{z})\geq 2[b]_{\alpha}\left(10 \widetilde \rho_{ \mathfrak{z}}\right)^{\alpha}.   
\end{align}
Since $G_{\tau, \lambda_{\mathfrak z}}(\mathfrak z)\subset Q_{\tau, \lambda_{\mathfrak z}}(\mathfrak z)$ and $G_{\tau, \lambda_{\mathfrak z}}(\mathfrak z)\subset J_{\tau, \lambda_{\mathfrak z}}(\mathfrak z),$ from \eqref{eq: choice condition 2 of rho} and \eqref{eq: choice condition 2 of q phase}, we obtain
\begin{align*}
    \miint{G_{\tau, \lambda_{\mathfrak z}}(\mathfrak z)}H(z, |Du|)+H(z, |F|)\, dz&\leq \frac{a(\mathfrak z)\lambda^p_{\mathfrak z}+b(\mathfrak z)\lambda^q_{\mathfrak z}}{a(\mathfrak z)\lambda^p_{\mathfrak z}}\miint{Q_{\tau, \lambda_{\mathfrak z}}(\mathfrak z)}H(z, |Du|)+H(z, |F|)\, dz\\
    &< a(\mathfrak z)\lambda^p_{\mathfrak z}+b(\mathfrak z)\lambda^q_{\mathfrak z}\quad \text{for every}\quad \tau \in (\rho_{\mathfrak z}, r_2-r_1).
\end{align*}
This gives
\begin{align*}
\miint{G_{\tau, \lambda_{\mathfrak z}}(\mathfrak z)}H_1(z, |Du|)+H_1(z, |F|)\, dz < a(\mathfrak z)\lambda^p_{\mathfrak z}+b(\mathfrak z)\lambda^q_{\mathfrak z}+2 \quad \text{for every}\quad \tau \in (\rho_{\mathfrak z}, r_2-r_1).   
\end{align*}
On the other hand, we have
\begin{align*}
    \miint{G_{\tau, \lambda_{\mathfrak z}}(\mathfrak z)}H(z, |Du|)+H(z, |F|)\, dz&\leq \frac{a(\mathfrak z)\lambda^p_{\mathfrak z}+b(\mathfrak z)\lambda^q_{\mathfrak z}}{b(\mathfrak z)\lambda^q_{\mathfrak z}}\miint{J_{\tau, \lambda_{\mathfrak z}}(\mathfrak z)}H(z, |Du|)+H(z, |F|)\, dz\\
    &< a(\mathfrak z)\lambda^p_{\mathfrak z}+b(\mathfrak z)\lambda^q_{\mathfrak z}\quad \text{for every}\quad \tau \in (\rho_{q; \mathfrak z}, r_2-r_1).
\end{align*}
Similarly, this gives
\begin{align*}
\miint{G_{\tau, \lambda_{\mathfrak z}}(\mathfrak z)}H_1(z, |Du|)+H_1(z, |F|)\, dz < a(\mathfrak z)\lambda^p_{\mathfrak z}+b(\mathfrak z)\lambda^q_{\mathfrak z}+2 \quad \text{for every}\quad \tau \in (\rho_{q; \mathfrak z}, r_2-r_1). 
\end{align*}
Since $\mathfrak z \in E(\Lambda_1, r_1)$ and $\Lambda_1=a(\mathfrak z)\lambda^p_{\mathfrak z}+b(\mathfrak z)\lambda^q_{\mathfrak z}+1,$ from the above two assertions, we get $\rho_{p, q; \mathfrak z}\in (0, \widetilde \rho_{\mathfrak z}]$ such that
\begin{equation}    \label{eq: choice condition 1 of pq phase}
    \miint{G_{\rho_{p,q; \mathfrak{z}},\lambda_\mathfrak{z}}(\mathfrak{z})} [H_1(z,|Du|)+H_1(z,|F|)]\; dz=a(\mathfrak{z})\lambda^p_{z}+b(\mathfrak{z})\lambda_\mathfrak{z}^q+2
\end{equation}
and
\begin{equation}    \label{eq: choice condition 2 of pq phase}
    \miint{G_{\tau,\lambda_\mathfrak{z}}(\mathfrak{z})} [H_1(z,|Du|)+H_1(z,|F|)]\; dz< a(\mathfrak{z})\lambda^p_{z}+b(\mathfrak{z})\lambda_\mathfrak{z}^q+2
\end{equation}
for every $\tau\in(\rho_{p, q; \mathfrak z},r_2-r_1).$

Also, from \eqref{eq: pq phase assmp on ab}, we get $a(z)\leq 3 a(\mathfrak z)$ for all $z \in Q_{\widetilde \rho_{\mathfrak z}, \lambda_{\mathfrak z}}(\mathfrak z)$ and $b(z)\leq 3 b(\mathfrak z)$ for all $z\in J_{\widetilde \rho_{\mathfrak z}, \lambda_{\mathfrak z}}(\mathfrak z).$ Since, by definition, $G_{\rho_{p, q; \mathfrak z}, \lambda_{\mathfrak z}}(\mathfrak z)\subset Q_{\widetilde \rho_{\mathfrak z}, \lambda_{\mathfrak z}}(\mathfrak z)\cap J_{\widetilde \rho_{q; \mathfrak z}, \lambda_{\mathfrak z}}(\mathfrak z),$ we get the second assumption of \eqref{Eq 3.8}-\eqref{Eq 3.9} by replacing the center and the radius by $\mathfrak z$ and $\rho_{p,q; \mathfrak z}.$

\subsection{Vitali covering lemma}\label{subsec:Vitali}

In Section \ref{Sec:Reverse_H} and Subsection \ref{subsec:stopping_time}, we considered three distinct intrinsic cases to obtain reverse H\"{o}lder inequalities.
In this subsection, we establish a Vitali type covering lemma for the following parabolic intrinsic cylinders: For each $\mathfrak{z}\in E(\Lambda,r_1)$, we denote the intrinsic cylinder of $\mathfrak{z}$ by

\begin{align*}
\displaystyle
\mathcal{Q}(\mathfrak{z})=\left\{\begin{array}{l}
Q_{2\rho_\mathfrak{z},\lambda_\mathfrak{z}}(\mathfrak{z})\quad \,\,\,\,\,\,\,\text{in $p$-intrinsic case},\\
J_{2\rho_{q; \mathfrak{z}},\lambda_\mathfrak{z}}(\mathfrak{z})\quad \,\,\,\,\,\,\text{in $q$-intrinsic case},\\
G_{2\rho_{p, q;\mathfrak{z}},\lambda_\mathfrak{z}}(\mathfrak{z})\quad \,\text{in $(p, q)$-intrinsic case}.
\end{array}\right.
\end{align*}
Let us define
\begin{align*}
\mathcal{F}=\left\{\mathcal{Q}(\mathfrak{z})\, : \,\mathfrak{z}\in E(\Lambda_1,r_1)\right\}\quad \text{and} \quad \ell(\mathcal{Q}(\mathfrak{z}))=\left\{
    \begin{array}{l }
        2\rho_\mathfrak{z} \quad\,\,\,\,\,\, \text{in $p$-intrinsic case},\\
        2\rho_{q; \mathfrak{z}} \quad\,\,\, \text{in $q$-intrinsic case},\\
        2\rho_{p,q;\mathfrak{z}} \quad \text{in $(p,q)$-intrinsic case}.
    \end{array}
\right.
\end{align*} 

Also, for $\mathfrak{z}, \mathfrak{w}\in E(\Lambda_1, r_1),$ there exist $\lambda_i\geq1$ and $\rho_i>0$ such that
\begin{align*}
    \Lambda_1=a(i)\lambda^p_{i}+b(i)\lambda^q_{i}+1, \quad \text{where} \ \, i\in \left\{\mathfrak{z}, \mathfrak{w}\right\},
\end{align*}
and 
\begin{equation}\label{eq : mean integral of H equals to lambda_p}
    \miint{Q_{\rho_{i},\lambda_i}(i)}\left[H(z,|Du|)+H(z,|F|)\right]\; dz=a(i)\lambda_i^p.
\end{equation}
We note from the previous subsection that $l(\mathcal{Q}(\mathfrak{z}))\in (0, R)$ by setting $R:=\frac{r_2-r_1}{c_v}$, where $c_v$ is given by \eqref{eq: definition of c_nu}. Now we consider the following subcollection of $\mathcal{F},$
\begin{align*}
  \mathcal{F}_j:=\left\{\mathcal{Q}(\mathfrak{z})\in \mathcal{F} : \frac{R}{2^j}< \ell({\mathcal{Q}(\mathfrak{z})})\leq \frac{R}{2^{j-1}}\right\}, \,\,\,\, j\in \mathbb{N}.
\end{align*}
We construct disjoint subcollections $\mathcal{I}_j \subset \mathcal{F}_j$ for $j \in \NN$ as follows. Let $\mathcal{I}_1$ be the maximal disjoint  subcollection in $\mathcal{F}_1.$ We note from the following
\begin{align*}
    \lim_{\La_1\to \infty}\La_1|E(\Lambda_1)|\leq \lim_{\La_1 \to \infty}\iint_{E(\Lambda_1)}H_1(z, |D u|)\, dz=0
\end{align*}
that $|E(\Lambda_1)|<\infty.$ 
Moreover, in the same way as in \eqref{eq : basic inequality 1 in section 5} and \eqref{eq: choice condition 1 of rho}, we have 
\begin{equation}    \label{eq : relation between lambda_z and lambda_0}
    \lambda_\mathfrak{z}\leq \left(\frac{2r}{\rho_\mathfrak{z}}\right)^{\frac{n+2}{2}}\lambda_0.
\end{equation}
Using $|E(\La_1, r_1)|\leq |E(\La_1)|< \infty$ and \eqref{eq : relation between lambda_z and lambda_0}, we conclude that the number of cylinders in $\mathcal{I}_1$ is finite. Suppose we have already chosen $\mathcal{I}_j \subset \mathcal{F}_j$ for $j=1,..,k-1.$ Then we construct $\mathcal{I}_k$ as
\begin{align*}
    \mathcal{I}_k=\left\{\mathcal{Q}(\mathfrak{z})\in \mcf_k \, : \, \mathcal{Q}(\mathfrak{z})\cap \mathcal{Q}(\mathfrak{w})=\emptyset\,\,\, \text{for every}\,\,\, \mathcal{Q}(\mathfrak{w}) \in \bigcup_{j=1}^{k-1}\mathcal{I}_j \right\}. 
\end{align*}
Therefore, 
$$
\mathcal{I}=\bigcup_{j=1}^{\infty}\mathcal{I}_j
$$
would be the maximal collection of pairwise disjoint cylinders in $\mcf.$ Now to prove Vitali's covering lemma, we need to show:
\begin{itemize}
\item[(i)] For any $\mathcal{Q}(\mathfrak{z}) \in \mcf,$ there exists $\mathcal{Q}(\mathfrak{w}) \in \mathcal{I}$ such that $\mathcal{Q}(\mathfrak{z}) \cap \mathcal{Q}(\mathfrak{w}) \neq \emptyset.$
\item[(ii)] There exists a universal constant $c_v>1$ such that $\mathcal{Q}(\mathfrak{z}) \subset c_v \mathcal{Q}(\mathfrak{w})$ where 
$$
c_v\mathcal{Q}(\mathfrak{w})=\left\{
    \begin{array}{l c}
        Q_{c_v(2\rho_\mathfrak{w}),\lambda_\mathfrak{w}}(\mathfrak{w})\,\,\,\,\,\, \quad \text{in $p$-intrinsic case},\\
        J_{c_v(2\rho_{q;\mathfrak{w}}),\lambda_\mathfrak{w}}(\mathfrak{w})\,\,\,\,\, \quad \text{in $q$-intrinsic case},\\
        G_{c_v(2\rho_{p,q;\mathfrak{w}}),\lambda_\mathfrak{w}}(\mathfrak{w}) \quad \text{in $(p,q)$-intrinsic case}.
    \end{array}
    \right.
$$
\end{itemize}
To show the first assertion, fix some $Q(\mathfrak{z}) \in \mcf.$ Then $Q(\mathfrak{z}) \in \mcf_j$ for some $j \in \NN.$ Using the maximality of $\mathcal{I}_j,$ we find that there exists $\mathcal{Q}(\mathfrak{w})\in \cup_{i=1}^{j}\mathcal{I}_i$ such that $\mathcal{Q}(\mathfrak{z})\cap \mathcal{Q}(\mathfrak{w})\neq \emptyset.$ 

To show the second assertion, we first observe that for any $\mathcal{Q}(\mathfrak{w}) \in \mathcal{I}_i,$ with $i\leq j,$ we have
\begin{align}\label{radius compare}
    l(\mathcal{Q}(\mathfrak{z}))\leq 2l(\mathcal{Q}(\mathfrak{w})).
\end{align}
This is due to the fact that
$\frac{R}{2^j}\leq \frac{R}{2^i}<\ell(\mathcal Q(\mathfrak w)) \leq \frac{R}{2^{i-1}}$ for $i\leq j$, and $\frac{R}{2^j}<\ell(\mathcal Q(\mathfrak z)) \leq \frac{R}{2^{j-1}}$. 

In the rest of the subsection, we show the existence of the Vitali constant $c_v>1$ such that $\mathcal{Q}(\mathfrak{z}) \subset c_v \mathcal{Q}(\mathfrak{w})$ by considering the following nine distinct cases depicted in the table below.
\begin{table}[ht]
\centering
\begin{tabular}{|c|c c c|}
\hline
\diagbox{$\mathcal{Q}(\mathfrak{z})$}{$\mathcal{Q}(\mathfrak{w})$} & $Q_{2\rho_{\mathfrak{w}},\lambda_{\mathfrak{w}}}(\mathfrak{w})$ & $J_{2\rho_{q;\mathfrak{w}},\lambda_{\mathfrak{w}}}(\mathfrak{w})$ & $G_{2\rho_{p,q; \mathfrak{w}},\lambda_{\mathfrak{w}}}(\mathfrak{w})$\\ 
\hline
$Q_{2\rho_{\mathfrak{z}},\lambda_\mathfrak{z}}(\mathfrak{z})$ & (1-1) & (1-2) & (1-3) \\
        
$J_{2\rho_{q;\mathfrak{z}},\lambda_\mathfrak{z}}(\mathfrak{z})$ & (2-1) & (2-2) & (2-3) \\
        
$G_{2\rho_{p,q; \mathfrak{z}},\lambda_\mathfrak{z}}(\mathfrak{z})$ & (3-1) & (3-2) & (3-3) \\
\hline
\end{tabular}
\medskip

\caption{The combinations of $\mathcal{Q}(\mathfrak{z})$ and $\mathcal{Q}(\mathfrak{w})$.}
\label{tab : 9 Cases} 
\end{table}

\noindent \textit{Claim:} $\lambda_{\mathfrak{z}}$ and $\lambda_{\mathfrak{w}}$ are comparable.

\noindent Since $G_{2\rho_{p, q; (\cdot)}, \lambda_{(\cdot)}}(\cdot)\subset Q_{2\rho_{ (\cdot)}, \lambda_{(\cdot)}}(\cdot)$ and $G_{2\rho_{p, q; (\cdot)}, \lambda_{(\cdot)}}(\cdot)\subset J_{2\rho_{q; (\cdot)}, \lambda_{(\cdot)}}(\cdot),$ it is enough to consider cases (1-1), (1-2), (2-1), (2-2). Let $\lambda_{\mathfrak{z}}\leq \lambda_{\mathfrak{w}}.$ We will show that $\lambda_{\mathfrak{w}}\leq (2K)^{\frac{1}{p}}\lambda_{\mathfrak{z}}.$

\noindent \textit{Case (1-1):} First, we note from the first assertion that $Q_{2\rho_{\mathfrak{z}}, \lambda_{\mathfrak{z}}}(\mathfrak{z})\cap Q_{2\rho_{\mathfrak{w}}, \lambda_{\mathfrak{w}}}(\mathfrak{w})\neq \emptyset.$ Using the H\"{o}lder continuity of $b(\cdot)$ and \eqref{radius compare}, we get
\begin{align*}
    |b(\mathfrak{z})-b(\mathfrak{w})|\leq [b]_{\alpha}\left(2\rho^{\alpha}_{\mathfrak{z}}+2\rho^{\alpha}_{\mathfrak{w}}\right)\leq 6[b]_{\alpha}\rho^{\alpha}_{\mathfrak{w}},
\end{align*}
and from the uniform continuity of $a(\cdot),$ we have
\begin{align*}
    |a(\mathfrak{z})-a(\mathfrak{w})|\leq \omega_a(\rho_{\mathfrak{z}})+\omega_a(\rho_{\mathfrak{w}}).
\end{align*}
On the contrary, assume $\lambda_{\mathfrak{w}}>(2K)^{\frac{1}{p}}\lambda_{\mathfrak{z}}.$ Moreover, using \eqref{eq : mean integral of H equals to lambda_p}, we can derive
\begin{align}\label{estimate on rho alpha2.0}
    \rho^{\alpha}_{\mathfrak w}\lambda^q_{\mathfrak w}\leq \frac{K}{40 [b]_{\alpha}}\frac{\delta}{3}\lambda^p_{\mathfrak w}
\end{align}
similar to the estimate \eqref{estimate on rho alpha}.
Now we use the counter assumption, \eqref{estimate on rho alpha2.0}, $a(\mathfrak{w})\geq \frac{\delta}{2},$ and  $\lambda_{\mathfrak{z}}\leq \lambda_{\mathfrak{w}}$ to obtain
\begin{align*}
\Lambda_1=a(\mathfrak{z})\lambda^p_{\mathfrak{z}}+b(\mathfrak{z})\lambda^q_{\mathfrak{z}}+1&\leq |a(\mathfrak{z})-a(\mathfrak{w})|\lambda^p_{\mathfrak{z}}+a(\mathfrak{w})\lambda^p_{\mathfrak{z}}+|b(\mathfrak{z})-b(\mathfrak{w})|\lambda^q_{\mathfrak{z}}+b(\mathfrak{w})\lambda^q_{\mathfrak{z}}+1\\
&\leq \left(\omega_{a}(\rho_{\mathfrak{z}})+\omega_{a}(\rho_{\mathfrak{w}})\right)\lambda^p_{\mathfrak{z}}+ a(\mathfrak{w})\lambda^p_{\mathfrak{z}}+ 6[b]_{\alpha}\rho^{\alpha}_{\mathfrak{w}}\frac{\lambda^q_{\mathfrak{w}}}{(2K)^{q/p}}+b(\mathfrak{w})\lambda^q_{\mathfrak{w}}+1\\ &\overset{\eqref{estimate on rho alpha2.0}}{\leq} \frac{K \delta}{10}\lambda^p_{\mathfrak{z}}+ a(\mathfrak{w})\lambda^p_{\mathfrak{z}}+\frac{1}{(2K)^{q/p-1}}\frac{\delta}{20}\lambda^p_{\mathfrak{w}}+b(\mathfrak{w})\lambda^q_{\mathfrak{w}}+1\\
&< \frac{K}{5}a(\mathfrak{w})\frac{\lambda^p_{\mathfrak{w}}}{2K}+a(\mathfrak{w})\frac{\lambda^p_{\mathfrak{w}}}{2K}+\frac{1}{(2K)^{q/p-1 }10}a(\mathfrak{w})\lambda^p_{\mathfrak{w}}+b(\mathfrak{w})\lambda^q_{\mathfrak{w}}+1\\
&=\left(\frac{1}{10}+\frac{1}{2K}+\frac{1}{(2K)^{q/p-1}10}\right)a(\mathfrak{w})\lambda^p_{\mathfrak{w}}+b(\mathfrak{w})\lambda^q_{\mathfrak{w}}+1\\
&<\frac{7}{10}a(\mathfrak{w})\lambda^p_{\mathfrak{w}}+b(\mathfrak{w})\lambda^q_{\mathfrak{w}}+1<a(\mathfrak{w})\lambda^p_{\mathfrak{w}}+b(\mathfrak{w})\lambda^q_{\mathfrak{w}}+1=\Lambda_1,
\end{align*}
which is a contradiction.

\noindent \textit{Case (2-1):} Since, by \eqref{radius compare} we have $\rho_{q; \mathfrak{z}}\leq 2\rho_{\mathfrak{w}},$ we can conclude this case by following the same argument as above.

\noindent \textit{Case (2-2):} Following the previous argument together with \eqref{radius compare} and $b(\mathfrak{w})\geq \frac{\delta}{2},$, we obtain 

\begin{align*}
\Lambda_1=a(\mathfrak{z})\lambda^p_{\mathfrak{z}}+b(\mathfrak{z})\lambda^q_{\mathfrak{z}}+1&\leq |a(\mathfrak{z})-a(\mathfrak{w})|\lambda^p_{\mathfrak{z}}+a(\mathfrak{w})\lambda^p_{\mathfrak{z}}+|b(\mathfrak{z})-b(\mathfrak{w})|\lambda^q_{\mathfrak{z}}+b(\mathfrak{w})\lambda^q_{\mathfrak{z}}+1\\
&<\left(\omega_{a}(\rho_{q; \mathfrak{z}})+\omega_{a}(\rho_{q; \mathfrak{w}})\right)\lambda^p_{\mathfrak{z}}+ a(\mathfrak{w})\lambda^p_{\mathfrak{z}}+ 6[b]_{\alpha}\rho^{\alpha}_{q;\mathfrak{w}}\frac{\lambda^q_{\mathfrak{w}}}{(2K)^{q/p}}+b(\mathfrak{w})\frac{\lambda^q_{\mathfrak{w}}}{(2K)^{\frac{q}{p}}}+1\\ &\overset{\eqref{estimate on rho alpha}}{\leq} \frac{K \delta}{10}\lambda^p_{\mathfrak{z}}+ a(\mathfrak{w})\lambda^p_{\mathfrak{z}}+\frac{1}{(2K)^{q/p-1}}\frac{\delta}{20}\lambda^p_{\mathfrak{w}}+b(\mathfrak{w})\frac{\lambda^q_{\mathfrak{w}}}{(2K)^{\frac{q}{p}}}+1\\
&< \frac{K}{5}b(\mathfrak{w})\frac{\lambda^p_{\mathfrak{w}}}{2K}+a(\mathfrak{w})\lambda^p_{\mathfrak{w}}+\frac{1}{(2K)^{q/p-1}}\frac{1}{10}b(\mathfrak{w})\lambda^p_{\mathfrak{w}}+b(\mathfrak{w})\frac{\lambda^q_{\mathfrak{w}}}{(2K)^{\frac{q}{p}}}+1\\
&=a(\mathfrak{w})\lambda^p_{\mathfrak{w}}+\left(\frac{1}{10}+\frac{1}{(2K)^{\frac{q}{p}}}+\frac{1}{(2K)^{q/p-1}10}\right)b(\mathfrak{w})\lambda^q_{\mathfrak{w}}+1\\
&<a(\mathfrak{w})\lambda^p_{\mathfrak{w}}+\frac{7}{10}b(\mathfrak{w})\lambda^q_{\mathfrak{w}}+1<a(\mathfrak{w})\lambda^p_{\mathfrak{w}}+b(\mathfrak{w})\lambda^q_{\mathfrak{w}}+1=\Lambda_1,
\end{align*}
which yields a contradiction. 

\noindent \textit{Case (1-2):} Since, by \eqref{radius compare} we have $\rho_{\mathfrak{z}}\leq 2\rho_{q, \mathfrak{w}},$ we can conclude this case by following the same argument as above.

Note that \eqref{estimate on rho alpha} is proved in the $p$-intrinsic case, but the same assertion can be proved for the $q$-intrinsic case as well.

Assuming $\lambda_{\mathfrak{w}}\leq \lambda_{\mathfrak{z}},$ by an argument similar to above, we can show $\lambda_{\mathfrak{z}}\leq (2K)^{\frac{1}{p}}\lambda_{\mathfrak{w}}$ in all the above cases. This completes the proof of the comparability of the scaling factors.

\subsubsection{Space inclusion} The inclusion of the space part follows from properties $\mathcal{Q}(\mathfrak{z})\cap \mathcal{Q}(\mathfrak{w})\neq \emptyset$ and \eqref{radius compare}. We only show case (1-1) in \cref{tab : 9 Cases} and other cases follow similarly. Let $\mathfrak{z}=(x_{\mathfrak{z}}, t_{\mathfrak{z}})$ and $\mathfrak{w}=(x_{\mathfrak{w}}, t_{\mathfrak{w}})$ and we have $B_{2\rho_{\mathfrak{z}}}(x_{\mathfrak{z}})\cap B_{2\rho_{\mathfrak{w}}}(x_{\mathfrak{w}})\neq \emptyset.$ For $x \in B_{2\rho_{\mathfrak{z}}}(x_{\mathfrak{z}}),$ we have

\begin{align*}
    |x-x_{\mathfrak{w}}|\leq |x-x_{\mathfrak{z}}|+|x_{\mathfrak{z}}-x_{\mathfrak{w}}|\leq 2\rho_{\mathfrak{z}}+2\rho_{\mathfrak{z}}+2\rho_{\mathfrak{w}}=4\rho_{\mathfrak z}+2\rho_{\mathfrak w}\overset{\eqref{radius compare}}{\leq} 10\rho_{\mathfrak{w}}\leq c_v (2\rho_{\mathfrak{w}}). 
\end{align*}
Choosing $c_*(\texttt{data})\geq \frac{5}{K},$ we have the conclusion that $x\in B_{2c_{v}\rho_{\mathfrak{w}}}(x_{\mathfrak w}).$
\subsubsection{Time inclusion} For the time part, we need to consider several cases.

\noindent \textit{Case (1-1):} In this case, we have
\begin{align*}
    I_{2\rho_{\mathfrak{z}}, \lambda_{\mathfrak{z}}}(t_{\mathfrak{z}})\cap I_{2\rho_{\mathfrak{w}}, \lambda_{\mathfrak{w}}}(t_{\mathfrak{w}})\neq \emptyset.
\end{align*}
Let $t \in I_{2\rho_{\mathfrak{z}}, \lambda_{\mathfrak{z}}}(t_{\mathfrak{z}}).$ Note that for $c_v= c_*(\texttt{data})K \geq c_*(\texttt{data}),$ we have $\omega_a(\rho_{\mfz})+\omega_a(\rho_{\mfw})\leq \frac{\delta}{2}.$ Using this and $a(\mathfrak{z})\geq \frac{\delta}{2},$ we have
\begin{align}\label{EQQ4.23}
a(\mfw)\lambda^p_{\mfw}&\leq |a(\mfw)-a(\mfz)|\lambda^p_{\mfw}+a(\mfz)\lambda^p_{\mfw}\nonumber\\
&\leq \left(\omega_a(\rho_{\mfz})+\omega_a(\rho_{\mfw})\right)\lambda^p_{\mfw}+a(\mfz)\lambda^p_{\mfw}\leq \frac{\delta}{2}\lambda^p_{\mfw}+a(\mfz)\lambda^p_{\mfw}\leq 2 a(\mfz)\lambda^p_{\mfw}\leq 4K a(\mfz)\lambda^p_{\mfz}.
\end{align}

Then, using \eqref{EQQ4.23} we get
\begin{align}\label{EQQ4.24}
    |t-t_{\mathfrak{w}}|\leq |t-t_{\mathfrak{z}}|+|t_{\mathfrak{z}}-t_{\mathfrak{w}}|&\leq \frac{2\lambda^{2}_{\mathfrak{z}}(2\rho_{\mathfrak{z}})^2}{\lambda^p_{\mfz}a(\mathfrak{z})}+\frac{\lambda^{2}_{\mathfrak{w}}(2\rho_{\mathfrak{w}})^2}{\lambda^p_{\mfw}a(\mathfrak{w})}\nonumber\\
    &\leq\frac{2\lambda^2_{\mfz}(2 \rho_{\mfz})^2(4K)}{a(\mfw)\lambda^p_{\mfw}}+\frac{\lambda^{2}_{\mathfrak{w}}(2\rho_{\mathfrak{w}})^2}{\lambda^p_{\mfw}a(\mathfrak{w})}.
\end{align}
Using $\lambda_{\mathfrak z} \leq (2K)^{\frac{1}{p}}\lambda_{\mathfrak w}$ and $\rho_{\mathfrak z}\leq 2\rho_{\mathfrak w}$ in \eqref{EQQ4.24}, we obtain
\begin{align*}
 \frac{2\lambda^2_{\mfz}(2 \rho_{\mfz})^2(4K)}{a(\mfw)\lambda^p_{\mfw}}+\frac{\lambda^{2}_{\mathfrak{w}}(2\rho_{\mathfrak{w}})^2}{\lambda^p_{\mfw}a(\mathfrak{w})}&\leq \frac{2(2K)^{2/p}\lambda^2_{\mfw}4(2\rho_{\mfw})^24K}{\lambda^p_{\mfw}a(\mfw)}+\frac{\lambda^2_{\mfw}(2\rho_{\mfw})^2}{\lambda^p_{\mfw}a(\mfw)}\\
 &\leq \left(64K^2+1\right)\frac{\lambda^2_{\mfw}(2\rho_{\mfw})^2}{\lambda^p_{\mfw}a(\mfw)}\leq \frac{\lambda^2_{\mfw}(2c_v\rho_{\mfw})^2}{\lambda^p_{\mfw}a(\mfw)}
\end{align*}
if \begin{align}\label{c_1}c_*(\texttt{data})\geq \frac{\sqrt{64K^2+1}}{K}.\end{align}

\noindent \textit{Case (1-2):} In this case, we have
\begin{align*}
    I_{2\rho_{\mathfrak{z}}, \lambda_{\mathfrak{z}}}(t_{\mathfrak{z}})\cap I^q_{2\rho_{q, \mathfrak{w}}, \lambda_{\mathfrak{w}}}(t_{\mathfrak{w}})\neq \emptyset.
\end{align*}
Here we first show that
\begin{align}\label{important claim}
    b(\mfw)\lambda^q_{\mfw}\apprle a(\mfz)\lambda^p_{\mfz}.
\end{align}
Indeed, we have
\begin{align*}
    b(\mfw)\lambda^q_{\mfw}\leq |b(\mfw)-b(\mfz)|\lambda^q_{\mfw}+b(\mfz)\lambda^q_{\mfw}&\leq 6[b]_{\alpha}\rho^{\alpha}_{q; \mfw}\lambda^q_{\mfw}+b(\mfz)\lambda^q_{\mfw}\\
    &\overset{\eqref{estimate on rho alpha}}{\leq} \frac{K \delta}{20}\lambda^p_{\mfw}+(2K)^{\frac{q}{p}}b(\mfz)\lambda^q_{\mfz}\\ 
    &\overset{\eqref{eq: p-phase condition}}{\leq} \frac{2K^2}{10}a(\mfz)\lambda^p_{\mfz}+(2K)^{\frac{q}{p}}b(\mfz)\lambda^q_{\mfz}\\
    &\overset{\eqref{eq: p-phase condition}}{\leq} \frac{2K^2}{10}a(\mfz)\lambda^p_{\mfz}+(2K)^{\frac{q}{p}}Ka(\mfz)\lambda^p_{\mfz}\\
    &\leq \left(\frac{K^2}{5}+(2K)^{\frac{q}{p}+1}\right)a(\mfz)\lambda^p_{\mfz}
    \leq (K^3+(2K)^3)a(\mfz)\lambda^p_{\mfz}=9K^3a(\mfz)\lambda^p_{\mfz}.
\end{align*}
This proves the claim \eqref{important claim}.
Let $t \in I_{2\rho_{\mathfrak{z}}, \lambda_{\mathfrak{z}}}(t_{\mathfrak{z}}).$ Then, using \eqref{important claim}, we have 
\begin{align*}
    |t-t_{\mathfrak{w}}|\leq |t-t_{\mathfrak{z}}|+|t_{\mathfrak{z}}-t_{\mathfrak{w}}|&\leq \frac{2\lambda^{2}_{\mathfrak{z}}(2\rho_{\mathfrak{z}})^2}{a(\mathfrak{z})\lambda^p_{\mfz}}+\frac{\lambda^{2}_{\mathfrak{w}}(2\rho_{q;\mathfrak{w}})^2}{b(\mathfrak{w})\lambda^q_{\mfw}}\\
    &\leq \frac{2\cdot 9K^3\lambda^2_{\mfz}(2\rho_{\mfz})^2}{b(\mfw)\lambda^q_{\mfw}}+\frac{\lambda^{2}_{\mathfrak{w}}(2\rho_{q;\mathfrak{w}})^2}{b(\mathfrak{w})\lambda^q_{\mfw}}\\
    &\leq\frac{8\cdot 9K^3(2K)^{2/p}\lambda^2_{\mfw}(2\rho_{q;\mfw})^2}{b(\mfw)\lambda^q_{\mfw}}+\frac{\lambda^{2}_{\mathfrak{w}}(2\rho_{q;\mathfrak{w}})^2}{b(\mathfrak{w})\lambda^q_{\mfw}}\\
    &\leq\left(8\cdot 9K^3(2K)+1\right)\frac{\lambda^{2}_{\mathfrak{w}}(2\rho_{q;\mathfrak{w}})^2}{b(\mathfrak{w})\lambda^q_{\mfw}}\leq \frac{\lambda^{2}_{\mathfrak{w}}(2c_v\rho_{q;\mathfrak{w}})^2}{b(\mathfrak{w})\lambda^q_{\mfw}}
\end{align*}
if \begin{align}\label{c_2}c_*(\texttt{data} )\geq \frac{\sqrt{144K^4+1}}{K}.\end{align}

\noindent \textit{Case (1-3):} In this case, we have
\begin{align*}
    I_{2\rho_{\mathfrak{z}}, \lambda_{\mathfrak{z}}}(t_{\mathfrak{z}})\cap I^{(p,q)}_{2\rho_{q, \mathfrak{w}}, \lambda_{\mathfrak{w}}}(t_{\mathfrak{w}})\neq \emptyset.
\end{align*}
Let $z_0\in \mathcal{Q}(\mfz)\cap \mathcal{Q}(\mfw).$ Then from \eqref{eq: p-phase condition}
 and \eqref{Eq 3.8}--\eqref{Eq 3.9}, we have
 \begin{align*}
     \frac{a(z_0)}{3}\leq a(\mfz)\leq 3a(z_0), \frac{a(z_0)}{3}\leq a(\mfw)\leq 3a(z_0),\,\, \text{and}\,\,\, \frac{b(z_0)}{3}\leq b(\mfw)\leq 3b(z_0).
 \end{align*}
 Thus, using the above estimate and \eqref{important claim}, we deduce
 \begin{align*}
a(\mathfrak{w})\lambda^p_{\mathfrak{w}}+b(\mathfrak{w})\lambda^q_{\mathfrak{w}}\leq 9(2K)a(\mfz)\lambda^p_{\mfz}+9K^3 a(\mfz)\lambda^p_{\mfz}\leq (3K)^3 a(\mfz)\lambda^p_{\mfz}.   
 \end{align*}
Let $t \in I_{2\rho_{\mathfrak{z}}, \lambda_{\mathfrak{z}}}(t_{\mathfrak{z}}).$ Now we have
\begin{align*}
 |t-t_{\mathfrak{w}}|\leq |t-t_{\mathfrak{z}}|+|t_{\mathfrak{z}}-t_{\mathfrak{w}}|&\leq \frac{2\lambda^{2}_{\mathfrak{z}}(2\rho_{\mathfrak{z}})^2}{a(\mathfrak{z})\lambda^p_{\mfz}}+\frac{\lambda^2_{\mathfrak{w}}(2\rho_{p,q; \mathfrak{w}})^2}{a(\mathfrak{w})\lambda^p_{\mathfrak{w}}+b(\mathfrak{w})\lambda^q_{\mathfrak{w}}}\\
 &\leq \frac{8\lambda^2_{\mfz}(3K)^3(2\rho_{p,q;\mfw})^2}{a(\mathfrak{w})\lambda^p_{\mathfrak{w}}+b(\mathfrak{w})\lambda^q_{\mathfrak{w}}}+\frac{\lambda^2_{\mathfrak{w}}(2\rho_{p,q; \mathfrak{w}})^2}{a(\mathfrak{w})\lambda^p_{\mathfrak{w}}+b(\mathfrak{w})\lambda^q_{\mathfrak{w}}}\\
 &\leq \frac{8(2K)^{2/p}\lambda^2_{\mfw}(3K)^3(2\rho_{p,q;\mfw})^2}{a(\mathfrak{w})\lambda^p_{\mathfrak{w}}+b(\mathfrak{w})\lambda^q_{\mathfrak{w}}}+\frac{\lambda^2_{\mathfrak{w}}(2\rho_{p,q; \mathfrak{w}})^2}{a(\mathfrak{w})\lambda^p_{\mathfrak{w}}+b(\mathfrak{w})\lambda^q_{\mathfrak{w}}}\\
 &\leq \left(16K(3K)^3+1\right)\frac{\lambda^2_{\mathfrak{w}}(2\rho_{p,q; \mathfrak{w}})^2}{a(\mathfrak{w})\lambda^p_{\mathfrak{w}}+b(\mathfrak{w})\lambda^q_{\mathfrak{w}}}\leq \frac{\lambda^2_{\mathfrak{w}}(2c_v\rho_{p,q; \mathfrak{w}})^2}{a(\mathfrak{w})\lambda^p_{\mathfrak{w}}+b(\mathfrak{w})\lambda^q_{\mathfrak{w}}}
\end{align*}
if 
\begin{align}\label{c_3}c_*(\texttt{data})\geq \frac{\sqrt{432K^4+1}}{K}.\end{align}

\noindent \textit{Case (2-1):} In this case, we have
\begin{align*}
 I^q_{2\rho_{q;\mathfrak{z}}, \lambda_{\mathfrak{z}}}(t_{\mathfrak{z}})\cap I_{2\rho_{ \mathfrak{w}}, \lambda_{\mathfrak{w}}}(t_{\mathfrak{w}})\neq \emptyset.   
\end{align*}
We first claim that
\begin{align}\label{important claim 2}
    a(\mfw)\lambda^p_{\mfw}\apprle b(\mfz)\lambda^q_{\mfz}.
\end{align}
Indeed, for $c_v=c_*(\texttt{data})K\geq c_*(\texttt{data})$, we see that
\begin{align*}
    a(\mfw)\lambda^p_{\mfw}\leq |a(\mfw)-a(\mfz)|\lambda^p_{\mfw}+a(\mfz)\lambda^p_{\mfw}&\leq \left(\omega_a(\rho_{\mfz})+\omega_a(\rho_{\mfw})\right)\lambda^p_{\mfw}+a(\mfz)\lambda^p_{\mfw}\\
    &\leq\frac{\delta}{2}\lambda^p_{\mfw}+(2K)a(\mfz)\lambda^p_{\mfz}\\
    &\leq (2K)b(\mfz)\lambda^p_{\mfz}+(2K)\frac{b(\mfz)\lambda^q_{\mfz}}{K}\leq 4K b(\mfz)\lambda^q_{\mfz}.
\end{align*}
Let $t \in I^q_{2\rho_{q;\mathfrak{z}}, \lambda_{\mathfrak{z}}}(t_{\mathfrak{z}}).$ Then using \eqref{important claim 2}, we get
\begin{align*}
    |t-t_{\mathfrak{w}}|\leq |t-t_{\mathfrak{z}}|+|t_{\mathfrak{z}}-t_{\mathfrak{w}}|&\leq \frac{\lambda^{2}_{\mathfrak{w}}(2\rho_{\mathfrak{w}})^2}{a(\mathfrak{w})\lambda^p_{\mfw}}+\frac{2\lambda^{2}_{\mathfrak{z}}(2\rho_{q;\mathfrak{z}})^2}{b(\mathfrak{z})\lambda^q_{\mfz}}\\
    &\leq \frac{\lambda^{2}_{\mathfrak{w}}(2\rho_{\mathfrak{w}})^2}{a(\mathfrak{w})\lambda^p_{\mfw}}+\frac{8K\lambda^2_{\mfz}(2\rho_{q;\mfz})^2}{a(\mfw)\lambda^p_{\mfw}}\\
    &\leq \frac{\lambda^{2}_{\mathfrak{w}}(2\rho_{\mathfrak{w}})^2}{a(\mathfrak{w})\lambda^p_{\mfw}}+ \frac{32K (2K)^{2/p}\lambda^2_{\mfw}(2\rho_{\mfw})^2}{a(\mfw)\lambda^p_{\mfw}}\\
    &\leq \left(64K^2+1\right)\frac{\lambda^{2}_{\mathfrak{w}}(2\rho_{\mathfrak{w}})^2}{a(\mathfrak{w})\lambda^p_{\mfw}}\leq \frac{\lambda^{2}_{\mathfrak{w}}(2c_v\rho_{\mathfrak{w}})^2}{a(\mathfrak{w})\lambda^p_{\mfw}}
\end{align*}
if
\begin{align}\label{c_4}
    c_*(\texttt{data})\geq \frac{\sqrt{64K^2+1}}{K}.
\end{align}

\noindent \textit{Case (2-2):} In this case, we have
\begin{align*}
 I^q_{2\rho_{q;\mathfrak{z}}, \lambda_{\mathfrak{z}}}(t_{\mathfrak{z}})\cap I^q_{2\rho_{q; \mathfrak{w}}, \lambda_{\mathfrak{w}}}(t_{\mathfrak{w}})\neq \emptyset.    
\end{align*}
We note that, for $z_0 \in \mathcal{Q}(\mfz)\cap \mathcal{Q}(\mfw)$,
\begin{align*}
\frac{b(z_0)}{3}\leq b(\mfz)\leq 3b(z_0) \quad \text{and}\quad \frac{b(z_0)}{3}\leq b(\mfw)\leq 3b(z_0).
\end{align*}
This implies 
\begin{align}\label{EQQ4.27}
    \frac{1}{b(\mfz)\lambda^q_{\mfz}}\leq \frac{9(2K)^{q/p}}{b(\mfw)\lambda^q_{\mfw}}.
\end{align}
Let $t \in I^q_{2\rho_{q;\mathfrak{z}}, \lambda_{\mathfrak{z}}}(t_{\mathfrak{z}}).$ Then, substituting \eqref{EQQ4.27} gives
\begin{align*}
    |t-t_{\mathfrak{w}}|\leq |t-t_{\mathfrak{z}}|+|t_{\mathfrak{z}}-t_{\mathfrak{w}}|&\leq \frac{\lambda^{2}_{\mathfrak{w}}(2\rho_{q;\mathfrak{w}})^2}{b(\mathfrak{w})\lambda^q_{\mfw}}+\frac{2\lambda^{2}_{\mathfrak{z}}(2\rho_{q;\mathfrak{z}})^2}{b(\mathfrak{z})\lambda^q_{\mfz}}\\
    &\leq \frac{\lambda^{2}_{\mathfrak{w}}(2\rho_{q;\mathfrak{w}})^2}{b(\mathfrak{w})\lambda^q_{\mfw}}+ \frac{32(2K)^{\frac{q}{p}+1}\lambda^2_{\mfw}(2\rho_{q;w})^2}{b(\mfw)\lambda^q_{\mfw}}\\
    &\leq \left(256K^3+1\right)\frac{\lambda^{2}_{\mathfrak{w}}(2\rho_{q;\mathfrak{w}})^2}{b(\mathfrak{w})\lambda^q_{\mfw}}\leq \frac{\lambda^{2}_{\mathfrak{w}}(2c_v\rho_{q;\mathfrak{w}})^2}{b(\mathfrak{w})\lambda^q_{\mfw}} 
\end{align*}
if
\begin{align}\label{c_5}
    c_*(\texttt{data})\geq \frac{\sqrt{256K^3+1}}{K}.
\end{align}

\noindent \textit{Case (2-3):} In this case, we have
\begin{align*}
 I^q_{2\rho_{q;\mathfrak{z}}, \lambda_{\mathfrak{z}}}(t_{\mathfrak{z}})\cap I^{(p,q)}_{2\rho_{p,q; \mathfrak{w}}, \lambda_{\mathfrak{w}}}(t_{\mathfrak{w}})\neq \emptyset.       
\end{align*}

We first note that, using \eqref{important claim 2} and \eqref{EQQ4.27}, we get
\begin{align*}
    a(\mfw)\lambda^p_{\mfw}+b(\mfw)\lambda^q_{\mfw}\leq 4K b(\mfz)\lambda^q_{\mfz}+9(2K)^{\frac{q}{p}}b(\mfz)\lambda^q_{\mfz}\leq 40K^2 b(\mfz)\lambda^q_{\mfz}.
\end{align*}
Let $t \in I^q_{2\rho_{q;\mathfrak{z}}, \lambda_{\mathfrak{z}}}(t_{\mathfrak{z}}).$ Then using this above estimate, we get
\begin{align*}
 |t-t_{\mathfrak{w}}|\leq |t-t_{\mathfrak{z}}|+|t_{\mathfrak{z}}-t_{\mathfrak{w}}|&\leq \frac{2\lambda^{2}_{\mathfrak{z}}(2\rho_{q;\mathfrak{z}})^2}{b(\mathfrak{z})\lambda^q_{\mfz}}+\frac{\lambda^2_{\mathfrak{w}}(2\rho_{p,q; \mathfrak{w}})^2}{a(\mathfrak{w})\lambda^p_{\mathfrak{w}}+b(\mathfrak{w})\lambda^q_{w}}\\
 &\leq \frac{160K^2(2K)^{\frac{2}{p}}\lambda^2_{\mfw}(2\rho_{p,q; \mfw})^2}{a(\mathfrak{w})\lambda^p_{\mathfrak{w}}+b(\mathfrak{w})\lambda^q_{w}}+\frac{\lambda^2_{\mathfrak{w}}(2\rho_{p,q; \mathfrak{w}})^2}{a(\mathfrak{w})\lambda^p_{\mathfrak{w}}+b(\mathfrak{w})\lambda^q_{w}}\\
 &\leq (320K^3+1)\frac{\lambda^2_{\mathfrak{w}}(2\rho_{p,q; \mathfrak{w}})^2}{a(\mathfrak{w})\lambda^p_{\mathfrak{w}}+b(\mathfrak{w})\lambda^q_{w}}\leq \frac{\lambda^2_{\mathfrak{w}}(2c_v\rho_{p,q; \mathfrak{w}})^2}{a(\mathfrak{w})\lambda^p_{\mathfrak{w}}+b(\mathfrak{w})\lambda^q_{w}}
\end{align*}
if
\begin{align}\label{c_6}
    c_*(\texttt{data})\geq \frac{\sqrt{320K^3+1}}{K}.
\end{align}

\noindent \textit{Case (3-1):} In this case, we have
\begin{align*}
    I^{(p,q)}_{2\rho_{p,q; \mfz}, \lambda_{\mfz}}(t_{\mfz})\cap I_{2\rho_{\mfw}, \lambda_{\mfw}}(t_{\mfw})\neq \emptyset.
\end{align*}
For $t \in I^{(p,q)}_{2\rho_{p,q; \mfz}, \lambda_{\mfz}}(t_{\mfz})$, we get
\begin{align*}
    |t-t_{\mfw}|\leq |t-t_{\mfz}|+|t_{\mfz}-t_{\mfw}|&\leq  \frac{2 \lambda^2_{\mfz}(2 \rho_{p,q; \mfz})^2}{a(\mfz)\lambda^p_{\mfz}+b(\mfz)\lambda^q_{\mfz}}+\frac{ \lambda^2_{\mfw}(2 \rho_{\mfw})^2}{a(\mfw)\lambda^p_{\mfw}}\\
    &\leq \frac{8 \lambda^2_{\mfz}(2 \rho_{\mfz})^2}{a(\mfz)\lambda^p_{\mfz}}+\frac{ \lambda^2_{\mfw}(2 \rho_{\mfw})^2}{a(\mfw)\lambda^p_{\mfw}}.
\end{align*}
This turns out to be the Case (1-1) with a different constant and following the calculations, we conclude $I^{(p,q)}_{2\rho_{p,q; \mfz}, \lambda_{\mfz}}(t_{\mfz})\subset I_{2c_v\rho_{\mfw}, \lambda_{\mfw}}(t_{\mfw})$ if 
\begin{align}\label{c_7}
    c_*(\texttt{data})\geq \frac{\sqrt{256K^2+1}}{K}.
\end{align}

\noindent \textit{Case (3-2):} In this case, we have
\begin{align*}
    I^{(p,q)}_{2\rho_{p,q; \mfz}, \lambda_{\mfz}}(t_{\mfz})\cap I^q_{2\rho_{q;\mfw}, \lambda_{\mfw}}(t_{\mfw})\neq \emptyset.
\end{align*}
For $t \in I^{(p,q)}_{2\rho_{p,q; \mfz}, \lambda_{\mfz}}(t_{\mfz})$, we get
\begin{align*}
    |t-t_{\mfw}|\leq |t-t_{\mfz}|+|t_{\mfz}-t_{\mfw}|&\leq  \frac{2 \lambda^2_{\mfz}(2 \rho_{p,q; \mfz})^2}{a(\mfz)\lambda^p_{\mfz}+b(\mfz)\lambda^q_{\mfz}}+\frac{ \lambda^2_{\mfw}(2 \rho_{q;\mfw})^2}{b(\mfw)\lambda^q_{\mfw}}\\
    &\leq \frac{8 \lambda^2_{\mfz}(2 \rho_{q;\mfz})^2}{b(\mfz)\lambda^q_{\mfz}}+\frac{ \lambda^2_{\mfw}(2 \rho_{\mfw})^2}{a(\mfw)\lambda^p_{\mfw}}.
\end{align*}
This turns out to be the Case (2-1) with a different constant and following the calculations, we conclude $I^{(p,q)}_{2\rho_{p,q; \mfz}, \lambda_{\mfz}}(t_{\mfz})\subset I^q_{2c_v\rho_{q;\mfw}, \lambda_{\mfw}}(t_{\mfw})$ if 
\begin{align}\label{c_8}
    c_*(\texttt{data})\geq \frac{\sqrt{256K^2+1}}{K}.
\end{align}
\noindent \textit{Case (3-3):} In this case, we have
\begin{align*}
    I^{(p,q)}_{2\rho_{p,q; \mfz}, \lambda_{\mfz}}(t_{\mfz})\cap I^{(p,q)}_{2\rho_{p,q;\mfw}, \lambda_{\mfw}}(t_{\mfw})\neq \emptyset.
\end{align*}
Using the comparability of $a(\cdot)$ and $b(\cdot),$ we observe that
\begin{align*}
    a(\mfw)\lambda^p_{\mfw}+b(\mfw)\lambda^q_{\mfw}\leq 9 (2K)a(\mfz)\lambda^p_{\mfz}+9(2K)^{\frac{q}{p}}b(\mfz)\lambda^q_{\mfz}&\leq 18K^2a(\mfz)\lambda^p_{\mfz}+36K^2b(\mfz)\lambda^q_{\mfz}\\
    &\leq 36K^2 \left(a(\mfz)\lambda^p_{\mfz}+b(\mfz)\lambda^q_{\mfz}\right).
\end{align*}
Let $t\in I^{(p,q)}_{2\rho_{p,q; \mfz}, \lambda_{\mfz}}(t_{\mfz}).$ Using the above estimate, we get
\begin{align*}
    |t-t_{\mfw}|\leq |t-t_{\mfz}|+|t_{\mfz}-t_{\mfw}|&\leq \frac{2 \lambda^2_{\mfz}(2 \rho_{p,q; \mfz})^2}{a(\mfz)\lambda^p_{\mfz}+b(\mfz)\lambda^q_{\mfz}}+\frac{ \lambda^2_{\mfw}(2 \rho_{p,q; \mfw})^2}{a(\mfw)\lambda^p_{\mfw}+b(\mfw)\lambda^q_{\mfw}}\\
    &\leq\frac{8(2K)^{2/p}36K^2\lambda^2_{\mfw}(2 \rho_{p,q; \mfw})^2}{a(\mfw)\lambda^p_{\mfw}+b(\mfw)\lambda^q_{\mfw}}+\frac{ \lambda^2_{\mfw}(2 \rho_{p,q; \mfw})^2}{a(\mfw)\lambda^p_{\mfw}+b(\mfw)\lambda^q_{\mfw}}\\
    &\leq \left(576K^3+1\right)\frac{ \lambda^2_{\mfw}(2 \rho_{p,q; \mfw})^2}{a(\mfw)\lambda^p_{\mfw}+b(\mfw)\lambda^q_{\mfw}}\leq \frac{ \lambda^2_{\mfw}(2 c_v\rho_{p,q; \mfw})^2}{a(\mfw)\lambda^p_{\mfw}+b(\mfw)\lambda^q_{\mfw}}
\end{align*}
if
\begin{align}\label{c_9}
    c_*(\texttt{data})\geq \frac{\sqrt{576K^3+1}}{K}.
\end{align}
Finally, comparing all the lower bounds of  $c_*(\texttt{data}),$ from time inclusion, i.e., \eqref{c_1}, \eqref{c_2}, \eqref{c_3}, \eqref{c_4}, \eqref{c_5}, \eqref{c_6}, \eqref{c_7}, \eqref{c_8}, \eqref{c_9} and from space inclusion, i.e., $4/K,$ we conclude that it is sufficient to take $c_*(\texttt{data})\geq \sqrt{\frac{4\times 576K^4}{K^2}}=48K.$ The above discussion leads to the following lemma.
\begin{lemma}[Vitali covering lemma]\label{vitali lemma}
 Let $E(\Lambda_1, r_1)$  and $c_v$ be given by \eqref{defn_E(Lambda, rho)} and \eqref{eq: definition of c_nu} respectively. Then there exists a collection $\{\mathcal{Q}({z_i})\}_{i\in \mathbb{N}}$ of cylinders that satisfies the following.
\begin{itemize}
\item [(i)] $\cup_{i\in\mathbb{N}}\,\,c_v \mathcal{Q}({z_i})=E(\La_1, r_1)$.
\item[(ii)] $\mathcal{Q}({z_i})\cap \mathcal{Q}({z_j})= \emptyset$ for every $i,j\in \mathbb{N}$ with $i\ne j$.
\end{itemize}   
\end{lemma}
\subsection{Proof of \cref{main_thm}}\label{subsec:final_proof}
In this subsection, we complete the proof of \cref{main_thm}. First, we derive the following lemma as a consequence of reverse H\"{o}lder inequalities.
\begin{lemma}\label{LEM6.1}
Let $u$ be a weak solution to \eqref{eq: main equation}. Then there exist constants $c=c(\textnormal{\texttt{data}}) \geq 1$ and $\theta_0\in (0,1)$ such that for any $\theta \in (\theta_0, 1),$ we have
\begin{align}\label{EQQ4.19}
\iint_{Q_{2c_v\rho, \lambda}(z_0)}H_1(z, |D u(z)|)\, dz &\leq c\Lambda_1^{1-\theta}\iint_{Q_{2\rho, \lambda}(z_0) \cap E\left(\frac{\Lambda_1}{c}\right)} H_1(z, |D u(z)|)^\theta\, dz\nonumber\\
&\quad + c\iint_{Q_{2\rho, \lambda}(z_0)\cap \Phi \left(\frac{\Lambda_1}{c}\right)} H_1(z, |F(z)|)\,dz,
\end{align}
for the $p$-intrinsic case,
\begin{align}\label{EQQ4.20}
  \iint_{J_{2c_v\rho, \lambda}(z_0)}H_1(z, |D u(z)|)\, dz &\leq c\Lambda_1^{1-\theta}\iint_{J_{2\rho, \lambda}(z_0) \cap E\left(\frac{\Lambda_1}{c}\right)} H_1(z, |D u(z)|)^\theta\, dz\nonumber\\
&\quad + c\iint_{J_{2\rho, \lambda}(z_0)\cap \Phi \left(\frac{\Lambda_1}{c}\right)} H_1(z, |F(z)|)\,dz,  
\end{align}
for the $q$-intrinsic case and 
\begin{align}\label{EQQ4.21}
  \iint_{G_{2c_v\rho, \lambda}(z_0)}H_1(z, |D u(z)|)\, dz &\leq c\Lambda_1^{1-\theta}\iint_{G_{2\rho, \lambda}(z_0) \cap E\left(\frac{\Lambda_1}{c}\right)} H_1(z, |D u(z)|)^\theta\, dz\nonumber\\
&\quad + c\iint_{G_{2\rho, \lambda}(z_0)\cap \Phi \left(\frac{\Lambda_1}{c}\right)} H_1(z, |F(z)|)\,dz,  
\end{align} for the $(p, q)$-intrinsic case,
whenever $Q_{2c_v\rho, \lambda}(z_0), J_{2c_v\rho, \lambda}(z_0), G_{2c_v\rho, \lambda}(z_0) \subset \Omega_T.$
\end{lemma}
\begin{proof}
Using the third condition of \eqref{eq: p-phase condition}, we obtain
\begin{align*}
    \left(\miint{Q_{2\rho, \lambda}(z_0)}H_1(z, |D u|)^{\theta}\, dz\right)^{\frac{1-\theta}{\theta}} \leq \left(\miint{Q_{2\rho, \lambda}(z_0)}H_1(z, |D u|)\, dz\right)^{1-\theta}<\left(a(z_0)\la^p+1\right)^{1-\theta}.
\end{align*}
Hence, we have
\begin{align}\label{EQQ6.32}
 \left(\miint{Q_{2\rho, \lambda}(z_0)}H_1(z, |D u|)^{\theta}\,dz\right)^{\frac{1}{\theta}}\leq \left(a(z_0)\la^p+1\right)^{1-\theta}\miint{Q_{2\rho, \lambda}(z_0)}H_1(z, |Du|)^{\theta}\,dz.    
\end{align}
Now we write
\begin{align*}
Q_{2\rho, \lambda}(z_0)= \left(Q_{2\rho, \lambda}(z_0) \cap E\left({\frac{a(z_0)\lambda^p+1}{(2c)^{1/\theta}}}\right)\right) \cup \left(Q_{2\rho, \lambda}(z_0) \setminus E\left({\frac{a(z_0)\lambda^p+1}{(2c)^{1/\theta}}}\right)\right).
\end{align*}
From the definition of $E(\Lambda_1)$ given in \eqref{Defn E_lambda}, the right-hand side of \eqref{EQQ6.32} can be estimated as

\begin{align}\label{EQUATION8.11}
 &\left(\miint{Q_{2\rho, \lambda}(z_0)}H_1(z, |D u(z)|)^{\theta}\,dz\right)^{\frac{1}{\theta}}\nonumber\\
 &\leq\left(a(z_0)\lambda^p+1\right)^{1-\theta}\frac{\left(a(z_0)\lambda^p+1\right)^{\theta}}{2c}\frac{\left|Q_{2\rho, \lambda}(z_0) \setminus E\left({\frac{a(z_0)\lambda^p+1}{(2c)^{1/\theta}}}\right)\right|}{\left|Q_{2\rho, \lambda}(z_0)\right|}\nonumber\\
 &\qquad + \frac{\left(a(z_0)\lambda^p+1\right)^{1-\theta}}{\left|Q_{2\rho, \lambda}(z_0)\right|}\iint_{Q_{2\rho, \lambda}(z_0) \cap E\left({\frac{a(z_0)\lambda^p+1}{(2c)^{1/\theta}}}\right)} H_1(z, |D u(z)|)^\theta\, dz\nonumber\\
 &\leq \frac{a(z_0)\lambda^p+1}{2c} + \frac{\left(a(z_0)\lambda^p+1\right)^{1-\theta}}{\left|Q_{2\rho, \lambda}(z_0)\right|}\iint_{Q_{2\rho, \lambda}(z_0) \cap E\left({\frac{a(z_0)\lambda^p+1}{(2c)^{1/\theta}}}\right)} H_1(z, |D u(z)|)^\theta\, dz.
\end{align}

Similarly, using the definition of $\Phi(\Lambda_1)$ given by \eqref{Defn Phi_lambda}, we also have
\begin{align}\label{EQQ4.26}
&\miint{Q_{2\rho, \lambda}(z_0)}H_1(z, |F|)\,dz\nonumber \\
&\quad \leq \frac{a(z_0)\lambda^p+1}{2c} + \frac{1}{\left|Q_{2\rho, \lambda}(z_0)\right|}\iint_{Q_{2\rho, \lambda}(z_0) \cap \Phi\left({\frac{a(z_0)\lambda^p+1}{(2c)^{1/\theta}}}\right)} H_1(z, |F|)\, dz.
\end{align}
Combining the reverse H\"{o}lder inequality (see \cref{lem: reverse holder H_1}), \eqref{EQUATION8.11} and \eqref{EQQ4.26}, we see that
\begin{align}\label{EQUATION8.12}
    &\miint{Q_{\rho, \lambda}(z_0)} [H_1(z, |D u|)+H_1(z, |F|)]\, dz\nonumber \\
    &\leq c\left(\miint{Q_{2\rho, \lambda}(z_0)}H_1(z, |D u|)^{\theta}\,dz\right)^{\frac{1}{\theta}}+c\miint{Q_{2\rho, \lambda}(z_0)}H_1(z, |F|)\, dz\nonumber\\
    &\leq \frac{a(z_0)\lambda^p+1}{2} + \frac{c \left(a(z_0)\lambda^p+1\right)^{1-\theta}}{\left|Q_{2\rho, \lambda}(z_0)\right|}\iint_{Q_{2\rho, \lambda}(z_0) \cap E\left({\frac{a(z_0)\lambda^p+1}{(2c)^{1/\theta}}}\right)} H_1(z, |D u|)^\theta\, dz\nonumber \\
    &\quad +\frac{c}{\left|Q_{2\rho, \lambda}(z_0)\right|}\iint_{Q_{2\rho, \lambda}(z_0) \cap \Phi\left({\frac{a(z_0)\lambda^p+1}{(2c)^{1/\theta}}}\right)} H_1(z, |F|)\, dz.
\end{align}
Now using the third and fourth conditions \eqref{eq: p-phase condition} in \eqref{EQUATION8.12}, we get
\begin{align*}
&\miint{Q_{2c_v\rho, \lambda}(z_0)} [H_1(z, |D u|)+H_1(z, |F|)]\, dz\\
&\leq \frac{2c \left(a(z_0)\lambda^p+1\right)^{1-\theta}}{\left|Q_{2\rho, \lambda}(z_0)\right|}\iint_{Q_{2\rho, \lambda}(z_0) \cap E\left({\frac{a(z_0)\lambda^p+1}{(2c)^{1/\theta}}}\right)} H_1(z, |D u|)^\theta\, dz\\
&\quad +\frac{2c}{\left|Q_{2\rho, \lambda}(z_0)\right|}\iint_{Q_{2\rho, \lambda}(z_0) \cap \Phi\left({\frac{a(z_0)\lambda^p+1}{(2c)^{1/\theta}}}\right)} H_1(z, |F|)\, dz.
\end{align*}
Thus, we obtain
\begin{align}\label{EQQ6.33}
    \iint_{Q_{2c_v \rho, \lambda}(z_0)}H_1(z, |Du|)\,dz &\leq 2c \left(a(z_0)\lambda^p+1\right)^{1-\theta}c^{n+2}_v\iint_{Q_{2\rho, \lambda}(z_0) \cap E\left({\frac{a(z_0)\lambda^p+1}{(2c)^{1/\theta}}}\right)} H_1(z, |D u|)^\theta\, dz\nonumber\\
    &\quad +2c c^{n+2}_v \iint_{Q_{2\rho, \lambda}(z_0) \cap \Phi\left({\frac{a(z_0)\lambda^p+1}{(2c)^{1/\theta}}}\right)} H_1(z, |F|)\, dz.
\end{align}
Using $Ka(z_0)\lambda^p\geq b(z_0)\lambda^q$ of \eqref{eq: p-phase condition} and $K\geq 1,$ we get $a(z_0)\lambda^p+1 \geq \frac{a(z_0)\lambda^p+b(z_0)\lambda^q+1}{2K}=\frac{\Lambda_1}{2K}.$ Hence we have $\frac{a(z_0)\lambda^p+1}{(2c)^{1/\theta}}\geq \frac{a(z_0)\lambda^p+1}{(2c)^{1/ \theta_0}} \geq \frac{\Lambda_1}{2K(2c)^{1/ \theta_0}}.$ This gives that $E\left({\frac{a(z_0)\lambda^p+1}{(2c)^{1/\theta}}}\right) \subset E\left({\frac{\Lambda_1}{2K(2c)^{1/\theta_0}}}\right).$ On the other hand, we also have $E\left({\frac{a(z_0)\lambda^p+1}{(2c)^{1/\theta}}}\right) \subset E\left(\frac{\Lambda_1}{2K(2c c^{n+2}_v)^{1/\theta_0}}\right).$ Also, similar inclusions hold for $\Phi(\cdot).$ Now by setting $\max\left\{2c, 2cc^{n+2}_v\right\}=:c_1=2cc^{n+2}_v$, we obtain from \eqref{EQQ6.33} that
\begin{align*}
\iint_{Q_{2c_v\rho, \lambda}(z_0)}H_1(z, |D u|)\, dz &\leq c_1\Lambda_1^{1-\theta}\iint_{Q_{2\rho, \lambda}(z_0) \cap E\left(\frac{\Lambda_1}{c}\right)} H_1(z, |D u|)^\theta\, dz\nonumber\\
&\quad + c_1\iint_{Q_{2\rho, \lambda}(z_0)\cap \Phi \left(\frac{\Lambda_1}{c}\right)} H_1(z, |F|)\,dz,
\end{align*}
which completes the proof of \eqref{EQQ4.19}. The estimates in \eqref{EQQ4.20} and \eqref{EQQ4.21} can be obtained similarly.
\end{proof}
Now we are ready to prove \cref{main_thm}. We present a detailed proof for the sake of completeness.
\begin{proof}[Proof of \cref{main_thm}]
Following \cref{vitali lemma}, we have a countable pairwise disjoint collection $\mathcal{I}:=\left\{\mathcal{Q}({z_j})\right\}_{j=1}^{\infty}$ for $z_j \in E\left({\Lambda_1, r_1}\right).$ From \cref{LEM6.1}, there exist constants $c \geq 1$ and $\theta_0\in (0,1)$ such that
\begin{align*}
    \iint_{c_v \mathcal{Q}(z_j)} H_1(z, |D u|)\, dz \leq c\Lambda_1^{1-\theta}\iint_{\mathcal{Q}(z_j) \cap E\left({\frac{\Lambda_1}{c}}\right)}H_1(z, |Du|)^{\theta}\, dz+ c\iint_{\mathcal{Q}(z_j)\cap \Phi \left(\frac{\Lambda_1}{c}\right)} H_1(z, |F|)\,dz
\end{align*}
holds for every $j \in \NN$ and $\theta \in (\theta_0, 1).$ Since the cylinders in $\mathcal{I}$ are pairwise disjoint, we get
\begin{align}\label{EQQ6.34}
    & \iint_{E\left({\Lambda_1, r_1}\right)} H_1(z, |D u|)\, dz  \nonumber \\
    &\quad \leq \sum_{j=1}^{\infty} \iint_{c_v \mathcal{Q}(z_j)} H_1(z, |D u|)\, dz\nonumber \\
    &\quad \leq c\Lambda_1^{1-\theta} \sum_{j=1}^{\infty}\iint_{\mathcal{Q}(z_j) \cap E\left({\frac{\Lambda_1}{c}}\right)}H_1(z, |D u|)^{\theta} dz+c\sum_{j=1}^{\infty}\iint_{\mathcal{Q}(z_j)\cap \Phi \left(\frac{\Lambda_1}{c}\right)} H_1(z, |F|)\,dz\nonumber\\
    &\quad \leq c\Lambda_1^{1-\theta}\iint_{E\left({\frac{\Lambda_1}{c}, r_2}\right)}H_1(z, |D u|)^{\theta} dz+ c\iint_{ \Phi \left(\frac{\Lambda_1}{c}, r_2\right)} H_1(z, |F|)\,dz.
\end{align}
On the other hand, we have
\begin{align}\label{EQQ6.35}
    \iint_{E\left({\frac{\Lambda_1}{c}, r_1}\right)\setminus E\left({\Lambda_1, r_1}\right)}H_1(z, |D u|)\, dz
    &= \iint_{E\left({\frac{\Lambda_1}{c}, r_1}\right)\setminus E\left({\Lambda_1, r_1}\right)}H_1(z, |D u|)^\theta H_1(z, |D u|)^{1-\theta}\, dz\nonumber\\
    &\leq \Lambda_1^{1-\theta}\iint_{E\left({\frac{\Lambda_1}{c}}, r_2\right)} H_1(z, |D u|)^{\theta}\,dz.
\end{align}
Now combining \eqref{EQQ6.34} and \eqref{EQQ6.35}, we conclude that 
\begin{align}\label{EQQ6.36}
    \iint_{E\left({\frac{\Lambda_1}{c}, r_1}\right)}H_1(z, |D u|) \, dz&=\iint_{E\left({\Lambda_1, r_1}\right)}H_1(z, |D u|) \, dz+\iint_{E\left({\frac{\Lambda_1}{c}, r_1}\right)\setminus E\left({\Lambda_1, r_1}\right)}H_1(z, |D u|) \, dz\nonumber\\
    &\leq c\Lambda_1^{1-\theta}\iint_{E\left({\frac{\Lambda_1}{c}, r_2}\right)}H_1(z, |D u|)^\theta dz+ c\iint_{ \Phi \left(\frac{\Lambda_1}{c}, r_2\right)} H_1(z, |F|)\,dz .
\end{align}
For any $k\in \mathbb{N},$ we define the truncation operator as
\begin{align*}
    H_1(z, |D u|)_k =\min \left\{H_1(z, |D u|), k\right\}
\end{align*}
and the superlevel set as 
\begin{align*}
    E^k\left({\Lambda_1, \uprho}\right):=\left\{z\in Q_{\uprho}(z_0) : H_1(z, |D u|)_k > \Lambda_1\right\}.
\end{align*}
From the above definitions, we have 
\begin{align}\label{EQQ6.37}
    E^k\left({\Lambda_1, \uprho}\right)=\begin{cases}
        E\left({\Lambda_1, \uprho}\right)\,\,\,\, &\text{if}\,\,\, \Lambda_1 \leq k,\\
        \emptyset\,\,\,\, &\text{if}\,\,\, \Lambda_1> k.
    \end{cases}
\end{align}
Therefore, we deduce from \eqref{EQQ6.36} and \eqref{EQQ6.37} that 
\begin{align*}
\iint_{E^k\left({\frac{\Lambda_1}{c}, r_1}\right)}H_1(z, |D u|)_k^{1-\theta}H_1(z, |D u|)^\theta dz&\leq \iint_{E\left({\frac{\Lambda_1}{c}, r_1}\right)}H_1(z, |D u|) \, dz \\
&\leq c\Lambda_1^{1-\theta}\iint_{E\left({\frac{\Lambda_1}{c}, r_2}\right)}H_1(z, |D u|)^\theta dz+ c\iint_{ \Phi \left(\frac{\Lambda_1}{c}, r_2\right)} H_1(z, |F|)\,dz .
\end{align*}
Let us denote
\begin{align*}
\Lambda_2=\frac{1}{c}\left(\frac{4 c_v r}{r_2-r_1}\right)^{\frac{q(n+2)}{2}}\Lambda_0.
\end{align*}
Then for any $\Lambda_1 > \Lambda_2,$ we have
\begin{align}\label{EQQ6.38}
\iint_{E^k\left({\Lambda_1, r_1}\right)}&H_1(z, |D u|)_k^{1-\theta}H_1(z, |D u|)^\theta dz\nonumber\\
&\leq c\Lambda_1^{1-\theta}\iint_{E^k\left({\Lambda_1, r_2}\right)}H_1(z, |D u|)^\theta dz+c\iint_{ \Phi \left(\Lambda_1, r_2\right)} H_1(z, |F|)\,dz .
\end{align}
Let $\varepsilon \in (0, 1)$, which is to be determined later. Multiplying the above inequality \eqref{EQQ6.38} with $\Lambda_1^{\varepsilon-1}$ and integrating over $(\Lambda_2, \infty)$, we get
\begin{align}\label{EQQ6.39}
&\int_{\Lambda_2}^{\infty}\Lambda_1^{\varepsilon-1}\iint_{E^k\left({\Lambda_1, r_1}\right)}H_1(z, |D u|)_k^{1-\theta}H_1(z, |D u|)^\theta dz\, d\Lambda_1\nonumber\\
&\quad \leq c\int_{\Lambda_2}^{\infty} \Lambda_1^{\varepsilon-\theta}\iint_{E^k\left({\Lambda_1, r_2}\right)}H_1(z, |D u|)^\theta dz\, d\Lambda_1+ c\int_{\Lambda_2}^{\infty}\Lambda_1^{\varepsilon-1}\iint_{\Phi(\Lambda_1, r_2)}H_1(z, |F|)\,dz \, d\Lambda_1.
\end{align}
Now following \cite{2023_Gradient_Higher_Integrability_for_Degenerate_Parabolic_Double-Phase_Systems} and using Fubini's theorem, we obtain from \eqref{EQQ6.39} that
\begin{align*}
    &\iint_{Q_{r_1}(z_0)} H_1(z, |D u|)_k^{1-\theta+\varepsilon}H_1(z, |D u|)^\theta dz\\ &\leq \Lambda^{\varepsilon}_2\iint_{Q_{2r}(z_0)}H_1(z, |D u|)_k^{1-\theta}H_1(z, |D u|)^\theta dz+ \frac{c\varepsilon}{1+\varepsilon-\theta} \iint_{Q_{r_2}(z_0)}H_1(z, |D u|)_k^{1-\theta+\varepsilon}H_1(z, |D u|)^\theta dz\\
    &\qquad +c \iint_{Q_{2r}(z_0)}H_1(z, |F|)^{1+\varepsilon}\, dz.
\end{align*}
We now choose a small $\varepsilon < \varepsilon_0 < 1$ such that $\frac{c\varepsilon}{1+\varepsilon-\theta} \leq \frac{1}{2}.$ To use \cref{iter_lemma}, we define
\begin{align*}
    h(r)=\iint_{Q_{r}(z_0)} H_1(z, |D u|)_k^{1-\theta+\varepsilon} H_1(z, |D u|)^{\theta} \, dz.
\end{align*}
Considering the choice of $\Lambda_2,$ we have
\begin{align*} 
&\Lambda^{\varepsilon}_2\iint_{Q_{2r}(z_0)}H_1(z, |D u|)_k^{1-\theta}H_1(z, |D u|)^{\theta} \, dz\\
&=\frac{1}{(r_2-r_1)^{\frac{\varepsilon q(n+2)}{2}}}\left[\frac{(4c_v r)^{\frac{\varepsilon q(n+2)}{2}}}{c^{\varepsilon}} \Lambda^{\varepsilon}_0\iint_{Q_{2r}(z_0)}H_1(z, |D u|)_k^{1-\theta}H_1(z, |D u|)^{\theta} \, dz\right]\\
&=:\frac{A}{(r_2-r_1)^{\frac{\varepsilon s(n+2)}{2}}}.
\end{align*}
Now applying \cref{iter_lemma} with $\varrho_1=r_1, \varrho_2=r_2, R=2r, \vartheta=\frac{c\varepsilon}{1+\varepsilon-\theta}, B=0$ and $\gamma=\frac{\varepsilon q(n+2)}{2},$ we get
\begin{align*}
&\iint_{Q_{r}(z_0)} H_1(z, |D u|)_k^{1-\theta+\varepsilon}H_1(z, |D u|)^\theta dz \\
&\apprle\frac{1}{r^{\frac{\varepsilon q(n+2)}{2}}}\frac{(4c_v r)^{\frac{\varepsilon q(n+2)}{2}}}{c^{\varepsilon}} \Lambda^{\varepsilon}_0 \iint_{Q_{2r}(z_0)}H_1(z, |D u|)_k^{1-\theta}H_1(z, |D u|)^\theta dz+ c \iint_{Q_{2r}(z_0)}H_1(z, |F|)^{1+\varepsilon}\,dz.
\end{align*}
Letting $k \to \infty,$ we finally obtain
\begin{align*}
  \iint_{Q_{r}(z_0)} H_1(z, |D u|)^{1+\varepsilon}\, dz \leq c \Lambda^{\varepsilon}_0 \iint_{Q_{2r}(z_0)}H_1(z, |D u|)\,dz+c \iint_{Q_{2r}(z_0)}H_1(z, |F|)^{1+\varepsilon}\,dz.
\end{align*}
Recalling \eqref{eq: definition of lambda_0}--\eqref{eq: definition of Lambda_0} and using $H_1(z, |Du|)\geq 1$ with $q\geq 2,$ we derive
\begin{align*}
    &\miint{Q_{r}(z_0)}H_1(z, |Du|)^{1+\varepsilon}\, dz\\
    &\leq c \left(1+||a||_{L^{\infty}(\Omega_T)}+||b||_{L^{\infty}(\Omega_T)}\right)\lambda^{\varepsilon q}_0\miint{Q_{2r}(z_0)}H_1(z, |Du|)\, dz+ c\miint{Q_{2r}(z_0)}H_1(z, |F|)^{1+\varepsilon}\, dz\\
    &\leq c \left(\miint{Q_{2r}(z_0)}H_1(z, |D u|)\,dz\right)^{1+\frac{\varepsilon q}{2}}+c\left(\miint{Q_{2r}(z_0)}H_1(z, |F|)\, dz\right)^{1+\frac{\varepsilon q}{2}}\\
    &\qquad + c\miint{Q_{2r}(z_0)}H_1(z, |F|)^{1+\varepsilon}\, dz\\
    &\leq c \left(\miint{Q_{2r}(z_0)}H_1(z, |D u|)\,dz\right)^{1+\frac{\varepsilon q}{2}}+c\left(\miint{Q_{2r}(z_0)}H_1(z, |F|)^{1+\varepsilon}\, dz\right)^{\frac{1}{1+\varepsilon}+\frac{\varepsilon}{1+\varepsilon}\cdot \frac{q}{2}}\\
    &\qquad + c\miint{Q_{2r}(z_0)}H_1(z, |F|)^{1+\varepsilon}\, dz\\
    &\leq c \left(\miint{Q_{2r}(z_0)}H_1(z, |D u|)\,dz\right)^{1+\frac{\varepsilon q}{2}}+c\left(\miint{Q_{2r}(z_0)}H_1(z, |F|)^{1+\varepsilon}\, dz\right)^{\frac{q}{2}}, 
\end{align*}
which completes the proof.
\end{proof}

\subsection*{Conflict of interest statement} The authors declare that they have no conflicts of interest.

\subsection*{Data availability statement} This manuscript has no associated data.

\subsection*{Acknowledgment} The authors thank the referee for several helpful comments that improved the manuscript.


\bibliographystyle{plain}
\bibliography{ref}
\end{document}